\documentclass[12pt]{amsart}
\usepackage{amssymb}
\usepackage{pstricks,pstricks-add}

\def\mywidth{30em}

\newtheorem{theorem}{Theorem}[section]
\newtheorem{lemma}[theorem]{Lemma}
\newtheorem{proposition}[theorem]{Proposition}
\newtheorem{corollary}[theorem]{Corollary}
\newtheorem{definition}[theorem]{Definition}

\theoremstyle{definition}
\newtheorem{remark}[theorem]{Remark}
\newtheorem{example}[theorem]{Example}

\numberwithin{equation}{section}

\newcommand{\Z}{\mathbb{Z}}
\newcommand{\N}{\mathbb{N}}
\newcommand{\R}{\mathbb{R}}
\newcommand{\T}{\mathbb{T}}
\newcommand{\C}{\mathbb{C}}
\newcommand{\K}{\mathbb{K}}
\newcommand{\CP}{\mathbb{CP}}
\newcommand{\PP}{\mathbb{P}}
\newcommand{\bfL}{\mathbf{L}}
\renewcommand{\Re}{\mathrm{Re}\,}

\newcommand{\Hom}{\mathrm{Hom}}

\renewcommand{\O}{\mathcal{O}}
\newcommand{\Wrap}{\mathcal{W}}
\newcommand{\F}{\mathcal{F}}
\newcommand{\cL}{\mathcal{L}}

\renewcommand{\v}{\mathbf{v}}
\newcommand{\bb}{\mathfrak{b}}
\newcommand{\cupprod}{\mathbin{\smile}}

\renewcommand{\mod}{\mathrm{mod}}
\newcommand{\cY}{\mathcal{Y}}
\newcommand{\id}{\mathrm{id}}
\DeclareMathOperator*{\colim}{colim}
\DeclareMathOperator*{\hocolim}{hocolim}
\DeclareMathOperator*{\holim}{holim}

\def\co{\colon\thinspace}

\title[Homological mirror symmetry for hypersurfaces in $(\C^*)^n$]{Homological mirror symmetry for \\ hypersurfaces in $(\C^*)^n$}
\author{Mohammed Abouzaid}
\address{Department of Mathematics, Columbia University, 2990 Broadway, New York, NY 10027, USA}
\email{abouzaid@math.columbia.edu}
\author{Denis Auroux}
\address{Department of Mathematics, Harvard University, 1 Oxford St, Cambridge, MA 02138, USA}
\email{auroux@math.harvard.edu}
\thanks{The first author was partially supported by NSF grants DMS-1308179,
DMS-1609148, DMS-1564172, and DMS-2103805, a Simons Fellowship, and a Poincar\'e visiting professorship at Stanford University. 
The second author was partially supported by NSF grants DMS-1406274,
DMS-1702049, DMS-1937869, and DMS-2202984. Both authors were supported by the Simons Foundation
collaboration grant ``Homological Mirror Symmetry and Applications'' (award
numbers 385571 and 385573).}

\begin{document}
\begin{abstract} 
We prove a homological mirror symmetry result for maximally degenerating families 
of hypersurfaces in $(\C^*)^n$ 
(B-model) and their mirror toric Landau-Ginzburg A-models. 
The main technical ingredient of our construction is a ``fiberwise wrapped'' version 
of the Fukaya category of a toric Landau-Ginzburg model. With the definition
in hand, we construct a fibered admissible Lagrangian submanifold whose fiberwise wrapped
Floer cohomology is isomorphic to the ring of regular functions of the hypersurface.
It follows that the derived category of coherent sheaves of the hypersurface
quasi-embeds into the fiberwise wrapped Fukaya category of the mirror.
We also discuss an extension to complete intersections. 
\end{abstract}

\setcounter{tocdepth}{1}
\maketitle
\tableofcontents
\setcounter{tocdepth}{2}

\section{Introduction}

The range of settings in which mirror symmetry is expected to hold has
steadily expanded since the mirror conjectures were first formulated for
projective Calabi-Yau varieties, and there are now candidate mirror
constructions in a wide range of settings. Outside of the Calabi-Yau
setting, the mirrors are in general Landau-Ginzburg models, i.e.\ pairs 
$(Y,W)$ where $Y$ is a quasi-projective variety and $W\in \O(Y)$ is a
regular function (the {\em superpotential}). 

We focus on the case of hypersurfaces in $(\C^*)^n$ (or rather hypersurfaces defined over the non-archimedean Novikov field 
$\K=\Lambda$, which arise from maximally
degenerating families of hypersurfaces near the tropical limit).
These have mirror Landau-Ginzburg models which consist of a noncompact
toric Calabi-Yau variety $Y$ of dimension $n+1$, equipped with a superpotential $W$ 
which is a toric monomial vanishing to order 1 on each irreducible
toric divisor of $Y$. The construction is summarized in
Section \ref{s:LGmodel}, following the description given in
\cite{AAK} which arrives at these mirrors from the perspective of SYZ
mirror symmetry; see also \cite{HV,Clarke,ChanLauLeung,GKR} for other viewpoints.

To be specific, consider a degenerating family of complex hypersurfaces defined by a Laurent polynomial of the form 
\begin{equation}\label{eq:intropoly}
f=\sum_{\alpha\in P_\Z} a_\alpha t^{2\pi\nu(\alpha)} x^\alpha,
\end{equation}
where $P_\Z$ is a finite subset of $\Z^n$, the exponents
$\nu(\alpha)\in \R$ are assumed to satisfy a convexity condition which ensures that Equation \eqref{eq:intropoly} is a sufficiently generic degeneration, and the coefficients $a_\alpha$ are complex numbers in the simplest situations, but will in general be given by elements of $\Lambda$ of vanishing valuation (see Section
\ref{s:LGmodel}). 
The space $Y$ which we associate to these data  is the K\"ahler toric variety determined by the polytope 
\begin{equation} \label{eq:mirror_polytope_Fan}
    \Delta_Y=\{(\xi,\eta)|\,\eta\geq \varphi(\xi)\} \subset \R^n\oplus \R,
\end{equation}
where the piecewise linear function $\varphi:\R^n\to \R$ is the
tropicalization of $f$, $$\varphi(\xi)=\max\limits_{\alpha\in P_\Z} \,\{\langle
\alpha,\xi\rangle-\nu(\alpha)\},$$
and the superpotential $W=-z^{(0,\dots,0,1)}$ is (up to sign) the toric monomial associated to the last coordinate of the ambient space $\R^n\oplus \R $ in Equation \eqref{eq:mirror_polytope_Fan}. 
The regular fibers of $W:Y\to\C$ are isomorphic to $(\C^*)^n$, 
while the unique singular fiber $Z=W^{-1}(0)=\bigcup_\alpha Z_\alpha$ is a union of toric varieties (the 
irreducible toric divisors of $Y$, which are in one-to-one correspondence
with the monomials appearing in $f$).

In the simplest example, the hypersurface $H$ is
the higher-dimensional pair of pants
$$
    \{(x_1,\dots,x_n)\,|\,1+x_1+\dots+x_n=0\}\subset (\C^*)^n
$$
with mirror the
Landau-Ginzburg model $(Y=\C^{n+1},W=-z_1\dots z_{n+1})$, whose singular
fiber is the union of the coordinate hyperplanes in $\C^{n+1}$;
however in most cases $Y$ is not affine and depends
on the choice of degeneration. 
\medskip

In one direction, homological mirror symmetry predicts that the wrapped Fukaya
category of a hypersurface $H\subset (\C^*)^n$ is equivalent to the
derived category of singularities of the mirror Landau-Ginzburg model,
$D^b_{sg}(Y,W)=D^b\mathrm{Coh}(Z)/\mathrm{Perf}(Z)$. This was first
verified for the wrapped Fukaya categories of open Riemann surfaces in
$(\C^*)^2$ and the derived categories of singularities of their mirror
Landau-Ginzburg models \cite{AAEKO, Lee}; see also \cite{LP1}, where the algebraic 
side is rather the derived category of coherent sheaves of a stacky nodal 
curve (equivalent to the Landau-Ginzburg model $(Y,W)$ via Orlov's
derived Kn\"orrer periodicity).  In higher dimensions, the result was
first verified for higher-dimensional pairs of pants in \cite{GammageNadler}
and \cite{LP2}; in the first of these, the wrapped Fukaya category is replaced
by the category of wrapped microlocal sheaves, but the two were subsequently
shown to be equivalent by Ganatra-Pardon-Shende \cite{GPS3}). Finally, the
case of general hypersurfaces in $(\C^*)^n$ was established by Gammage and Shende
\cite{GammageShende}, also using wrapped microlocal sheaves.

Here we consider the other direction of mirror symmetry, comparing 
coherent sheaves on the family of hypersurfaces $H_t$ defined by $f$ to a suitable
version of the Fukaya category of the Landau-Ginzburg model $(Y,W)$,
where $Y$ is equipped with a suitable toric K\"ahler form in the class
$[\omega_Y]\in H^2(Y,\R)$ determined by the polytope $\Delta_Y$, and
also a bulk deformation class (or B-field) $\bb\in H^2(Y,\Lambda_{\geq 0})$ (the subscript $\geq 0$ indicates that we only consider elements of non-negative valuation). 
This direction has been much less studied; in fact, at the start of our
project there wasn't even yet a candidate definition for the appropriate Fukaya 
category, because the initial formulation required that $Y$ be affine and that $W$ have isolated non-degenerate singularities~\cite{SeBook}.

\subsection{Fiberwise wrapped Fukaya categories}

The first step in our approach is to
define the {\em fiberwise wrapped Fukaya category} $\Wrap(Y,W)$
of a toric Landau-Ginzburg model. The objects of $\Wrap(Y,W)$ are
properly embedded Lagrangian submanifolds $L\subset Y$ which satisfy two
different types of geometric requirements: (1) in the base direction, 
we require that $L$ is fibered at infinity, i.e.\ that outside of a compact
subset of $\C$ the image of $L$ under $W:Y\to\C$ is a union of embedded arcs,
which are further required to be disjoint from the negative real axis and along which the distance from the origin increases strictly; (2) we
require $L$ to be fiberwise ``flat'' at infinity with respect to a
weakly plurisubharmonic fiberwise ``height'' function $h:Y\to [0,\infty)$,
i.e.\ the restriction of $d^c h$ to $L$ is required to vanish outside of
a bounded subset of each fiber of $W$. We call such Lagrangians {\em admissible}; 
see Definition \ref{def:admissibleLagrangian}. The Lagrangians we consider
are also required to be tautologically unobstructed (in the sense of not bounding any holomorphic disc with respect to a prescribed almost complex structure), and are equipped with the grading data and
local systems needed to construct Floer complexes.

Morphism spaces in $\Wrap(Y,W)$ are defined as direct limits of Floer
complexes for the images of admissible Lagrangians under a suitable
geometric flow,
which combines (1)~in the base direction, 
admissible isotopies acting on the complex plane by positive rotations
without crossing the 
negative real axis (as in the more familiar setting of Fukaya categories
of Lefschetz fibrations), and (2) in the fiber direction,
the flow of a Hamiltonian $H:Y\to\R$ which 
preserves the fibers of $W$ and whose restriction to each fiber is a
linear-growth wrapping Hamiltonian (hence the name ``fiberwise wrapped'').
The details of the construction are given in Section \ref{s:Fukayacat}.

In the toric case, the fiberwise 
behavior of our admissible Lagrangians is enforced by fixing a
collection of monomials $z^\v\in \O(Y)$ and open subsets $C_\v$ of $Y$,
and requiring $\arg(z^\v)$ to be locally constant over $L\cap C_\v$.
This amounts to a fiberwise version of the notion of {\em monomial admissibility}
considered in Andrew Hanlon's thesis \cite{Hanlon}; in fact, even though
we treat the monomial $W$ separately, the condition we impose in the base
direction could also be reformulated in the language of monomial admissibility.

Since our Lagrangians are required to be both fibered with respect to
$W:Y\to\C$ and fiberwise monomially admissible within the fibers, our
setup requires symplectic parallel transport between smooth 
fibers of $W$ to be compatible with monomial admissibility. 
This compatibility is easy to achieve for parallel transport along
radial lines in the complex plane by using elementary toric geometry (or
by directly imposing monomial admissibility in the total space $Y$).
However, the explicit calculation of Floer complexes and differentials
at the heart of our verification of homological mirror
symmetry requires us to consider Lagrangians that are everywhere fibered over
U-shaped arcs in the complex plane. Achieving fiberwise monomial
admissibility for such Lagrangians requires
some extra care in the choice of the toric K\"ahler form $\omega_Y$ on $Y$
within the given cohomology class; see Section \ref{s:toric} for details.

\begin{remark}
The several years elapsed since our results were first announced have
brought forth key advances and new viewpoints on Fukaya categories of
Landau-Ginzburg models which suggest other possible approaches. 

For example, partially wrapped Floer theory for
Liouville domains with stops \cite{Sylvan} and sectors
\cite{GPS1} has led to considerable progress in the exact setting.
However it is not clear that viewing $(Y,W)$ as a non-exact sector
would yield any simplification to our setup and main calculation, as the alternative 
description in terms of wrapped microlocal sheaves used by Nadler in the
case of higher-dimensional pairs of pants \cite{Nadler} would not be 
applicable outside of the exact setting, and direct calculation by
counting holomorphic discs would likely be no easier than the approach taken here. 

Monomial admissibility, as used by Hanlon to revisit mirror symmetry for
toric varieties \cite{Hanlon}, is much more directly suited to our goals,
and in fact we use this viewpoint to constrain the fiberwise behavior
of our Lagrangians and to arrive at a maximum principle.
 Defining $\Wrap(Y,W)$ directly in 
the language of monomial admissibility (adding $W$ itself to the list of 
monomials $z^\v$ whose arguments we constrain at infinity) would be fairly
straightforward, but the explicit calculation of Floer cohomology would 
likely still require the Lagrangian to be everywhere fibered with respect to the 
projection $W:Y\to \C$ (not just near infinity), making the setup essentially identical to that
considered here.

One can alternatively attempt to replace monomial admissibility with a variant of Groman's formulation for Floer theory on open manifolds \cite{Gr}, adapted to the setting of Landau-Ginzburg models. Early drafts of this text pursued a related approach based on geometric estimates on parallel
transport and monotonicity type arguments, but the relevant estimates turned out to be quite
challenging.
\end{remark}

\subsection{A Floer cohomology calculation}

The main protagonist of our argument is a specific admissible Lagrangian $L_0$
in the toric Landau-Ginzburg model $(Y,W)$, which is expected to generate
the fiberwise wrapped Fukaya category. 

Consider a Laurent polynomial $f\in
\K[x_1^{\pm 1},\dots,x_n^{\pm 1}]$ defining a maximally degenerating family
of hypersurfaces $H_t$ as above, and let $(Y,W)$ be the toric
Landau-Ginzburg model constructed in Section \ref{s:LGmodel}, equipped with
the toric K\"ahler form $\omega_Y$ constructed in Section \ref{s:toric} and
a bulk deformation\footnote{In the literature, one usually considers bulk classes of strictly positive valuation; the $0$-valuation part of $\bb$ corresponds to (a logarithm) of what is sometimes called a background class, which in our case is valued in $\C^*$, but is usually considered with $\Z_2$ coefficients, and modifies Floer theory by changing the sign contributions of discs.}  $\bb\in H^2(Y,\Lambda_{\geq 0})$.  Since the fiber $W^{-1}(-1)\subset Y$ is isomorphic to $(\C^*)^n$, it contains
a distinguished Lagrangian $\ell_0=(\R_+)^n$ along which the toric monomials
$z^\v$ are all real positive. The parallel transport of $\ell_0$ over a
U-shaped arc in the complex plane connecting $-1$ to $+\infty$ around the origin yields 
an admissible Lagrangian submanifold $L_0$ in $(Y,W)$. Our main result is:

\begin{theorem}\label{thm:main}
For a suitable choice of bulk deformation class $\bb\in H^2(Y,\Lambda_{\geq
0})$, the fiberwise wrapped Floer cohomology ring $H\Wrap^*(L_0,L_0)$ is
isomorphic to the quotient $\K[x_1^{\pm 1},\dots,x_n^{\pm 1}]/(f)$ of the
ring of Laurent polynomials by the ideal generated by $f$, the defining
equation of the family of hypersurfaces $H_t$.
\end{theorem}
\begin{remark}
We refer the reader to Remark \ref{rem:mirror-map} for a discussion of the relationship between the bulk class appearing in the statement of the above theorem and the expression of the mirror map in terms of Gromov-Witten theory.
\end{remark}
In other terms, $H\Wrap^*(L_0,L_0)$ is isomorphic to the ring of functions
of the non-archimedean hypersurface $\mathcal{H}$ defined by $f$ over $\K$:
\begin{equation}\label{eq:ringiso}
H\Wrap^*(L_0,L_0)\simeq \K[x_1^{\pm 1},\dots,x_n^{\pm 1}]/(f)\simeq
H^0(\mathcal{H},\O_{\mathcal{H}})=\Hom(\O_{\mathcal{H}},\O_{\mathcal{H}}).
\end{equation}
Since this ring is supported in degree $0$, it is intrinsically formal, so we conclude that the Floer algebra $ \Wrap^*(L_0,L_0)$ is $A_\infty$ equivalent to the ring of functions on $\mathcal{H}$. On the other hand,  
since $\mathcal{H}$ is affine, its derived category is generated by the
structure sheaf $\O_{\mathcal{H}}$, and by mapping a twisted complex built
from copies of $\O_{\mathcal{H}}$ to the corresponding twisted complex
built from $L_0$ inside $\Wrap(Y,W)$, we arrive at:
\begin{corollary}\label{cor:main}
The derived category of coherent sheaves of $\mathcal{H}$ admits a
fully faithful quasi-embedding into $\Wrap(Y,W)$.
\end{corollary}

One can then return from the non-archimedean setting to the complex 
hypersurfaces $H_t$ by observing that, when $f$ is of the form
\eqref{eq:intropoly} with $a_\alpha\in \C^*$, the outcome of our
calculation is manifestly convergent over complex numbers and we can
treat $t$ as an actual parameter rather than a formal variable.

The calculation of $H\Wrap^*(L_0,L_0)$ involves counts of holomorphic
sections of the fibration $W:Y\to\C$ over domains in the complex plane,
with boundary in fibered Lagrangians, and the argument is essentially
within the realm of the ``Seidel TQFT'' \cite{SeBook} even though $W$ is
not a Lefschetz fibration; see Section \ref{s:calculation}. 
Our approach is concrete and explicit, but
a more conceptual interpretation can be given in terms of the Orlov cup
functor; see below.

\begin{remark}
The object $L_0$ is expected to generate the category $\Wrap(Y,W)$,
which would imply that the embedding of Corollary \ref{cor:main} is an
equivalence. Stop removal (wrapping past the negative real axis in the
base direction) yields an {\em acceleration} functor from $\Wrap(Y,W)$
to a suitable version of the wrapped Fukaya category of $Y$, under
which $L_0$ maps to the zero object (cf.\ \cite{A-S}). The stop removal results 
of \cite{Sylvan,GPS1} (to the extent that they hold in our setup) should 
imply that $\Wrap(Y)$ is precisely the quotient of $\Wrap(Y,W)$ by the 
full subcategory generated by $L_0$. The generation statement is then
equivalent to the vanishing of $\Wrap(Y)$. This argument can be made 
precise in the case of the pair of pants, where $Y=\C^{n+1}$ is a
subcritical Liouville manifold and its wrapped Fukaya category vanishes.
However, given that a complete argument in the general case where $Y$ 
is not exact would involve several pieces of machinery that have not yet
been developed in that setting, we do not investigate this question further
in this paper.
\end{remark}
\subsection{A functorial perspective}\label{ss:functorial}
The fiberwise wrapped Fukaya category is the target of a functor
$$\cup:\Wrap((\C^*)^n)\to \Wrap(Y,W)$$
(the Orlov cup functor), which is given on objects by parallel transport of
admissible Lagrangian submanifolds of $W^{-1}(-1)\simeq (\C^*)^n$ 
along a U-shaped arc in the complex plane, and on morphisms by observing that
the portions of the fiberwise wrapped
Floer complexes which live in the fiber over 
$-1$ are closed under all $A_\infty$-operations. In this language, the computation at the heart of
the proof of Theorem \ref{thm:main} gives a commutative diagram
of functors

\begin{equation}\label{eq:cupdiagram}
\begin{psmatrix}[colsep=1.5cm,rowsep=1.5cm]
\mathrm{Perf}((\K^*)^n) & \mathrm{Perf}(\mathcal{H})\\
\Wrap((\C^*)^n) & \Wrap(Y,W)
\end{psmatrix}
\psset{nodesep=5pt,arrows=->,hooklength=1mm,hookwidth=1mm}
\ncline{1,1}{2,1} \tlput{\simeq}
\ncline[arrows=H->]{1,2}{2,2} 
\ncline{1,1}{1,2} \taput{i^*}
\ncline{2,1}{2,2} \tbput{\cup}
\end{equation}
\vskip3mm

\noindent
where the restriction functor $i^*$ and the cup functor $\cup$ intertwine mirror
symmetry for the ambient torus $(\K^*)^n$ and the hypersurface $\mathcal{H}$.
The core of our argument amounts to a verification of this statement for
the structure sheaves on the algebraic side, and the admissible Lagrangians
$\ell_0=(\R_+)^n$ and $L_0={\cup}\ell_0$ on the symplectic side.

To continue further in this direction, the functor $\cup$ has an
adjoint functor $\cap:\Wrap(Y,W)\to\mathrm{Perf}\,\Wrap((\C^*)^n)$
(``restriction to the fiber at $+\infty$''), under which a fibered 
Lagrangian $L=\cup\ell$ maps to a twisted complex built from the 
fiberwise Lagrangians at the two ends of the U-shaped arc, with a
connecting differential $s^0_\ell$ which counts holomorphic sections of $W:Y\to\C$
bounded by $L$ over the region enclosed by the U-shaped arc. After
choosing a suitable identification of the fiber near $+\infty$ with $(\C^*)^n$,
we find that the image of $\ell$ under the composite functor $\cap \cup$ is isomorphic to a cone
$$\cap\cup\ell\simeq
\Bigl\{\,\mu^{-1}(\ell)\stackrel{s^0_\ell}{\longrightarrow} \ell\,\Bigr\},$$
where $\mu^{-1}$ is the clockwise monodromy of the fibration $W$ around the
origin, acting on the wrapped Fukaya category of the fiber (in our case
$\mu^{-1}\simeq \mathrm{id}$), and $s^0_\ell$ is a count of sections. This is part of
an exact triangle of functors 

\begin{equation*}
\begin{psmatrix}[colsep=1.2cm,rowsep=1cm]
\mu^{-1} & & \mathrm{id}\\
&\cap\cup\vphantom{x_j}
\end{psmatrix}
\psset{nodesep=5pt,arrows=->}
\ncline{2,2}{1,1} \tlput{[1]}
\ncline{1,1}{1,3} \taput{s}
\ncline{1,3}{2,2}
\end{equation*}
acting on
$\Wrap((\C^*)^n)$, originating in Seidel's work \cite{SeDehn} on Dehn twists, and which has been the subject of some recent work
(cf.\ \cite{A-G}, \cite[Appendix A]{A-Sm}, and
\cite[Theorem 1.3]{Sylvan2}).

Our calculation of the fiberwise wrapped Floer complex of
$L_0=\cup\ell_0$ can then be rewritten as
$$\Wrap_{(Y,W)}(\cup \ell_0,\cup \ell_0)\simeq 
\Wrap_{(\C^*)^n}(\ell_0,\cap\cup \ell_0)\simeq 
\mathrm{Cone}\Bigl\{HW^*(\ell_0,\mu^{-1}(\ell_0))
\stackrel{s^0_{\ell_0}}{\longrightarrow} HW^*(\ell_0,\ell_0)\Bigr\}$$
which, after verifying that the section-counting natural transformation $s^0_{\ell_0}$
amounts to multiplication by the Laurent polynomial $f$, corresponds on the algebraic side to
$$\Hom_{\mathcal{H}}(\O_\mathcal{H},\O_\mathcal{H})\simeq
\Hom_{(\K^*)^n}(\O,i_*i^*\O)\simeq \mathrm{Cone}
\Bigl\{\Hom(\O,\O)
\stackrel{f}{\longrightarrow} \Hom(\O,\O)\Bigr\}.$$

\subsection{Complete intersections and compactifications}

Our results admit extensions in at least two directions. The first one,
which we briefly discuss in Section \ref{s:ci}, concerns complete
intersections. The mirror of a codimension $k$ complete intersection
in $(\C^*)^n$ (or rather, of a family of complete intersections degenerating
to a tropical limit) is a Calabi-Yau toric K\"ahler manifold $Y$ of complex
dimension $n+k$, equipped with a superpotential which is a sum of $k$ toric
monomials $W_1,\dots,W_k \in \O(Y)$; taken together these determine a holomorphic map
$\mathbf{W}:Y\to \C^k$, whose fibers over $(\C^*)^k$ are again isomorphic
to $(\C^*)^n$ \cite[Section 11]{AAK}. We then define a version of the
fiberwise wrapped Fukaya category $\Wrap(Y,\mathbf{W})$ whose objects are
Lagrangian submanifolds which are
simultaneously admissible for each of the projections $W_1,\dots,W_k$;
the morphisms are direct limits of Floer complexes under a combination of 
admissible isotopies acting on each factor of $\C^k$ by positive
rotations without crossing the negative real axis and wrapping in the
fibers of $\mathbf{W}$. The key object $L_0$ under consideration is now obtained
by parallel transport of $(\R_+)^n\subset (\C^*)^n$ over a product
of U-shaped arcs inside $\C^k$. By an argument similar to our main
calculation, its fiberwise wrapped Floer complex is isomorphic to the 
Koszul resolution of the ring of functions of the complete intersection;
the upshot is that the obvious analogues
of Theorem \ref{thm:main} and Corollary \ref{cor:main} continue to hold in
this setting. See Section \ref{s:ci} and Theorem \ref{thm:ci_main}.

Another extension is to hypersurfaces (and complete intersections) in
toric varieties.  Namely, a Laurent polynomial of the form
\eqref{eq:intropoly} defines not only hypersurfaces in $(\C^*)^n$ or 
$(\K^*)^n$ but also (partial) compactifications in suitable toric varieties
or stacks~-- for example, the projective toric variety or stack $\overline{V}$ whose moment 
polytope is the convex hull of $P_\Z$. Following \cite{AAK}, the mirror to
$\overline{H}\subset \overline{V}$ is the same Calabi-Yau toric variety $Y$ as in our
main construction, now equipped with a superpotential $\overline{W}$ which is the sum of 
the previously encountered monomial $W_0=-z^{(0,\dots,0,1)}$ and extra 
terms consisting of one monomial for each of the irreducible toric divisors 
of $\overline{V}$. The latter turn out to be exactly the collection of 
monomials $z^\v$ we consider in Definition \ref{def:extremalv}.
Consequently, we can define the Fukaya category $\F(Y,\overline{W})$ by
considering exactly the same admissible Lagrangian submanifolds
of $Y$ as in the construction of $\Wrap(Y,W_0)$: namely, Lagrangians 
which are fibered at infinity with respect to $W_0:Y\to\C$, and within 
the fibers of $W_0$, monomially admissible for the collection
of monomials $z^\v$. However, we now consider colimits under
perturbations which only increase the argument of each monomial $z^\v$
within a small bounded interval, rather than by an unbounded amount of 
fiberwise wrapping. Starting from monomially admissible Lagrangian sections
$\ell,\ell'\subset (\C^*)^n$ such as those considered in \cite{Hanlon},
which are mirrors to line bundles $\mathcal{L},\mathcal{L}'$ on the toric variety $\overline{V}$, we
now find an isomorphism 
$$\Hom_{\F(Y,\overline{W})}(\cup\ell,\cup\ell')\simeq \mathrm{Cone}\Bigl\{
\Hom_{\F((\C^*)^n,\{z^\v\})}(\ell,\mu^{-1}(\ell'))
\stackrel{s}{\longrightarrow} \Hom_{\F((\C^*)^n,\{z^\v\})}(\ell,\ell')
\Bigr\}.$$
After checking that the action of the monodromy $\mu^{-1}$ on monomially admissible
Lagrangian sections is mirror to the functor $-\otimes \O(-\overline{H})$
and that the natural transformation $s:\mu^{-1}\to \mathrm{id}$ still
corresponds to multiplication by the defining section $f$ of $\overline{H}$,
this corresponds on the algebraic side to the isomorphism
$$\Hom_{\overline{H}}(\mathcal{L}_{|\overline{H}},
\mathcal{L}'_{|\overline{H}})\simeq \mathrm{Cone}\Bigl\{
\Hom_{\overline{V}}(\mathcal{L},\mathcal{L'}\otimes \O(-\overline{H}))
\stackrel{f}{\longrightarrow} \Hom_{\overline{V}}(\mathcal{L},\mathcal{L'})
\Bigr\}.$$
This in turn implies cohomology-level mirror symmetry statements for
restrictions of ample line bundles (compare with \cite{Cannizzo} which
establishes analogous results in a different setting).
A more detailed account of this will appear elsewhere \cite{AA2}.

\subsection{Related works}
In the time elapsed since our results were first announced, at least two
papers have appeared establishing conceptually similar homological
mirror symmetry results relating coherent sheaves on hypersurfaces to
the symplectic geometry of mirror Landau-Ginzburg models.

On one hand, Nadler introduced the category of wrapped microlocal
sheaves and gave an explicit computation  for the Landau-Ginzburg model
$(\C^{n},z_1\dots z_{n})$, which is mirror to the $(n-2)$-dimensional
pair of pants \cite{Nadler}. (Wrapped microlocal sheaves were subsequently
shown by \cite{GPS3} to be equivalent to the Fukaya category of the
corresponding Liouville sector.) Nadler's paper showcases the remarkable
computational power of microlocal sheaves in the exact setting, and
also identifies structural properties which are closely related to those
described in \S\ref{ss:functorial} above.

On the other hand, Cannizzo's thesis work \cite{Cannizzo} considers the
case of a genus 2 curve embedded in a principally polarized abelian surface
(its Jacobian) and the mirror Landau-Ginzburg model. The approach is fairly
similar to ours, but avoids the need to discuss fiberwise admissibility
because the mirror is proper (the generic fibers are $T^4$). However,
the monodromy is topologically non-trivial, and involves a twist 
mirror to the defining section of the genus 2 curve, so that the objects
of interest are a sequence of admissible Lagrangians mirror to powers of
the canonical bundle of the genus 2 curve (somewhat similarly
to the toric variety case outlined above). Another notable difference with
our setting is that, despite the non-exact nature of the mirror and the
presence of rational curves in the zero fiber, 
no bulk deformation is required as the instanton corrections only result
in a rescaling of the section-counting natural transformation \cite{Cannizzo}.

\subsection*{Plan of the paper}
The first part of this paper is concerned with the definition of the
fiberwise wrapped Fukaya category $\Wrap(Y,W)$. After reviewing
the construction of the Landau-Ginzburg model $(Y,W)$ in Section
\ref{s:LGmodel}, we develop the foundations of fiberwise wrapped Fukaya
categories in Section \ref{s:Fukayacat}, while  Section \ref{s:toric} is
devoted to the construction of the appropriate
toric K\"ahler form and verification of the required geometric properties.
The heart of the paper is then Section \ref{s:calculation}, which is devoted to the calculation of the
fiberwise wrapped Floer cohomology of $L_0$ and the proof of Theorem
\ref{thm:main}. Finally, in Section \ref{s:ci} we briefly discuss the
extension to complete intersections and prove Theorem \ref{thm:ci_main}.

\subsection*{Acknowledgements}

The first author was partially supported by NSF grants DMS-1308179,
DMS-1609148, DMS-1564172, and DMS-2103805, and by Stanford University through a Poincar\'e visiting professorship.  
The second author was partially supported by NSF grants DMS-1406274,
DMS-1937869, and DMS-2202984. Both authors were supported by the Simons Foundation
collaboration grant ``Homological Mirror Symmetry and Applications'' (award
numbers 385571 and 385573).

We wish to acknowledge our former students 
Zack Sylvan, Andrew Hanlon and Catherine Cannizzo,
whose PhD thesis projects evolved alongside this work and have provided
inspiration, technical foundations, and validation for various aspects of 
the strategy employed here. We would also like to thank the homological mirror symmetry community for
its collective patience with the long delay between the announcement of our
results and the completion of this text. Finally, we thank the anonymous
referee for their careful comments on the previous version of this paper.

\section{The mirror Landau-Ginzburg model} \label{s:LGmodel}

\subsection{The main construction}

Consider a  Laurent  polynomial
\begin{equation}
  f = \sum_{\alpha \in \Z^{n}} a_{\alpha} x^{\alpha},
\end{equation}
with complex coefficients, and denote by 
\begin{equation}
  H = f^{-1}(0)  \subset (\C^*)^n
\end{equation}
the corresponding hypersurface.

The construction of a mirror for $H$ depends on a choice of degeneration;
we specifically consider a maximal degeneration to a tropical limit,
and assume that the corresponding tropical variety is smooth in the sense
we explain now.

Let $P$ denote the Newton polytope of $f$, and $P_{\Z}$ its integral points. 
For simplicity, we assume that the interior of $P$ is non-empty (i.e., $P$
has positive volume); otherwise
we can always reduce to this case by splitting off some $\C^*$ factors.

A {\em tropically smooth}\/ maximal degeneration of $H$ is induced by the
choice of a strictly convex piecewise linear function
\begin{equation}\label{eq:define_nu}
\nu \co P \to \R
\end{equation}
whose domains of linearity determine a subdivision $\mathcal{P}$ of $P$ into 
standard integral simplices, i.e.\ simplices that are equivalent by an
integral affine transformation to the simplex spanned by the origin and the unit coordinate vectors in $\mathbb{Z}^n$;
this condition ensures that the mirror toric variety we construct below is smooth. 
The corresponding degeneration is then
\begin{equation}
f_{\nu} = \sum_{\alpha \in P_{\Z}} a_{\alpha} t^{2\pi\nu(\alpha)} x^{\alpha}.  
\end{equation}
We can associate to $f_{\nu}$ either a family of hypersurfaces parametrised 
by $t \in \C$, or a variety $\mathcal{H}$ over the Novikov
field $\K=\Lambda$ of power series in the formal variable $t$ with real exponents. The second point of view is more natural for the purpose of proving the well-definedness and invariance of the Fukaya category, and providing clear formulations of homological mirror symmetry.

Denote by $\varphi\co \R^n\to \R$ the tropicalisation of $f_{\nu}$, i.e.\ the 
piecewise linear function
\begin{equation}\label{eq:tropf}
  \varphi(\xi)=\max\{\langle \alpha,\xi\rangle-\nu(\alpha)\,|\,\alpha\in P_\Z\}.
\end{equation}
Let $Y$ be the (noncompact) K\"ahler toric manifold defined by the moment polytope
\begin{equation}\label{eq:Delta_Y}
  \Delta_Y=\{(\xi,\eta)\in \R^n\oplus \R\,|\,\eta\ge \varphi(\xi)\}.
\end{equation}

The polytope $\Delta_Y$ determines a K\"ahler class 
$[\omega_Y]\in H^2(Y,\R)$. In \S \ref{s:toric}, we shall specify an explicit K\"ahler form $\omega_Y$, 
obtained by Hamiltonian reduction from a vector space, which will be particularly well-adapted to our Floer-theoretic constructions.

Dually, $Y$ can also be described by the fan 
\begin{equation}
  \Sigma_Y=\R_{\ge 0}\cdot(\mathcal{P}\times \{1\})\subseteq 
  \R^{n+1}=\R^n\oplus\R,
\end{equation}
whose rays are generated by the integer vectors $(-\alpha,1)$, $\alpha\in P_\Z$,
and which is obtained as the union of the cones on polyhedra appearing in the subdivision $\mathcal{P}$.
Since we have assumed that this subdivision is maximal, all such cones are simplicial, and
since the simplices are further assumed to be congruent to the standard one,
it follows that $Y$ is a smooth
toric manifold. It is in fact a smooth toric Calabi-Yau, since the defining
equation of its toric anticanonical divisor is a regular
function (see below);
in particular its canonical bundle is trivial,
i.e.\ $c_1(Y)=0$, which will allow us to introduce $\Z$-gradings in Floer theory
(and also simplify our discussion of sphere bubbling). 

Denote by $z^{(0,\dots,0,1)}\in \O(Y)$ the toric monomial with weight
$(0,\dots,0,1)$, and equip $Y$ with the superpotential
\begin{equation}
  W=-z^{(0,\dots,0,1)} \co Y\to\C.
\end{equation}
The toric Landau-Ginzburg model $(Y,W)$ has been constructed as a
candidate mirror to $H$ from various perspectives; see in particular
\cite[Theorem 1.4]{AAK}. 

The level set $W^{-1}(0)$ is the union of the toric divisors in $Y$ (each with
multiplicity one), while the other level sets of $W$ are smooth and
isomorphic to $(\C^*)^n$. 
(The fact that the toric anticanonical divisor is 
defined by a regular function, namely $W$, verifies the above claim that $Y$
is Calabi-Yau).

\begin{example}\label{ex:example_21}
As a running example to illustrate the construction, we consider the Laurent polynomial
$f(x_1,x_2)=1+x_1+x_2+t^{2\pi}x_1x_2+t^{4\pi}x_1^2$ (which defines a degenerating
family of genus 0 curves with 5 punctures in $(\C^*)^2$). The tropicalization
of $f$ is given by $\varphi(\xi_1,\xi_2)=\max\{0,\xi_1,\xi_2,\xi_1+\xi_2-1,2\xi_1-2\}$.
The domains of linearity of $\varphi$, which also correspond to the facets
of the polytope $\Delta_Y$ ``seen from above'', are depicted on Figure
\ref{fig:tropicalexample}, along with the fan $\Sigma_Y$, whose generators
$(-\alpha,1)$, $\alpha\in P_\Z$, give the primitive (inward) normal vectors to
the facets of $\Delta_Y$.
\end{example}

\begin{figure}[t]
\setlength{\unitlength}{1cm}
\begin{picture}(5,3.5)(-1.5,-1.2)
\psset{unit=\unitlength}
\psline(0,-1.2)(0,0)
\psline(-1.2,0)(0,0)(1,1)(1,2.3)
\psline(1,1)(2.1,1)(2.1,-0.9)
\psline(2.1,1)(3.4,2.3)
\put(-1.4,-1.2){\tiny $\alpha_1$=(0,0)}
\put(-1.3,-0.7){\tiny $\varphi=0$}
\put(-0.7,1){\tiny $\alpha_3$=(0,1)}
\put(-0.6,1.5){\tiny $\varphi=\xi_2$}
\put(0.5,-0.8){\tiny $\alpha_2$=(1,0)}
\put(0.6,-0.3){\tiny $\varphi=\xi_1$}
\put(2.5,0){\tiny $\alpha_5$=(2,0)}
\put(2.4,0.5){\tiny $\varphi=2\xi_1\!-\!2$}
\put(1.3,1.6){\tiny $\alpha_4$=(1,1)}
\put(1.2,2.1){\tiny $\varphi=\xi_1\!+\!\xi_2\!-\!1$}
\end{picture}
\qquad \qquad
\setlength{\unitlength}{17mm}
\begin{picture}(2.3,1.5)(-2,-0.5)
\psset{unit=\unitlength,dash=4pt 2pt,dotsep=2pt}
\psline{->}(0,0)(0,1)  
\psline{->}(0,0)(-0.3,0.5)  
\psline{->}(0,0)(-1,1)  
\psline{->}(0,0)(-1.3,0.5)  
\psline{->}(0,0)(-2,1)  
\psline[linestyle=dotted](0,0.75)(-0.225,0.375)
\psline[linestyle=dotted](-0.225,0.375)(-0.75,0.75)
\psline[linestyle=dotted](0,0.75)(-1.5,0.75)
\psline[linestyle=dotted](-0.225,0.375)(-0.975,0.375)
\psline[linestyle=dotted](-0.75,0.75)(-0.975,0.375)
\psline[linestyle=dotted](-1.5,0.75)(-0.975,0.375)
\put(0,1.02){\makebox(0,0)[cb]{\tiny (0,0,1)}}
\put(-1,1.02){\makebox(0,0)[cb]{\tiny (-1,0,1)}}
\put(-2,1.02){\makebox(0,0)[cb]{\tiny (-2,0,1)}}
\put(-0.16,0.45){\makebox(0,0)[lc]{\tiny (0,-1,1)}}
\put(-1.35,0.45){\makebox(0,0)[rc]{\tiny (-1,-1,1)}}
\end{picture}
\caption{Constructing the mirror: $f(x_1,x_2)=1+x_1+x_2+t^{2\pi}x_1x_2+t^{4\pi}x_1^2$}
\label{fig:tropicalexample}
\end{figure}
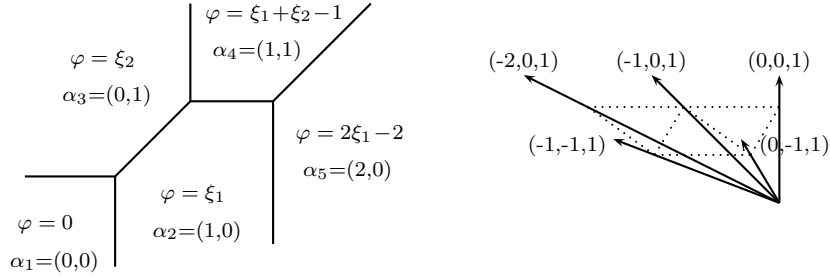

\subsection{Construction as a Hamiltonian reduction}\label{ss:ham_red}
We have a surjective map
\begin{equation}\label{eq:surjlattices}
  \Z^{P_{\Z}} \to \Z^{n} \oplus \Z
\end{equation}
which assigns to a lattice point $\alpha \in P$ the pair $(-\alpha, 1)$; the kernel is a lattice which we denote $M$.

We shall consider the reduction of $\C^{P_{\Z}}$ (equipped with a suitable toric
K\"ahler form, described in \S \ref{s:toricCN}) by the Hamiltonian action of the torus
\begin{equation}
\T_M=M\otimes (\R/\Z) \subset  \T^{P_{\Z}}.
\end{equation}
Fixing a regular value $\lambda$ for the moment map
\begin{equation}
\mu \co  \C^{P_{\Z}} \to \Hom(M, \R)= M^*_\R,
\end{equation}
the quotient $\mu^{-1}(\lambda)/\T_M$ inherits a canonical symplectic form
$\omega_\lambda$. 
By the Kempf-Ness theorem, this quotient can be naturally identified with
the quotient of an open subset of $\C^{P_{\Z}}$ by a complex torus, and the
symplectic form $\omega_\lambda$ is K\"ahler with respect to the induced
complex structure. Thus, $\mu^{-1}(\lambda)/\T_M$ is naturally equipped with
a toric K\"ahler form (induced by that constructed in \S \ref{s:toricCN} on
$\C^{P_\Z}$); see also \cite{Guillemin}.

We now explain how the choice of
level set $\lambda$ corresponds to the integral affine function in Equation \eqref{eq:define_nu}.
Dualizing \eqref{eq:surjlattices} we obtain a short exact sequence 
\begin{equation}
  0\to \R^{n+1}\stackrel{i}{\longrightarrow} \R^{P_\Z}\stackrel{\pi}{\longrightarrow} M^*_\R\to 0,
\end{equation}
where the first map is given explicitly by
\begin{equation}\label{eq:incldual}
  i(\xi_1,\dots,\xi_n,\eta)=\bigl(-\langle \alpha,\xi\rangle+\eta\bigr)_{\alpha\in P_\Z.}
\end{equation}
Viewing the piecewise linear function $\nu:P\to\R$
as an element of $\R^{P_\Z}$, we set $$\lambda=\pi(\nu).$$
Then the image of the moment map for the
action of $\T^{n+1}\simeq \T^{P_\Z}/\T_M$ on $\mu^{-1}(\lambda)/\T_M$ is the
intersection of $\pi^{-1}(\lambda)$ with the non-negative orthant in
$\R^{P_\Z}$, i.e.\ the set of all $(\xi,\eta)\in
\R^{n+1}=\R^n\oplus \R$ such that all the components of $i(\xi,\eta)+\nu$
are non-negative.
Comparing with \eqref{eq:Delta_Y}, this moment polytope 
is precisely $\Delta_Y$.

This yields a Hamiltonian quotient description of $Y$ equipped with the 
toric K\"ahler form $\omega_Y$. Moreover, the function 
\begin{equation}
  W_0=-{\prod}_{\alpha \in P_{\Z}}  z_\alpha \co \C^{P_{\Z}} \to \C
\end{equation}
descends to the toric potential $W : Y \to \C$.  (Note that both are toric
monomials vanishing to order 1 on each toric divisor).
Setting $N=|P_\Z|$, we can thus view
the Landau-Ginzburg model $(Y,W)$ as a Hamiltonian reduction (by $\T_M$) 
of the ``standard'' Landau-Ginzburg model $(\C^{N}, W_0=-\prod_{i=1}^N z_{i})$.

\begin{example}\label{ex:example_22}
In Example \ref{ex:example_21}, the kernel of the map
\eqref{eq:surjlattices}, i.e.\ the space of linear
relations among the $(-\alpha_i,1)$ (the generators of the fan
$\Sigma_Y$, shown on Figure \ref{fig:tropicalexample} right) is
a rank 2 lattice $M$, spanned by elements corresponding to the
linear relations $\alpha_1-\alpha_2-\alpha_3+\alpha_4=0$ and $\alpha_1-2\alpha_2+\alpha_5=0$
among the elements of $P_\Z$.
Thus, we can realize the toric 3-fold $Y$
as the quotient of $\C^5$ by a 2-dimensional
torus $\T_M$ whose generators act with weights $(1,-1,-1,1,0)$ and
$(1,-2,0,0,1)$. The moment map of the $\T_M$-action is obtained from that
of the standard action of $\T^5$ by the projection
$\pi(\mu_1,\dots,\mu_5)=(\mu_1-\mu_2-\mu_3+\mu_4,\mu_1-2\mu_2+\mu_5)$,
where $(\mu_1,\dots,\mu_5)$ take values in the standard moment polytope of
$\C^5$, i.e.\ the non-negative orthant (however, since the toric K\"ahler form on $\C^5$ we will construct in 
Section \ref{s:toricCN} differs from the standard one, it will not be the
case that $\mu_i=\frac12|z_i|^2$).

Setting $\lambda=\pi(\nu)=(1,2)$, we find that $\mu^{-1}(\lambda)\subset
\C^5$ is the set of points whose moment map coordinates for the $\T^5$ action
satisfy
\begin{equation}\label{eq:ex_momentconstraints}
\mu_1-\mu_2-\mu_3+\mu_4=1 \text{ and } \mu_1-2\mu_2+\mu_5=2.
\end{equation}
The moment polytope of the toric variety
$\mu^{-1}(\lambda)/\T_M$ is then the intersection of the non-negative
orthant with the affine subspace determined by
\eqref{eq:ex_momentconstraints}, which is identified with 
$\Delta_Y=\{(\xi_1,\xi_2,\eta)\in
\R^3\,|\,\eta\geq \varphi(\xi_1,\xi_2)\}$
via the affine embedding
$$i(\xi_1,\xi_2,\eta)+\nu=(\eta,\eta-\xi_1,\eta-\xi_2,\eta-\xi_1-\xi_2+1,\eta-2\xi_1+2).$$
\end{example}

\begin{remark}\label{rmk:reduction}
There is a uniform way of producing all the examples that we consider from a universal construction:
$(\C^N,W_0)$ is mirror to an $(N-2)$-dimensional
pair of pants, i.e.\ the intersection of the hyperplane
$X_0+\dots+X_{N-1}=0$ with the open stratum
$(\K^*)^{N-1}$ in $\PP^{N-1}$. 
The embedding of $(\K^*)^n$ into the open stratum of $\PP(\K^{P_\Z})$
defined by $$(x_1,\dots, x_n)\mapsto \bigl(a_\alpha
t^{2\pi \nu(\alpha)}x^\alpha\bigr)_{\alpha \in P_\Z}$$ defines an algebraic
subtorus, whose intersection with the pair of pants is the
hypersurface $\mathcal{H}$. 
Thus, the mirror pairs we consider can be viewed as ``reductions'' of the
mirror pair consisting of the
$(N-2$)-dimensional pair of pants and the Landau-Ginzburg model 
$(\C^N,W_0)$: namely, $\mathcal{H}$ is
the intersection of the pair of pants with an algebraic subtorus, while its mirror 
$(Y,W)$ is the quotient of $(\C^N,W_0)$ by the complementary subtorus.

However, the graph of the projection $\mu^{-1}(\lambda)\to Y$, viewed as a
Lagrangian correspondence in $\C^N\times Y$, bounds non-trivial families
of holomorphic discs; this causes a discrepancy between moduli spaces
of discs in $Y$ with boundary on given Lagrangian submanifolds
of $Y$, and moduli spaces of discs in $\C^N$ with boundary on
the lifts of these Lagrangians to $\mu^{-1}(\lambda)$. The instanton
corrections that arise out of this are responsible for the presence of the
bulk deformation class $\bb\in H^2(Y,\Lambda_{\geq 0})$ in the statement of
Theorem \ref{thm:main}, as we shall see in Section \ref{s:calculation}. 
\end{remark}

\section{The Fukaya category of a Landau-Ginzburg model}\label{s:Fukayacat}

\subsection{Landau-Ginzburg models}\label{sec:LGmodel-setup}
Let $ (Y, \omega) $ be a symplectic manifold, and
\begin{equation}
    W \co Y \to \C
\end{equation}
a map which is a symplectic fibration outside a compact subset of $\C$. 
We shall define a Fukaya category associated to the pair $(Y,W)$, subject 
to additional auxiliary choices: (i) a compatible almost complex structure $J$
making $W$ holomorphic outside a compact subset of $\C$, 
(ii) a continuous function $h \co Y \to [0,\infty)$ which is weakly
$J$-plurisubharmonic,  (iii) a non-negative \emph{wrapping Hamiltonian} 
\begin{equation}
    H \co Y \to \R,
\end{equation}
and (iv) a closed subset $Y^{in}\subset Y$, whose intersection with every fiber of $W$
is a (compact) sublevel set of $h$; more precisely, we take $Y^{in}$ to be
the set of points where $h\leq r(|W|)$, where $r(|W|)$ is a
non-decreasing function of $|W|$, constant over $[0,R_0]$ for some $R_0$.

We require these data to be compatible as follows:
\begin{enumerate}
    \item The restrictions of $h$ and $H$ to every fiber of $W$ are proper.
\smallskip
    \item The Hamiltonian flow of $H$ preserves the fibers of $W$\!, and 
outside of $Y^{in}$ it preserves the
level sets of $h$:
          \begin{equation}
            dW(X_H) = 0, \quad \text{and}\ dh(X_H)=0\text{ outside }Y^{in}.
          \end{equation}
Also, horizontal parallel transport preserves $H$ everywhere, as well as $h$ outside of
$Y^{in}$. By this we mean that, if $\xi^\#$ is the horizontal lift of a vector on $\C$, then
          \begin{equation}
            dH(\xi^\#)=0, \quad \  \text{and}\ dh(\xi^\#)=0\text{ outside }Y^{in}.
          \end{equation}
This in turn guarantees that
horizontal parallel transport is well-defined (except at critical points) despite the non-compactness of the fibers, since horizontal
lifts are contained in the level sets of $H$ which is fiberwise proper.%
\smallskip
\item Outside of $Y^{in}$, 
the $1$-form $d^c h = -dh \circ J$ vanishes on the symplectic orthogonal to the fibers of $W$,  i.e. if
$\xi^\#$ is the horizontal lift of a vector on $\C$, we have
  \begin{equation}\label{eq:dch_horiz}
    d^ch (\xi^\#) = 0.    
  \end{equation}
Moreover, $d^c h$ is  preserved by (i) parallel transport and (ii) the Hamiltonian flow $X_H$, i.e.\ the Lie derivative with respect to $X_H$ and to the horizontal pullback of a vector field $\xi$ on $\C$ both vanish:
   \begin{equation}\label{eq:dch_horizLie}
          \cL_{X_H}   d^ch  = \cL_{\xi^\#} d^ch = 0.            
          \end{equation}
          \item The function $h$ grows along $-JX_H$ outside of $Y^{in}$, i.e.
            \begin{equation} \label{eq:dch_XH}
             0 \leq  d^c h(X_H).               
            \end{equation}
\end{enumerate}

\begin{remark}
Condition (2) essentially states that $W$, $H$ and $h$ Poisson commute
outside of $Y^{in}$. Moreover, the fact that $W$ is holomorphic outside of a
compact subset implies that the horizontal subspace is $J$-invariant, and
hence the vanishings of $dh$ and $d^ch$ on the horizontal distribution are
equivalent to each other. On the other hand, 
the condition $  \cL_{X_H}   d^ch = 0$ is particularly strong, 
and is analogous to considering only \emph{linear} Hamiltonians in the 
situation of a manifold with contact boundary. 
\end{remark}

\begin{remark} \label{rmk:ddch_horiz}
By the Cartan formula, given \eqref{eq:dch_horiz} the condition
$\cL_{\xi^\#} d^ch=0$ is equivalent to requiring that $\iota_{\xi^\#}
dd^ch=0$ for every horizontal vector $\xi^\#$.
\end{remark}

\begin{remark} \label{rmk:horiz_exclude_crit}
In our main examples, the requirements concerning the behavior of $h$ 
along the horizontal
distribution ($dh(\xi^\#)=d^c h(\xi^\#)=0$, $\cL_{\xi^\#} d^ch=0$) only
hold outside of $Y^{in}\cup W^{-1}(\Delta')$, where 
$\Delta'$ is a small neighborhood of $\mathrm{crit}(W)=\{0\}$ in the
complex plane. We will see that this weakening of the assumptions is
not problematic as long as the Lagrangians we consider 
remain outside of $W^{-1}(\Delta')$ and the isotopies of the complex plane
whose lifts we consider are supported outside of $\Delta'$. 
\end{remark}

\begin{definition}\label{def:admissiblearc}
  An \emph{admissible arc} is a properly embedded arc $\gamma:[0,\infty)\to
\C$ that is disjoint from the critical values of $W$ and from the negative
real axis, and along which distance from the origin is strictly increasing
outside of the disc of radius $R_0$.
\end{definition}

\begin{definition}\label{def:admissibleLagrangian}
  An \emph{admissible Lagrangian} with respect to the above data is a
properly embedded Lagrangian $L \subset Y$ such that (i) the image $W(L)
\subset \C$ agrees outside of a compact subset $\Delta$ 
with a finite union of admissible arcs which do not reenter $\Delta$;
and (ii) the restriction of
$d^ch$ to $L$ vanishes outside of $Y^{in}$.
\end{definition}

The main examples we consider below are in fact fibered over properly
embedded arcs in $\C$ which avoid the critical values of $W$ and
are asymptotic to radial straight lines at infinity. In this
case we can take $\Delta$ to be a single base point on the arc.

Given an admissible Lagrangian $L\subset Y$ and an isotopy $\rho^t$ of the 
complex plane, pointwise preserving $\Delta\cup \mathrm{crit}(W)$ (or the slightly larger set $\Delta\cup \Delta'$) and setwise preserving
the negative real axis, there exists a unique Lagrangian
isotopy, which we denote by $\rho^t(L)$, with the following properties:
(i) $\rho^t(L)=L$ in $W^{-1}(\Delta)$, and (ii) outside of $W^{-1}(\Delta)$, 
$\rho^t(L)$ fibers over the collection of arcs which is the image of $W(L)$
under $\rho^t$.  We say that the lifted
isotopy $\rho^t(L)$ is admissible if the images of the arcs under $\rho^t$
are admissible. The Lagrangian $\rho^t(L)$ can be constructed by
intersecting $L$ with $W^{-1}(\Delta)$ and parallel transporting its
boundary along the images of the arcs under $\rho^t$.

\begin{remark}
If the symplectic connection on $W:Y\to\C$ has vanishing curvature outside
of $\Delta$ then $\rho^t$ can be directly constructed as the horizontal lift 
of the isotopy of the base. However, the geometric
models required for our applications do not naturally satisfy this
condition.
\end{remark}

\begin{lemma}\label{l:admissible_invariant}
The set of admissible Lagrangians is invariant under compositions of (i)
Hamiltonian isotopies supported in $Y^{in}$ that preserve the fibers of $W$
outside of a compact subset, (ii) the
Hamiltonian flow of $H$, and (iii) admissible lifted isotopies $\rho^t(L)$
as defined
above.
\end{lemma}

\begin{proof}
The first statement is obvious from the definition.
The Hamiltonian flow of $H$ preserves admissibility because
we have required that $dW(X_H) = 0$, so that the projection to the base is
preserved, and $\cL_{X_H} d^c h = 0$, so that $d^c h$ vanishes on a
Lagrangian if and only if it vanishes on its image under the flow.  
The third statement follows from the observation that parallel transport
along an admissible arc preserves $Y^{in}$ and preserves the vanishing of
$d^c h$ outside $Y^{in}$ by \eqref{eq:dch_horiz} and \eqref{eq:dch_horizLie}.
\end{proof}

We also note that admissible lifted isotopies commute with the Hamiltonian
flow of $H$, so the two operations can be performed in either order.

It will be useful for us to have a more explicit description of lifted
admissible isotopies as Hamiltonian flows. 

\begin{lemma}\label{l:ham_lifted_isotopy}
Given an admissible arc $\gamma:[0,\infty)\to \C$ and a vector field $v$
on the complex plane which vanishes at $\gamma(0)$
and generates an admissible isotopy of arcs $\gamma_t=\rho^t(\gamma)$, 
we define a Hamiltonian $K_{\gamma,t,v}\in C^\infty(W^{-1}(\gamma_t),\R)$ by:
\begin{itemize} 
\item $K_{\gamma,t,v}=0$ everywhere in the fiber $W^{-1}(\gamma_t(0))$,
\item the derivative of $K_{\gamma,t,v}$ along the horizontal lift
of $\gamma_t$ is 
\begin{equation}\label{eq:dK_along_gamma}
dK_{\gamma,t,v}(\dot\gamma_t^\#)=\omega(\dot\gamma_t^\#,v^\#),
\end{equation}
where $v^\#$ is the horizontal lift of $v$.
\end{itemize}
Denote by $\psi^t$ the Hamiltonian flow generated by (an arbitrary extension
of) $K_{\gamma,t,v}$.

Then, for any point $p\in W^{-1}(\gamma(0))$, $\psi^t$ maps the horizontal
lift of $\gamma$ through $p$ to the horizontal lift of $\gamma_t$ through
$p$. In particular, if $L$ is an admissible Lagrangian which fibers over
$\gamma$, then $\psi^t(L)=\rho^t(L)$. 

Moreover, at every point of $W^{-1}(\gamma_t)$ which lies outside of
$Y^{in}$, the Hamiltonian vector field $X_{\gamma,t,v}$ generated by
$K_{\gamma,t,v}$ satisfies 
\begin{equation}\label{eq:dch_lifted}
dh(X_{{\gamma,t,v}})=d^ch(X_{{\gamma,t,v}})=0 \quad \text{and} \quad
\iota_{X_{{\gamma,t,v}}}dd^ch=0.
\end{equation}
\end{lemma}

\begin{remark}\label{rmk:ham_lifted_isotopy}
The ambiguity in extending $K_{\gamma,t,v}$ to a neighborhood
of $W^{-1}(\gamma_t)$ affects $X_{\gamma,t,v}$ by a multiple of
$\dot\gamma_t^\#$, which does not affect the conclusions of the lemma,
but implies that the isotopy  $\psi^t$ that we construct does not in
general lift the isotopy $\rho_t$ in the sense that the  $W \circ
\psi^t = \rho_t$. By  appropriately choosing the extension of the
Hamiltonian, we may arrange to have such a lift for a fixed point
$p\in W^{-1}(\gamma(0))$, but the curvature of the symplectic
connection on $W:Y\to\C$ prevents the existence of a lift
simultaneously for all $p$.


We note for future reference that $K_{\gamma,t,v}$ can be extended to a
smooth Hamiltonian whose support is contained in a small neighborhood of
$W^{-1}(\gamma_t)$ and such that the corresponding vector field
satisfies \eqref{eq:dch_lifted} everywhere outside of $Y^{in}$. The simplest
way to do this is to foliate a neighborhood of $\gamma$ in the complex plane
by a family of admissible arcs $\gamma^\tau$, $\tau\in (-\tau_0,\tau_0)$,
and consider a Hamiltonian which equals
$\chi(\tau) K_{\gamma^\tau,t,v}$ over 
the preimage of $\rho^t(\gamma^\tau)$, where $\chi(\tau)$ is a cut-off function with
compact support.
\end{remark}

\proof[Proof of Lemma \ref{l:ham_lifted_isotopy}]
Since \eqref{eq:dK_along_gamma} can be rewritten as $\omega(\dot
\gamma_t^\#,X_{\gamma,t,v})=\omega(\dot\gamma_t^\#,v^\#)$, the vector
field $X_{\gamma,t,v}-v^\#$ is symplectically orthogonal to
$\dot\gamma_t^\#$, hence tangent to $W^{-1}(\gamma_t)$. It follows that
the flow $\psi^t$ maps $W^{-1}(\gamma)$ to $W^{-1}(\gamma_t)$. 

Since $\psi^t$ is a Hamiltonian diffeomorphism, it maps Lagrangian
submanifolds of $Y$ which fiber over $\gamma$ to Lagrangian submanifolds
which fiber over $\gamma_t$. Moreover, since $v$ vanishes at $*=\gamma(0)$, the Hamiltonian $K_{\gamma,t,v}$ 
and its first derivative both vanish along $W^{-1}(*)$, hence $X_{\gamma,t,v}=0$ 
everywhere in $W^{-1}(*)$.
In particular, given a Lagrangian $\ell\subset
W^{-1}(*)$, $\psi^t$ maps the parallel transport of $\ell$ over $\gamma$ to
the parallel transport of $\ell$ over $\gamma_t$. Now consider 
two small Lagrangian discs $\ell_1,\ell_2\subset W^{-1}(*)$ which intersect
transversely at a given point $p$. The parallel transports of $\ell_1$ and 
$\ell_2$ over $\gamma$ intersect cleanly along the horizontal lift of $\gamma$ through
$p$, and are mapped by $\psi^t$ to the parallel transports of $\ell_1$ and
$\ell_2$ over $\gamma_t$, which intersect along the horizontal lift of
$\gamma_t$ through $p$. Thus, $\psi^t$ maps horizontal lifts of $\gamma$ to horizontal
lifts of $\gamma_t$.

In order to prove \eqref{eq:dch_lifted}, we consider the map
$u:W^{-1}(*)\times [0,\infty)\times [0,t_0]\to Y$ such that
$u(p,s,t)$ is the point of $W^{-1}(\gamma_t(s))$ obtained by
parallel transport of $p$ over $\gamma_t$. In other terms, $u(p,0,0)=p$, and
$\partial_s u=\dot\gamma_t^\#$. 

Since the flow $\psi^t$ maps $u(\{p\}\times
[0,\infty)\times \{0\})$ to $u(\{p\}\times [0,\infty)\times \{t\})$ for all
$p$, the vector field $X_{\gamma,t,v}$ lies in the span of $\partial_t u$ and 
$\partial_s u$. On the other hand, $\partial_s u=\dot\gamma_t^\#$ lies in
the kernel of $d^c h$ and $dd^c h$ by \eqref{eq:dch_horiz} and
\eqref{eq:dch_horizLie}.

The 2-form $u^* dd^c h$ vanishes on $\partial_s$, so it
can be written in the form $$u^* dd^c h = dt\wedge \alpha(s,t) + \beta(s,t)$$
where $\alpha(s,t)$ and $\beta(s,t)$ are forms on $W^{-1}(*)$. 
Since $\partial_t u=0$ whenever $s=0$, we find that $\alpha(0,t)\equiv 0$, and
$\beta(0,t)=dd^ch_{|W^{-1}(*)}$ is independent of $t$.  On the other hand,
$u^*dd^ch$ is closed, so necessarily $\alpha$ and $\beta$ are independent of
$s$, i.e.\ $\alpha(s,t)\equiv 0$ and $\beta(s,t)\equiv \beta_0=dd^ch_{|W^{-1}(*)}$.
We conclude that the span of $\partial_s$ and $\partial_t$ lies in the
kernel of $u^*dd^ch$, and hence $X_{\gamma,t,v}$ lies in the kernel of $dd^c h$.

Similarly, $u^* d^ch$ vanishes on $\partial_s$, so it can be written in the
form $$u^* d^ch = f(s,t)\,dt + \eta(s,t)$$ for $\eta(s,t)$ a 1-form on $W^{-1}(*)$.
Using again the fact that $\partial_t u=0$ for $s=0$, we find that
$f(0,t)\equiv 0$ and $\eta(0,t)$ is independent of $t$. Moveover, since
$u^*dd^ch$ vanishes on the span of $\partial_s$ and $\partial_t$, we have
$\partial_s f=0$, so that $f(s,t)\equiv 0$. This in turn implies that $u^*d^c h$
vanishes on the span of $\partial_s$ and $\partial_t$, and hence
$d^ch(X_{\gamma,t,v})=0$.

Finally, the vanishing of $dh(X_{\gamma,t,v})$ is a direct consequence of
the assumption that horizontal parallel transport preserves the levels of
$h$ outside of $Y^{in}$.
\endproof

\subsection{Maximum principle and energy estimates}
\label{sec:maximum-principle}

Our construction of the Fukaya category of a Landau-Ginzburg model involves
not only structure maps for Lagrangian Floer theory
with boundary on admissible Lagrangians, but also
natural morphisms and continuation maps associated to certain isotopies of
admissible Lagrangians. In this section we establish the results needed to
prove compactness of the corresponding moduli spaces.

Let $\Sigma$ be the complement of finitely many boundary marked points on a
compact Riemann surface with boundary, 
and $\Lambda$ a moving family of admissible Lagrangian boundary conditions on
$\Sigma$, i.e.\ a smoothly varying family of admissible Lagrangian submanifolds 
of $Y$, constant near the ends of each component of $\partial
\Sigma$.  The manner in which
$\Lambda$ varies along the boundary of $\Sigma$ can be described by a compactly 
supported 1-form on $\partial\Sigma$ with values in vector fields. 

We assume
that $\Lambda$ varies along each boundary component by a combination of
(i)~a multiple of the flow of the wrapping Hamiltonian $H$, namely
$X_H\otimes \eta$ for $\eta$ a 1-form on $\partial \Sigma$, and (ii)
the lift of an admissible isotopy of the complex plane supported away from
$\Delta\cup \Delta'$, where $\Delta'\supset \mathrm{crit}(W)$ (cf.\ Remark
\ref{rmk:horiz_exclude_crit}). We note that Lemma \ref{l:admissible_invariant} asserts the invariance of the class of admissible Lagrangians under this class of isotopies.
We shall impose the following (semi)-positivity assumption on the isotopy:
\begin{itemize}
\item In the fiber direction, we require:
\begin{equation}
\label{eq:maxprinciple_cond_fiberwise}
\parbox{\mywidth}{The total fiberwise wrapping is non-positive, i.e.\
$\displaystyle \int_{\partial\Sigma} \eta\leq 0$.}
\end{equation}
\item In the base direction, denote by $\Gamma$ the family of admissible arcs in
the complex plane to which $\Lambda$ projects outside of $\Delta$. We
assume:
\begin{equation}
\label{eq:maxprinciple_cond_base}
\parbox{\mywidth}{There exists an isotopy $\rho^t$ of the complex plane
rel.\ $\Delta\cup \Delta'$,
and a function $\tau\in C^\infty(\Sigma,\R)$ which is constant near the punctures,
such that along each component of $\partial\Sigma$ the arcs $\rho^\tau(\Gamma)$ 
vary by an admissible isotopy that moves in 
the clockwise direction outside of a
compact set.}
\end{equation}
\end{itemize}

For example, if $\Gamma$ only moves in the clockwise direction
outside of a compact set (or does not move at all) then we can take the 
isotopy $\rho^t$ to be trivial, and $\tau\equiv 0$.

Condition \eqref{eq:maxprinciple_cond_fiberwise} implies the existence
of a 1-form $\alpha$ on $\Sigma$ with the
following two properties:
\begin{equation}\label{eq:alpha_subclosed}
\text{$\alpha$ is sub-closed, i.e., $d\alpha\leq 0$;}
\end{equation}
\begin{equation}\label{eq:alpha_boundary_neg}
\alpha_{|\partial \Sigma}\geq \eta \quad \text{pointwise along
$\partial\Sigma$.}
\end{equation}
(As is customary for Floer theory, $\alpha_{|\partial\Sigma}$ and $d\alpha$
should also be compactly supported).
For instance, if $\eta$ is pointwise non-positive, then we can take
$\alpha\equiv 0$.

We consider maps $u \co \Sigma \to Y$ with boundary conditions given by $\Lambda$
(i.e., $u(z)\in \Lambda_z$ for all $z\in \partial\Sigma$) and subject to
a convergence condition (see \eqref{eq:converge_at_punctures} below),
satisfying the perturbed pseudo-holomorphic curve equation
\begin{equation} \label{eq:pseudo-holomorphic_equation}
(du - X_H \otimes \alpha + (\xi^\tau)^\# \otimes d\tau)^{0,1} = 0,  
\end{equation}
where $\xi^t$ is the vector field on the complex plane which generates
the isotopy $\rho^t$ in \eqref{eq:maxprinciple_cond_base}, and $(\xi^t)^\#$ is its horizontal lift to $Y$.
The expression \eqref{eq:pseudo-holomorphic_equation} takes values in the
space of $(0,1)$-forms on $\Sigma$ with values in $u^*TY$, i.e.\ complex 
antilinear maps from $T\Sigma$ to $u^*TY$. 
(We only need to consider smooth maps, but as usual in Lagrangian Floer theory,
the functional analysis setup involves an extension
to a space of maps of suitable Sobolev regularity; see e.g.\ \cite[Chapter 8]{SeBook}.)
We will also consider modifications of this equation by further adding a
compactly supported inhomogeneous perturbation term for transversality
purposes.

The inhomogeneous term $X_H\otimes \alpha$ in \eqref{eq:pseudo-holomorphic_equation} is 
the same type of Hamiltonian perturbation that commonly appears in the
construction of continuation maps (and other operations) in (ordinary or wrapped)
Lagrangian Floer theory,
and the term $(\xi^\tau)^\#\otimes d\tau$ plays a similar role in the horizontal direction.  
In the presence of moving boundary conditions, one frequently requires
that the restriction of the inhomogeneous perturbation term to $\partial\Sigma$ 
generates the isotopy by which $\Lambda$ varies, see e.g.\ \cite[Section 8k]{SeBook}.
However, when the variation of $\Lambda$ is pointwise non-positive
everywhere along $\partial\Sigma$ the maximum principle readily holds without
the need for inhomogeneous terms; our setup encompasses both
cases. 

The vanishing of $\xi^t$ inside $\Delta'$ ensures that, even if the
compatibility of $h$ with the horizontal distribution is relaxed over
$W^{-1}(\Delta')$ as in Remark \ref{rmk:horiz_exclude_crit}, the quantities
$dh((\xi^t)^\#)$, $d^c h((\xi^t)^\#)$ and $\iota_{(\xi^t)^\#}dd^ch$ still vanish
identically outside of $Y^{in}$.

We only ever consider finite energy solutions to
\eqref{eq:pseudo-holomorphic_equation}, in the sense that the  {\em geometric energy} 
\begin{equation}
 E_{geom}(u):=\int_\Sigma |du-X_H\otimes
\alpha+(\xi^\tau)^\#\otimes d\tau|^2\,dvol_\Sigma   
\end{equation}
 is finite. The norm in the above integral is taken with respect to the
metric induced by $\omega$ and $J$ on $Y$, and any $j$-compatible metric on $\Sigma$ (the integrand is conformally invariant). By the usual decay estimates for solutions of Floer's equation on
strips, this is equivalent to the condition that
\begin{equation}\label{eq:converge_at_punctures}
\parbox{\mywidth}{near each puncture of
$\Sigma$, $u$ converges to a generator of the
appropriate Floer complex (i.e., when the perturbation term is compactly supported
over $\Sigma$, an intersection point between the boundary
conditions $\Lambda$ on either side of the puncture).}
\end{equation}

\begin{proposition} \label{prop:maxprinciple_base}
Assuming \eqref{eq:maxprinciple_cond_base},
solutions to \eqref{eq:pseudo-holomorphic_equation} satisfy the maximum principle with respect to
the quantity $|\rho^\tau\circ W|$ (outside of a compact subset of $\C$).
\end{proposition}

\proof Outside of a compact subset of $\C$, $W$ is $J$-holomorphic, so
$w=W\circ u$ solves the perturbed Cauchy-Riemann equation
\begin{equation}\label{eq:pseudohol_projected}
(dw+\xi^\tau\otimes d\tau)^{0,1}=0.
\end{equation}
Hence, $\tilde{w}=\rho^\tau\circ w:\Sigma\to\C$ solves an unperturbed Cauchy-Riemann
equation with respect to the domain-dependent complex structure $(\rho^\tau)_*
j$ on the complex plane:
\begin{equation}\label{eq:pseudohol_projected_rewritten}
(d\tilde{w})^{0,1}_{(\rho^\tau)_* j}=0,
\end{equation}
and the maximum principle holds at interior points.
Along $\partial \Sigma$ we use a variant of the maximum principle with
Neumann boundary conditions. Namely,  pick local coordinates $z=s+it$ which
locally identify $\Sigma$ with the upper half-plane. If
$|\tilde{w}|$ has a local maximum, then necessarily 
$$\partial_s |\tilde{w}|=0 \ \ \mathrm{and} \ \ \partial_t |\tilde{w}|<0.$$ 
It follows that $\partial_s \arg(\tilde{w})> 0$,
since otherwise $\partial_t \tilde{w}$ would point clockwise from
$\partial_s \tilde{w}$, contradicting \eqref{eq:pseudohol_projected_rewritten}.

On the other hand, recall that the boundary conditions for $\tilde{w}$ are given by the family
of admissible arcs $\rho^\tau(\Gamma)$, along which the distance from the origin 
is strictly increasing. Thus, at a
boundary maximum, $\partial_s\tilde{w}$ points counterclockwise
from the tangent vector to $\rho^\tau(\Gamma)$. This contradicts the assumption 
\eqref{eq:maxprinciple_cond_base}, and we conclude that
$|\tilde{w}|$ has no local maxima.
\endproof

\begin{proposition}\label{prop:maxprinciple_fiber}
Solutions to \eqref{eq:pseudo-holomorphic_equation} satisfy the maximum principle with respect to $h$
(outside of $Y^{in}$).
\end{proposition}
\begin{proof}
The argument is similar to other instances of the maximum principle in Floer theory:
since $h$ is weakly plurisubharmonic,
its values along a holomorphic curve satisfy the maximum principle at
interior points, and also at the boundary under the assumption that $d^c h$ 
vanishes there; the conditions \eqref{eq:dch_horiz}--\eqref{eq:dch_XH}, which govern
the behavior of $d^c h$ along the directions of the inhomogeneous
terms appearing in \eqref{eq:pseudo-holomorphic_equation}, ensure that the
maximum principle continues to hold for solutions of the perturbed Cauchy-Riemann equation,
as we now show by an explicit calculation.

We begin by showing that the maximum principle for $h \circ u$ holds at interior points. Let $z = x + i y$ be coordinates near a point in $\Sigma$.  Since $h$ is weakly plurisubharmonic, we have
\begin{align}
\nonumber
 0 & \leq d d^c h \bigl( \partial_x u - X_H\!\otimes\! \alpha(\partial_x)+\partial_x\tau\cdot(\xi^\tau)^\#, 
J \bigl(\partial_x u - X_H\!\otimes\!\alpha(\partial_x)+\partial_x\tau\cdot(\xi^\tau)^\#\bigr)\bigr) \\
& =  d d^c h \bigl( \partial_x u - X_H \!\otimes\! \alpha(\partial_x)+\partial_x\tau\cdot(\xi^\tau)^\#, 
\partial_y u - X_H \!\otimes\! \alpha(\partial_y)+\partial_y \tau\cdot(\xi^\tau)^\#) \bigr) \\
\nonumber & = \left(  u^{*}(d d^c h) - \alpha \wedge u^*\left( \iota_{X_H} d d^c h \right)
+ d\tau\wedge u^*(\iota_{(\xi^\tau)^\#} dd^ch\right) ( \partial_x, \partial_y ). 
\end{align}
By the Cartan formula, we have
\begin{equation}
    d( \iota_{X_H} d^c h ) =  -\iota_{X_H} d d^c h + \cL_{X_H} d^c h,
\end{equation}
where the second term vanishes by assumption \eqref{eq:dch_horizLie},
whereas $\iota_{(\xi^\tau)^\#} dd^ch=0$ by Remark \ref{rmk:ddch_horiz}, so we conclude that
\begin{equation}
  0 \leq   u^{*}(d d^c h) - d \left( u^* d^c h(X_H) \right) \wedge  \alpha,
\end{equation}
where the right hand side is considered as a $2$-form on $\Sigma$. The Leibniz rule implies that
\begin{equation}
  d (   u^* d^c h(X_H)  \cdot \alpha) =   d \left( u^* d^c h(X_H)  \right) \wedge \alpha + u^* d^c h(X_H) \cdot d \alpha,
\end{equation}
so we derive the inequality:
\begin{equation}
  0 \leq   u^{*}(d d^c h) -  d (   u^* d^c h(X_H)  \cdot \alpha) + u^* d^c h(X_H) \cdot d \alpha.
\end{equation}
The assumptions that $0 \leq d^c h(X_H)$ and that $\alpha$ is subclosed imply that
\begin{equation}\label{eq:ddc_h_composed_u_ineq}
  0 \leq    u^{*}(d d^c h) -  d (   u^* d^c h(X_H)  \cdot \alpha).
\end{equation}
We claim that the right hand side is the Laplacian of $h \circ u$. Indeed,
since $dh(X_H)=dh((\xi^\tau)^\#)=0$ and
$d^ch((\xi^\tau)^\#)=0$ by assumption, we compute that
\begin{align}
\nonumber  d^c (h \circ u) & = -dh \circ du \circ j \\
\nonumber & = -dh\circ \left( du\circ j - X_H\otimes \alpha \circ j+(\xi^\tau)^\#\otimes d\tau\circ j\right)\\
\label{eq:dc_h_composed_u} 
& = -dh \circ  \left( J \circ du - J X_H \otimes \alpha + J(\xi^\tau)^\# \otimes d\tau \right) \\
\nonumber & = u^*(d^c h) - u^* d^c h (X_H) \cdot \alpha + u^* d^c h((\xi^\tau)^\#)\cdot d\tau \\
\nonumber & = u^*(d^c h) - u^* d^c h (X_H) \cdot \alpha.
\end{align}
Hence,
\begin{equation}
  d d^c (h \circ u) = u^*(d d^c h) - d( u^* d^c h (X_H) \cdot \alpha),
\end{equation}
and comparing with \eqref{eq:ddc_h_composed_u_ineq}, we conclude that
\begin{equation}
dd^c(h\circ u)\geq 0.
\end{equation}
Thus, the maximum principle holds at interior points.

Along $\partial \Sigma$ we use the maximum principle with Neumann boundary
conditions. For this, we need to check that, in local coordinates $z=s+it$
which locally identify $\Sigma$ with the upper half-plane, 
the inequality $d(h\circ u)(\partial_t)\geq 0$ holds, or equivalently,
$d^c(h\circ u)(\partial_s)\leq 0$. 
We have computed above that
\begin{equation}
d^c (h \circ u) =   u^*(d^c h) - u^* d^c h (X_H) \cdot \alpha,
\end{equation}
and we now need to check that the restriction of this 1-form to
$\partial\Sigma$ is everywhere non-positive. 

The vanishing of $d^c h$ on each admissible Lagrangian $\Lambda_s$, by
Definition \ref{def:admissibleLagrangian}, and on the vector fields which
generate lifted admissible isotopies, by Lemma \ref{l:ham_lifted_isotopy},
imply that the only contribution to $u_{|\partial\Sigma}^*(d^ch)$ comes
from the fiberwise wrapping term $X_H\otimes \eta$ in the moving boundary
condition, so
$$d^c(h\circ u)_{|\partial \Sigma}=u^*d^ch(X_H)\cdot \eta - u^* d^ch(X_H)\cdot
\alpha_{|\partial\Sigma}.$$
The non-positivity of this quantity is now immediate, since
$d^c h(X_H)\geq 0$ and $\alpha_{|\partial\Sigma}\geq \eta$ pointwise by
assumption.
\end{proof}

\begin{remark} \label{rmk:maxprinciple_maxhv} In our setting, rather than being smooth, $h$ will be given
by the maximum of a finite collection of smooth plurisubharmonic functions
$h_{\mathbf{v}}$, where for each $\mathbf{v}$ the 1-form $d^c h_{\mathbf{v}}$
satisfies all the required
properties wherever $h_{\mathbf{v}}$ achieves the maximum (i.e., $h_{\mathbf{v}}=h$)
outside of $Y^{in}$. The above argument gives the maximum principle for all
$h_{\mathbf{v}}$ which achieve the maximum, and hence a fortiori for
$h=\max\{h_{\mathbf{v}}\}$.
\end{remark}

The next result asserts the existence of a bound of the geometric energy of solutions to \eqref{eq:pseudo-holomorphic_equation}: such a bound is necessary to appeal to any version of Gromov's compactness theorem, and requires fixing  a homotopy class   $\beta$ of maps from $(\Sigma,\partial\Sigma)$ to $(Y,\Lambda)$
with fixed asymptotic conditions,  given by generators of the Floer complexes, at the punctures of $\Sigma$.  The key point is that  Propositions \ref{prop:maxprinciple_base} and \ref{prop:maxprinciple_fiber}
provide maximum principles for the solutions of
\eqref{eq:pseudo-holomorphic_equation} in both base and fiber directions,
so that solutions which converge to given generators at the punctures of
$\Sigma$ remain within a fixed compact subset of $Y$. It thus suffices to bound the difference between the topological and geometric energy for solutions to a perturbed Cauchy-Riemann equation with image lying in a bounded region; this goes back all the way to Gromov's original paper \cite{Gromov} which established compactness for perturbed equations, and is standard for Hamiltonian perturbations. We nonetheless provide a detailed proof because of the (non-standard) appearance of the horizontal lift in our equation.
\begin{proposition}\label{prop:energy}
There is a constant $E_{max}(\beta)$ so that all 
solutions $u$ to \eqref{eq:pseudo-holomorphic_equation} in the homotopy class $\beta$ 
 satisfy the a priori bound 
\begin{equation}\label{eq:geom_energy}
E_{geom}(u) \leq E_{max}(\beta) 
\end{equation}
\end{proposition}

\proof
Let $z=x+iy$ be coordinates near a point of $\Sigma$. Since
$du-X_H\otimes \alpha+(\xi^\tau)^\#\,d\tau$ is complex linear with respect
to $j$ and $J$, the integrand in \eqref{eq:geom_energy} is equal to
$$\omega\left(\partial_x u-X_H\!\otimes\!\alpha(\partial_x)+\partial_x\tau\,(\xi^\tau)^\#,
\partial_y u-X_H\!\otimes\!\alpha(\partial_y)+\partial_y\tau\,(\xi^\tau)^\#\right).$$
Since $X_H$ is tangent to the fibers of $W$ and $(\xi^\tau)^\#$ is
horizontal, $\omega(X_H,(\xi^\tau)^\#)=0$,
and so 
\begin{equation}\label{eq:geom_energy2}
E_{geom}(u)=\int_\Sigma u^*\omega -\alpha\wedge u^*(\iota_{X_H}\omega)+
d\tau \wedge u^*(\iota_{(\xi^\tau)^\#}\omega).
\end{equation}
This quantity is not invariant under deformations of the map $u$ relative to
the boundary condition $\Lambda$. On the other hand, the
variation of $\Lambda$ along the boundary of $\Sigma$ is described by a
vector field valued 1-form on $\partial\Sigma$ of the form $$X_H\otimes
\eta+X_K\otimes \vartheta,$$
where $K$ is a family of Hamiltonians (dependent on the point of
$\partial\Sigma$) generating the lifted isotopy, as in Lemma
\ref{l:ham_lifted_isotopy}. Then the variation of $\int_\Sigma u^*\omega$
along a vector field $v$ (tangent to $\Lambda$ at the boundary) is equal to
$$\int_{\partial\Sigma} \omega(v,\partial_s u)\,ds = \int_{\partial\Sigma}
\omega(v, X_H)\,\eta+\omega(v,X_K)\,\vartheta=\int_{\partial\Sigma}
dH(v)\,\eta+dK(v)\,\vartheta,$$ so the
{\em topological energy}
\begin{equation}\label{eq:top_energy}
E_{top}([u])=\int_\Sigma u^*\omega-\int_{\partial\Sigma} u^*H\cdot \eta
-\int_{\partial\Sigma} u^*K\cdot \vartheta\end{equation}
depends only on the relative homotopy class $[u]$ of the map $u$.

Returning to Equation \eqref{eq:geom_energy2}, Stokes' theorem expresses the second term as 
$$\int_\Sigma -\alpha\wedge u^*(\iota_{X_H}\omega)=-\int_\Sigma 
u^* dH\wedge \alpha =-\int_{\partial \Sigma} u^* H\cdot \alpha+\int_\Sigma
u^*H \cdot d\alpha.$$
Putting this together with Equation \eqref{eq:top_energy}, we conclude
that
\begin{multline}\label{eq:Egeom_vs_Etop}
E_{geom}(u)=E_{top}([u])+\int_{\partial\Sigma} u^*H\cdot
(\eta-\alpha_{|\partial\Sigma})+\int_\Sigma u^*H \cdot d\alpha \\
+ \int_{\partial\Sigma} u^*K\cdot \vartheta
+ \int_\Sigma d\tau \wedge u^*(\iota_{(\xi^\tau)^\#}\,\omega).
\end{multline}

The first two integrals in the right-hand side of 
\eqref{eq:Egeom_vs_Etop} are non-positive, since $H\geq 0$ by assumption
and $\alpha$ is required to satisfy \eqref{eq:alpha_subclosed} and \eqref{eq:alpha_boundary_neg}.

The existence of a compact subset $\Omega\subset Y$ which a priori contains the image of $u$  (as a consequence of Propositions \ref{prop:maxprinciple_base} and \ref{prop:maxprinciple_fiber}) provides a bound for the last two terms as follows: The third integral can be bounded by
$(\sup_\Omega |K|) \|\vartheta\|_{L^1(\partial\Sigma)}$, which depends only
on the size of $\Omega$ and the geometric bounds on the lifted isotopy of
the boundary condition $\Lambda$ within the compact subset $\Omega$.
Finally, the last integral can be rewritten as
\begin{equation} \label{eq:energy_horizterm}
\int_\Sigma d\tau\wedge \left(\iota_{(\xi^\tau)^\#}\omega \circ
du\right)=\int_\Sigma
d\tau \wedge \left(\iota_{(\xi^\tau)^\#}\omega \circ
(du-X_H\otimes\alpha+(\xi^\tau)^\#\otimes d\tau)\right).\end{equation}
Since the vector field $\xi^\tau$ vanishes at the critical values of $W$,
the norm of its horizontal lift $(\xi^\tau)^\#$ is bounded everywhere in
$\Omega$, and we can bound \eqref{eq:energy_horizterm} by
$$\bigl(\sup_\Omega |(\xi^\tau)^\#|\bigr)\,\|d\tau\|_{L^2(\Sigma)}\,
\bigl\|du-X_H\otimes \alpha+(\xi^\tau)^\#\otimes d\tau\bigr\|_{L^2(\Sigma)}.$$
Combining these bounds, we find that
\begin{equation}\label{eq:Egeom_bound}
E_{geom}(u)\leq E_{top}([u])+\bigl(\sup_\Omega
|K|\bigr)\,\|\vartheta\|_{L^1}+ \bigl(\sup_\Omega
|(\xi^\tau)^\#|\bigr)\,\|d\tau\|_{L^2} E_{geom}(u)^{1/2}.
\end{equation}
This implies a bound on $E_{geom}(u)$ in terms of the other quantities
appearing in \eqref{eq:Egeom_bound}.
\endproof

\begin{remark}
Proposition \ref{prop:energy} continues to hold if
\eqref{eq:pseudo-holomorphic_equation} is further modified by a compactly
supported (hence uniformly bounded) inhomogeneous perturbation term.
\end{remark}

\begin{remark}\label{rmk:invariant_up_to_homotopy}
In the next sections we will define Floer-theoretic operations in terms
of certain moduli spaces of solutions to (compactly supported perturbations of)
\eqref{eq:pseudo-holomorphic_equation}.
In each case we will make specific choices for the parameters $\alpha$ and $\tau$,
but we note that, since the set of allowable choices is contractible hence
connected, the operations we define are independent of these up to homotopy. Likewise for other auxiliary
data such as compactly 
supported inhomogeneous perturbation terms or deformations of the almost complex structure.
\end{remark}

\subsection{Definition of the directed category}
\label{sec:defin-direct-categ}

We fix a collection $\mathbf{L}$ of admissible Lagrangians in $Y$, for which 
the subset $\Delta$ appearing in Definition \ref{def:admissibleLagrangian}
is always the same, and whose images in $\C$ agree  near infinity with a fixed finite
collection of radial straight lines. (In our case $\Delta$ will be the
single point $\{-1\}$). Also fix a subset $\Delta'\supset \mathrm{crit}(W)$ (in our case $\Delta'$ will be a small disc centered at
the origin).

Let $\rho$ be an autonomous flow on $\C$ which fixes $\Delta\cup \Delta'$
and the negative real axis, maps radial lines to radial lines away from a compact set, and
moves all radial lines other than the negative real axis in the counterclockwise direction.
This isotopy preserves the admissibility of the arcs over which the objects of $\bfL$
fiber outside of $\Delta$. 
We define
\begin{equation}
  L(t) := \phi^{t} \rho^{t}(L),  
\end{equation}
where $\phi^t$ is the flow of the wrapping Hamiltonian $H$, and $\rho^t$ is
the lifted admissible isotopy generated by $\rho$. Since $\phi^t$ and
$\rho^t$ commute, we can think of this as an autonomous
flow on $Y$, in particular $(L(t))(t')=L(t+t')$.

By construction, the admissible arcs over which $L(t+\lambda)$ and $L'(t)$ fiber
outside of $\Delta$ are asymptotic to different straight lines for all but finitely many values
of $\lambda$. We will essentially require that, in the fiberwise direction, these
Lagrangians also go to infinity in different directions for generic $\lambda$, so
that their intersections are contained in a compact subset. More precisely,
we assume: 
\begin{equation} 
\label{eq:t-distance_bounded_from_0}
\parbox{\mywidth}{there exists an open (or Baire) dense set $U\subset\R$
such that, for all $L$ and $L'$ in $\bfL$ and $\lambda\in U$, $L(\lambda)\cap L'$ is contained in a
compact subset of $Y$ (the same then holds for $L(t+\lambda)\cap L'(t)$ for all $t\in \R$).}
\end{equation}
(In our case it will be possible to choose the compact subset in \eqref{eq:t-distance_bounded_from_0} to be
independent of $L$, $L'$ and $\lambda$, but there is no reason to require this in general.)


In addition, we impose the following conditions on elements $L \in \bfL$:
\begin{align} \label{eq:no_hol_discs}
  & \parbox{\mywidth}{for all $t\in \R$, $L(t)$ does not bound any (unperturbed) holomorphic discs;} \\ \label{eq:spingraded}
& \parbox{\mywidth}{$L$ is equipped with a spin structure and with a grading (i.e., after choosing a
holomorphic volume form $\Omega$ on $Y$, a lift of the phase map $\arg(\Omega_{|L})$ to $\R$).}
\end{align}

\noindent 
Condition \eqref{eq:no_hol_discs}, which may be replaced by unobstructedness,
ensures that Floer homology is well-defined; while \eqref{eq:spingraded}
ensures that it is $\Z$-graded and can be constructed over a field of characteristic zero.

We will also on occasion equip Lagrangians in $\mathbf{L}$
with local systems; since this will only come up in specific places, we omit
local systems from the notation for now.

\begin{lemma}\label{l:good_arithmetic_progression}
There are arbitrarily small values of $\epsilon>0$ such that, for each pair of Lagrangians $L_0, L_1 \in \bfL$, and
for all integers $ k_0 \neq k_1 $,
\begin{equation} \label{eq:good_arithmetic_progression}
 \parbox{\mywidth}{the images of $L_0(\epsilon k_0)$ and $L_1(\epsilon k_1)$
under $W$ are asymptotic to different radial straight lines in $\C$, and
$L_0(\epsilon k_0)\cap L_1(\epsilon k_1)$ is compact.}
\end{equation}
\end{lemma}
\begin{proof}
After removing a finite set of
values $u$ from the set $U$ in Condition \eqref{eq:t-distance_bounded_from_0}
we can assume that for $\lambda\in U$ the images of $L(t+\lambda)$ and 
$L'(t)$ under $W$ are asymptotic to different radial straight lines 
in $\C$. Now the desired properties hold whenever
$\epsilon$ lies in the intersection of the sets $k^{-1}\cdot U\subset \R$,
for all positive integers $k$.
This is a countable intersection of Baire sets and hence dense as well.
\end{proof}



Choose $0< \epsilon$ such that 
Condition  \eqref{eq:good_arithmetic_progression} holds for all pairs of objects. 

We construct a directed category  $\O$ with objects $L^k:=L(-\epsilon k)$ for 
all $k  \in \Z$ and $L \in \bfL$, whose morphisms are 
\begin{equation}
\O(L_{0}^{k_0},L_{1}^{k_1}) \equiv \begin{cases} CF^{*}(L_{0}(-\epsilon k_0), L_{1}(-\epsilon k_1))  & \textrm{ if } k_0 < k_1 \\
\K\cdot\mathrm{id} &   \textrm{ if } k_0 = k_1 \textrm{ and } L_0 =  L_1 \\
0 & \textrm{ otherwise.}
\end{cases} 
\end{equation}

The $A_{\infty}$ structure is obtained by counting solutions to pseudo-holomorphic curve equations
(for suitable $J$, see Remark \ref{rmk:genericJtransverse})
with compactly supported inhomogeneous perturbation terms
(when the integers $k_0,k_1,\dots$ form a strictly increasing sequence; in
all other cases the structure maps are defined tautologically). The
compactly supported perturbations are used to achieve transversality, and
are chosen in a consistent manner (cf.\ e.g.\ \cite{SeBook}).  Since we 
work over the Novikov field, the count of solutions in each
homotopy class is weighted by (topological) energy (as well as the bulk
deformation class, and holonomies of local systems along the boundary of
the disc when applicable). 

The key compactness property required for this construction is
a direct consequence of the maximum principle:

\begin{lemma}\label{l:holom_discs_stay_bounded}
Given any sequence of Lagrangians $L_0,\dots,L_r\in \bfL$ and integers $k_0<k_1<\dots<k_r$, 
there exists a bounded subset of\/ $Y$ which contains the images of all 
$J$-holomorphic discs with boundary on
$L_0(-\epsilon k_0)\cup\dots\cup L_r(-\epsilon k_r)$. 
The same property also holds in the presence of a compactly supported inhomogeneous
perturbation. 
\end{lemma}

\begin{proof}
This follows immediately from Propositions \ref{prop:maxprinciple_base} and 
\ref{prop:maxprinciple_fiber},
in the special case where the Lagrangian boundary condition remains 
constant along each component of $\partial \Sigma$ and there are no 
perturbation terms.
\end{proof}

\begin{remark}\label{rmk:genericJtransverse}
Disc bubbling is excluded by assumption \eqref{eq:no_hol_discs}, but sphere
bubbling can happen in our setting, so the regularity of the moduli spaces we
consider is not immediate. 

To deal with sphere bubbling, we assume that $J$ is chosen
generically within a suitable class of compatible almost-complex structures,
so that simple $J$-holomorphic spheres are regular, and evaluation maps at
interior points for somewhere injective $J$-holomorphic curves
are mutually transverse (see \cite[Theorem 3.4.1]{mcduff-salamon} for
the closed case; the argument works similarly for discs). For
our main example the standard complex structure is not regular, but all
holomorphic spheres lie inside $W^{-1}(0)$, so it is enough to perturb $J$
in a neighborhood of $W^{-1}(0)$ (or, in fact, its intersection with the
bounded subset provided by Lemma \ref{l:holom_discs_stay_bounded}, so that
the conditions we have set in Section \ref{sec:LGmodel-setup} on the
geometry at infinity are not affected).

With this understood, bubbling of simple $J$-holomorphic spheres is a
real codimension 2 phenomenon, and does not affect our ability to count
solutions to Floer's equations in zero-dimensional moduli spaces, or to
compare counts of solutions by considering one-dimensional moduli spaces.
Moreover, since $c_1(Y)=0$ we need not worry about multiply covered sphere 
bubbles either. Indeed, regularity for simple spheres implies that
for generic $J$ the union of the images 
of all pseudo-holomorphic spheres in $Y$ has real codimension 4.
By transversality of evaluation maps, it is therefore disjoint from the
images of holomorphic discs (or solutions to Floer's equation)
in (a fixed countable collection of)
zero- or one-dimensional moduli spaces.
\end{remark}

\subsection{Quasi-units and continuation maps}\label{ss:quasiunits}

The next ingredient in the construction of the fiberwise
wrapped category $\Wrap(Y, W)$ is a distinguished
collection of morphisms
\begin{equation}
e_{L^k} \in HF^{0}(L^{k}, L^{k+1})
\end{equation}
for all $L\in \bfL$ and $k\in \Z$, called {\em quasi-units}.

The quasi-unit $e_{L^k}$ is the image of the
identity in $H^0(L)$ under a PSS-type homomorphism
from $H^*(L)$ to $HF(L^k,L^{k+1})$ which can be constructed exactly as in 
\cite{Albers} (see below for the specific case at hand); note
however that the reverse map from $HF(L^k,L^{k+1})$ to $H^*(L)$ is not well-defined
in our setting, as it involves Floer data for which the analytic estimates of \S
\ref{sec:maximum-principle} do not hold. (Nonetheless, given that our
Lagrangians do not bound any holomorphic discs, the PSS map often turns out 
to be an isomorphism for small enough $\epsilon$, under additional geometric
assumptions which ensure that $L^{k+1}$ is contained within a 
Weinstein tubular neighborhood of $L^k$; this is e.g.\ the case in our main
example, by Proposition \ref{prop:CFL0explicit}.)

Chain-level quasi-units can be constructed by counting solutions
to a Cauchy-Riemann equation with moving boundary condition,
whose domain $\Sigma$ is a disc with a single boundary puncture which we consider
as an output, and where the boundary condition $\Lambda$ is given by the isotopy
$L^{t}=L(-\epsilon t)$, $t\in [k,k+1]$ (parametrized using some choice of
monotonically increasing smooth function from $\partial \Sigma$ to 
$[k,k+1]$ which is constant near the ends). Since the isotopy along $\partial\Sigma$
moves the complex plane in the clockwise direction and wraps fiberwise in
the negative direction only, we can apply the results of Section
\ref{sec:maximum-principle}, with $\alpha\equiv 0$ and $\tau\equiv 0$,
to control the behavior of solutions. We denote again by 
\begin{equation}
e_{L^k}\in CF^0(L(-\epsilon k), L(-\epsilon(k+1)))=\O(L^k,L^{k+1})
\end{equation}
the chain-level quasi-unit constructed in this manner. While
$e_{L^k}$ depends on auxiliary 
choices (e.g., of a function from $\partial \Sigma$ to 
$[k,k+1]$), the chain-level quasi-units constructed using different choices
only differ by an explicit homotopy, and can be used interchangeably.

Let $Z$ denote the collection of all such morphisms. The fiberwise
wrapped category $\Wrap(Y, W)$ is the localisation of $\O$ with respect 
to these morphisms (i.e., the quotient of $\O$ by the cones of the morphisms in $Z$,
in the sense of Lyubashenko-Ovsienko \cite{Lyubashenko}; see
also \cite[\S 3.1.3]{GPS1}, as well as Section \ref{ss:colimits} below):
\begin{equation}
  \Wrap(Y,W) := Z^{-1} \O.
\end{equation}

We shall use a concrete model of the morphisms in $\Wrap(Y,W)$, introduced in the next
section, in which they are expressed as homotopy
colimits (i.e., direct limits) of morphism spaces in $\O$. In order to compute these morphism spaces explicitly in terms of Floer theory, we shall introduce {\em continuation maps}
\begin{equation} \label{eq:continuation_map}
F_{L_0^k,L_1^j}: \O(L_0^k,L_1^j)\to \O(L_0^{k+1},L_1^{j+1}).
\end{equation}
These are defined by counting solutions to a perturbed
Cauchy-Riemann equation, with domain $\Sigma=\R\times [0,1]$, and
where the boundary conditions are given by $\Lambda_{s,0}=L_0^{
k+\chi(s)}=L_0(-\epsilon(k+\chi(s)))$ along $\R\times \{0\}$ and 
$\Lambda_{s,1}=L_1^{j+ \chi(s)}=L_1(-\epsilon(j+\chi(s)))$ along $\R\times \{1\}$. Here $\chi:\R\to [0,1]$ is a
monotonically decreasing smooth function, constant near the ends, so that
the boundary conditions are $(L_0^{k}, L_1^{j})$
at the input end $s\to +\infty$, and 
$(L_0^{k+1}, L_1^{j+1})$ at the output end
$s\to -\infty$.

We use the setup of Section \ref{sec:maximum-principle}, with a fiberwise
wrapping perturbation given by $\alpha=-\epsilon \chi'(s)\,ds$ (so that
$d\alpha=0$ and $\alpha_{|\partial\Sigma}=\eta$), and a horizontal
perturbation given by the autonomous flow $\rho$ and $\tau=\epsilon \chi(s)$
(so $\rho^\tau$ exactly cancels the horizontal isotopy of the boundary
condition).  Propositions \ref{prop:maxprinciple_base},
\ref{prop:maxprinciple_fiber}, and \ref{prop:energy} then imply that the
counts of index 0 solutions to \eqref{eq:pseudo-holomorphic_equation} (weighted by topological energy) can be used
to define $F_{L_0^k,L_1^j}$.

Despite the slight differences in 
technical setup, these continuation maps have all the usual properties of continuation maps
associated to symplectic isotopies in Lagrangian Floer theory: they are quasi-isomorphisms, and
extend to an $A_\infty$-functor $F:\O\to \O$ which acts on objects
by $L^k\mapsto L^{k+1}$. Since we shall not need these properties, we omit the proofs.

%
%

\begin{lemma}\label{l:quasiunit_natural}
The quasi-units are natural with respect to continuation maps, in the sense
that both triangles in the diagram

\begin{equation}\label{eq:quasiunit_natural}
\begin{psmatrix}[colsep=2cm,rowsep=2cm]
\O(L_0^{k+1},L_1^j) & \O(L_0^{k+1},L_1^{j+1}) \\
\O(L_0^k,L_1^j) &  \O(L_0^k,L_1^{j+1})
\end{psmatrix}
\psset{nodesep=5pt,arrows=->}
\ncline{2,1}{1,2} \tbput[tpos=0.7]{F_{L_0^k,L_1^j}}
\ncline{1,1}{2,1} \tlput{\mu^2(\cdot,e_{L_0^k})}
\ncline{1,1}{1,2} \taput{\mu^2(e_{L_1^j},\cdot)}
\ncline{2,1}{2,2} \tbput{\mu^2(e_{L_1^j},\cdot)}
\ncline{1,2}{2,2} \trput{\mu^2(\cdot,e_{L_0^k})}
\end{equation}\vskip3mm

\noindent are commutative up to homotopy.
\end{lemma}

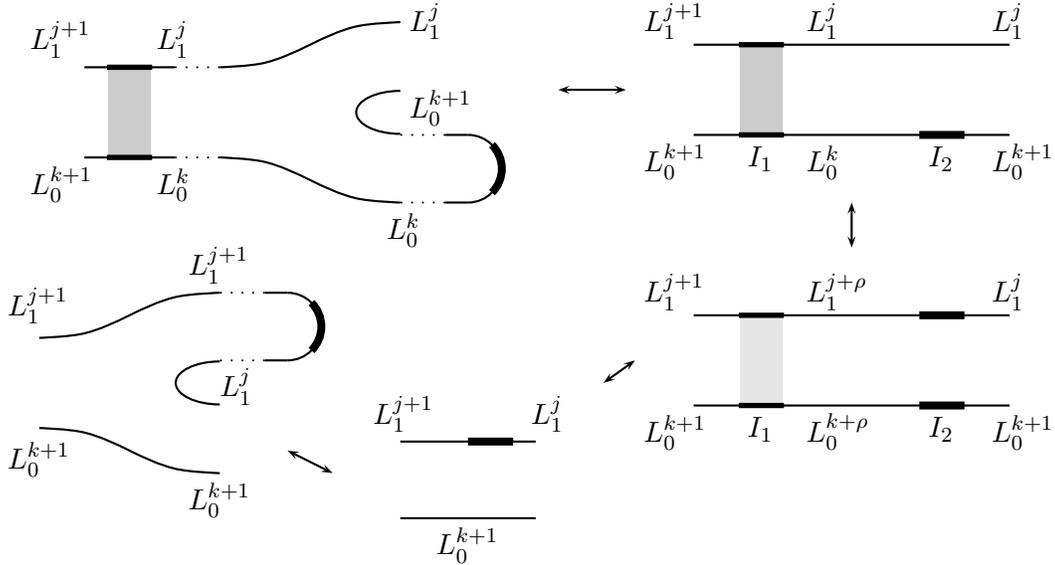
\begin{figure}[t]
\setlength{\unitlength}{6mm}
\begin{picture}(23,12)(-0.5,-1.5)
\newgray{litegray}{0.8}
\newgray{vlitegray}{0.9}
\psset{unit=\unitlength}
\psline(1,9)(3,9) 
\psline[linestyle=dotted](3,9)(4,9)
\psline(1,7)(3,7)
\psframe[fillstyle=solid,fillcolor=litegray,linestyle=none](1.5,7)(2.5,9)
\psline[linewidth=2pt](1.5,7)(2.5,7)
\psline[linewidth=2pt](1.5,9)(2.5,9)
\psline[linestyle=dotted](3,7)(4,7)
\pscurve(4,9)(5,9.1)(7,9.9)(8,10)
\pscurve(4,7)(5,6.9)(7,6.1)(8,6)
\psellipticarc(8,8)(1,0.5){90}{270}
\psline[linestyle=dotted](8,6)(9,6)
\psline[linestyle=dotted](8,7.5)(9,7.5)
\psline(9,6)(9.5,6)
\psline(9,7.5)(9.5,7.5)
\psarc(9.5,6.75){0.75}{-90}{90}
\psarc[linewidth=3pt](9.5,6.75){0.75}{-45}{45}
\put(-0.2,9.4){\small $L_1^{j+1}$}
\put(2.6,9.4){\small $L_1^j$}
\put(-0.2,6.2){\small $L_0^{k+1}$}
\put(2.6,6.2){\small $L_0^k$}
\put(8.2,9.8){\small $L_1^j$}
\put(8.2,7.9){\small $L_0^{k+1}$}
\put(7.7,5.2){\small $L_0^k$}
\psline{<->}(11.5,8.5)(13,8.5)
\psline(14.5,7.5)(21.5,7.5)
\psline(14.5,9.5)(21.5,9.5)
\psframe[fillstyle=solid,fillcolor=litegray,linestyle=none](15.5,7.5)(16.5,9.5)
\psline[linewidth=2pt](15.5,7.5)(16.5,7.5)
\psline[linewidth=2pt](15.5,9.5)(16.5,9.5)
\psline[linewidth=3pt](19.5,7.5)(20.5,7.5)
\put(13.4,6.7){\small $L_0^{k+1}$}
\put(13.4,9.8){\small $L_1^{j+1}$}
\put(17,6.7){\small $L_0^k$}
\put(17,9.8){\small $L_1^j$}
\put(21.1,6.7){\small $L_0^{k+1}$}
\put(21.1,9.8){\small $L_1^j$}
\put(15.7,6.8){\small $I_1$}
\put(19.7,6.8){\small $I_2$}
\psline{<->}(18,6)(18,5)
\psline(14.5,1.5)(21.5,1.5)
\psline(14.5,3.5)(21.5,3.5)
\psframe[fillstyle=solid,fillcolor=vlitegray,linestyle=none](15.5,1.5)(16.5,3.5)
\psline[linewidth=2pt](15.5,1.5)(16.5,1.5)
\psline[linewidth=2pt](15.5,3.5)(16.5,3.5)
\psline[linewidth=3pt](19.5,1.5)(20.5,1.5)
\psline[linewidth=3pt](19.5,3.5)(20.5,3.5)
\put(13.4,0.7){\small $L_0^{k+1}$}
\put(13.4,3.8){\small $L_1^{j+1}$}
\put(17,0.7){\small $L_0^{k+\rho}$}
\put(17,3.8){\small $L_1^{j+\rho}$}
\put(21.1,0.7){\small $L_0^{k+1}$}
\put(21.1,3.8){\small $L_1^j$}
\put(15.7,0.8){\small $I_1$}
\put(19.7,0.8){\small $I_2$}
\psline{<->}(13.2,2.5)(12.5,2)
\psline(8,0.7)(11,0.7)
\psline(8,-1)(11,-1)
\psline[linewidth=3pt](9.5,0.7)(10.5,0.7)
\put(8.8,-1.8){\small $L_0^{k+1}$}
\put(7.4,1.1){\small $L_1^{j+1}$}
\put(10.9,1.1){\small $L_1^j$}
\pscurve(0,3)(1,3.1)(3,3.9)(4,4)
\pscurve(0,1)(1,0.9)(3,0.1)(4,0)
\psellipticarc(4,2)(1,0.5){90}{270}
\psline[linestyle=dotted](4,4)(5,4)
\psline[linestyle=dotted](4,2.5)(5,2.5)
\psline(5,4)(5.5,4)
\psline(5,2.5)(5.5,2.5)
\psarc(5.5,3.25){0.75}{-90}{90}
\psarc[linewidth=3pt](5.5,3.25){0.75}{-45}{45}
\put(-0.7,3.5){\small $L_1^{j+1}$}
\put(3.3,4.4){\small $L_1^{j+1}$}
\put(-0.7,0.1){\small $L_0^{k+1}$}
\put(3.3,-0.8){\small $L_0^{k+1}$}
\put(4,1.7){\small $L_1^j$}
\psline{<->}(6.5,0)(5.5,0.5)
\end{picture}
\caption{A homotopy between $\mu^2(e_{L_1^j},\cdot)$ and 
$F(\mu^2(\cdot,e_{L_0^k}))$.}\label{fig:quasiunit_natural1}
\end{figure}

\proof
We start with the upper triangle, i.e.\ the homotopy between 
$F(\mu^2(\cdot,e_{L_0^k}))$ and $\mu^2(e_{L_1^j},\cdot)$. The
argument relies on comparing a series of moduli spaces of perturbed
holomorphic curves, presented pictorially on Figure 
\ref{fig:quasiunit_natural1}, where the thick edges correspond to 
regions where the Lagrangian boundary condition is moving and
the shaded areas correspond to the support of the inhomogeneous
perturbation terms in~\eqref{eq:pseudo-holomorphic_equation}.

The main protagonists in the homotopy are a family of perturbed
holomorphic strips with domain $\Sigma=\R\times [0,1]$, depicted
on the right-hand side of Figure \ref{fig:quasiunit_natural1}. 
Fix two
disjoint compact intervals $I_1,I_2\subset \R$, with $I_1$ to the
left of $I_2$, as well as two smooth monotonic functions
$\chi_1,\chi_2:\R\to [0,1]$, such that
$\chi_1$ equals 1 to the left of $I_1$ and 0 to its right, while
$\chi_2$ equals 0 to the left of $I_2$ and 1 to its right; we arrange
that the ``profiles'' of these functions are identical to those used
in the construction of the continuation maps and quasi-units.
Also fix a parameter $\rho\in [0,1]$, and define 
\begin{align*}
k_\rho(s)&=k+\rho+(1-\rho)\chi_1(s)+(1-\rho)\chi_2(s),\\
j_\rho(s)&=j+\rho+(1-\rho)\chi_1(s)-\rho\chi_2(s).
\end{align*}
Along $\R\times \{0\}$ we consider
the moving boundary condition $\Lambda_{0,s}=L_0^{k_\rho(s)}$, 
while along $\R\times \{1\}$ we use
$\Lambda_{1,s}=L_1^{j_\rho(s)}$.
While the boundary condition $\Lambda_{1,s}$ always moves in the negative
direction as $s$ decreases ($j_\rho$ is a monotonic function of $s$), 
the boundary condition $\Lambda_{0,s}$ moves in the positive direction
over $I_1$.  Accordingly, we set $\alpha=-\epsilon (1-\rho)
\chi'_1(s)\,ds$, and $\tau=\epsilon (1-\rho) \chi_1(s)$, for the perturbation
terms in \eqref{eq:pseudo-holomorphic_equation}. 

By Section \ref{sec:maximum-principle} the solutions to \eqref{eq:pseudo-holomorphic_equation} with these boundary
conditions and perturbations satisfy maximum principles and energy
estimates, so we can define operations 
$$\Phi_{I_1,I_2,\rho}:\O(L_0^{k+1},L_1^j)\to
\O(L_0^{k+1},L_1^{j+1})$$ by counting
rigid (index 0) solutions. 
These operations are chain maps, since
the ends of the moduli spaces of index 1 solutions for fixed $I_1,I_2,\rho$ 
are in bijection with the broken trajectories which contribute to
$\partial\circ \Phi_{I_1,I_2,\rho}$ and $\Phi_{I_1,I_2,\rho}\circ \partial$; and they
are all homotopic to each other, with explicit homotopies given by counts of index $-1$
solutions that may arise as the parameters $I_1,I_2,\rho$ vary, as can be
seen by considering the ends of moduli spaces of index 0 solutions for a
one-parameter family of choices of $I_1,I_2,\rho$. (These are standard
arguments in Lagrangian Floer theory, so we omit the details; see 
e.g.\ \cite[\S 2]{AuBeginner}, \cite[\S 17]{SeBook}, \cite{Albers}, etc.
for similar proofs.)

For $\rho=0$, the boundary conditions and perturbations near $I_1$ are
identical to those used to define the continuation map, while along
$I_2\times \{0\}$
the boundary condition $\Lambda_{0,s}$ varies from $L_0^k$ to $L_0^{k+1}$
(top-right diagram in Figure \ref{fig:quasiunit_natural1}).
Moving $I_1$ towards $-\infty$ and shrinking $I_2$ to a point then causes
the solutions to converge to limit configurations consisting of (typically)
three components (upper-left diagram in Figure
\ref{fig:quasiunit_natural1}). 
The ``main'' component is an unperturbed holomorphic disc 
with two inputs, corresponding to the Floer product $\mu^2$, while at
$s=-\infty$ we have a strip with moving boundary conditions and
inhomogeneous perturbations, corresponding to the continuation map $F$, 
and the rescaling limit near $I_2\times \{0\}$ gives a half-plane 
with a moving boundary condition which corresponds to the quasi-unit.
Thus, the operations $\Phi_{I_1,I_2,\rho}$ are homotopic to $F(\mu^2(\cdot,
e_{L_0^k}))$.

On the other hand, for $\rho=1$, there are no perturbations near $I_1$,
and along $I_2\times \{1\}$ the boundary condition $\Lambda_{1,s}$ varies 
from $L_1^j$ to $L_1^{j+1}$. Shrinking $I_2$ to a point then causes a holomorphic
half-plane with moving boundary condition to break off (lower-left diagram
in Figure \ref{fig:quasiunit_natural1}), showing that
$\Phi_{I_1,I_2,\rho}$ is also homotopic to $\mu^2(e_{L_1^j},\cdot)$.

\begin{figure}[t]
\setlength{\unitlength}{6mm}
\begin{picture}(24,6)(0.5,-1)
\newgray{litegray}{0.8}
\newgray{vlitegray}{0.9}
\psset{unit=\unitlength}
\pscurve(19,3)(19.8,3.1)(21.2,3.9)(22,4)
\pscurve(19,1)(19.8,0.9)(21.2,0.1)(22,0)
\psellipticarc(22,2)(1,0.5){90}{270}
\psline[linestyle=dotted](22,4)(23,4)
\psline[linestyle=dotted](22,2.5)(23,2.5)
\psline[linestyle=dotted](22,0)(23,0)
\psline[linestyle=dotted](22,1.5)(23,1.5)
\psline(23,0)(23.5,0)
\psline(23,1.5)(23.5,1.5)
\psline(23,4)(25,4)
\psline(23,2.5)(25,2.5)
\psframe[fillstyle=solid,fillcolor=litegray,linestyle=none](23.5,2.5)(24.5,4)
\psline[linewidth=2pt](23.5,2.5)(24.5,2.5)
\psline[linewidth=2pt](23.5,4)(24.5,4)
\psarc(23.5,0.75){0.75}{-90}{90}
\psarc[linewidth=3pt](23.5,0.75){0.75}{-45}{45}
\put(18.1,3.45){\small $L_1^{j+1}$}
\put(18.1,0.2){\small $L_0^k$}
\put(21.8,4.35){\small $L_1^{j+1}$}
\put(24.5,4.35){\small $L_1^j$}
\put(24.5,1.8){\small $L_0^k$}
\put(21.8,1.8){\small $L_0^{k+1}$}
\put(22,-0.7){\small $L_0^k$}
\psline{<->}(16.5,2)(17.5,2)
\psline(8.5,1)(15.5,1)
\psline(8.5,3)(15.5,3)
\psframe[fillstyle=solid,fillcolor=vlitegray,linestyle=none](13.5,1)(14.5,3)
\psline[linewidth=3pt](9.5,1)(10.5,1)
\psline[linewidth=3pt](9.5,3)(10.5,3)
\psline[linewidth=2pt](13.5,1)(14.5,1)
\psline[linewidth=2pt](13.5,3)(14.5,3)
\put(7.6,0.2){\small $L_0^k$}
\put(7.6,3.35){\small $L_1^{j+1}$}
\put(11,0.2){\small $L_0^{k+\rho}$}
\put(11,3.35){\small $L_1^{j+\rho}$}
\put(15.1,0.2){\small $L_0^k$}
\put(15.1,3.35){\small $L_1^j$}
\pscurve(0,3)(0.8,3.1)(2.2,3.9)(3,4)
\pscurve(0,1)(0.8,0.9)(2.2,0.1)(3,0)
\psellipticarc(3,2)(1,0.5){90}{270}
\psline[linestyle=dotted](3,4)(4,4)
\psline[linestyle=dotted](3,2.5)(4,2.5)
\psline(4,4)(4.5,4)
\psline(4,2.5)(4.5,2.5)
\psarc(4.5,3.25){0.75}{-90}{90}
\psarc[linewidth=3pt](4.5,3.25){0.75}{-45}{45}
\put(-0.7,3.5){\small $L_1^{j+1}$}
\put(2.5,4.4){\small $L_1^{j+1}$}
\put(-0.7,0.1){\small $L_0^k$}
\put(2.5,-0.8){\small $L_0^k$}
\put(3,1.7){\small $L_1^j$}
\psline{<->}(6.2,2)(7.2,2)
\end{picture}
\caption{A homotopy between $\mu^2(e_{L_1^j},\cdot)$ and
$\mu^2(F(\cdot),e_{L_0^k})$.}\label{fig:quasiunit_natural2}
\end{figure}
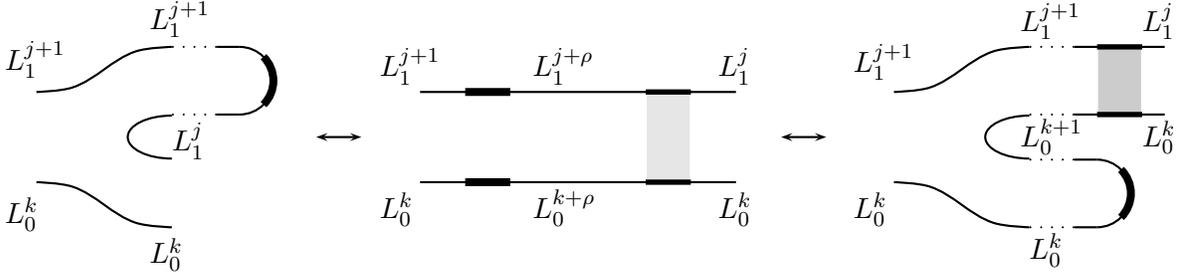

The commutativity up to homotopy of the lower triangle in 
\eqref{eq:quasiunit_natural} is proved in exactly the same manner, by
considering a family of perturbed holomorphic strips depicted 
in Figure \ref{fig:quasiunit_natural2}.  The construction is identical,
except that the roles of the two regions $I_1$ and $I_2$ are now reversed.
By considering the limit configurations as $\rho$ becomes 0 or 1 while
the left-most interval is degenerated to a point and the right-most interval
escapes towards $s=+\infty$, this yields a homotopy between
$\mu^2(e_{L_1^j},\cdot)$ (Figure \ref{fig:quasiunit_natural2} left)
and $\mu^2(F(\cdot),e_{L_0^k})$ (Figure \ref{fig:quasiunit_natural2} right).
\endproof

\begin{remark}
Lemma \ref{l:quasiunit_natural} can be strengthened to show that the
quasi-units form the leading order term of a natural transformation $e$ from the
identity to the $A_\infty$-functor $F$. The next (order 1) term in the
natural transformation is precisely the homotopy between $\mu^2(F(\cdot),
e_{L_0^k})$ and $\mu^2(e_{L_1^j},\cdot)$ that arises
in the proof of Lemma \ref{l:quasiunit_natural}, i.e.\ it can be defined
by counting index $-1$ solutions that come up in the family of perturbed
Cauchy-Riemann equations depicted in Figure \ref{fig:quasiunit_natural2}.
The construction of the higher order terms of the natural transformation
is technically more involved, and we do not discuss it here since we will
not be needing it.

Since the localization at all quasi-units amounts to making the natural
transformation $e$ invertible up to homotopy, the localized category
$\Wrap(Y,W)$ is also sometimes denoted $\O[e^{-1}]$; this notation
is also suggestive of the fact that the localization effectively enlarges
morphism spaces by inverting all quasi-units (up to homotopy).
\end{remark}

\begin{remark}
It is natural to ask to what extent the category $\Wrap(Y,W)$
depends on the choice of the collection of Lagrangians $\mathbf{L}$ and on
the parameter $\epsilon$ (the time step with respect to which we consider quasi-units). 
Here we do not address the first question,
which relates to the existence of generation criteria for $\Wrap(Y,W)$;
we simply assume that we have a collection $\mathbf{L}$ satisfying 
the required hypotheses, and if this collection is too small the category we
construct might only be a subcategory of the one we would obtain
from a larger collection of objects. 

On the other hand, the algebraic
properties of quasi-units imply that
the choice of the parameter $\epsilon$ does not affect the outcome
of our construction. The key observation is that we can define quasi-units
$e_{L(t')\to L(t)}\in HF^0(L(t'),L(t))$ for all $t'>t$ such that $L(t')\cap
L(t)$ is contained in a compact subset of $Y$, and an argument similar to
the proof of Lemma \ref{l:quasiunit_natural} shows that, for $t''>t'>t$, these
satisfy $$e_{L(t')\to L(t)}\cdot e_{L(t'')\to L(t')}=
e_{L(t'')\to L(t)}$$ (in cohomology, or up to homotopy).
Assume that $L(t'),L(t)$ are both objects of $\O$ for some $t'>t$, and let 
$n$ be such that $n\epsilon>t'-t$. Since $e_{L(t+n\epsilon)\to L(t)}$ is the
product of the quasi-units $e_{L(t+k\epsilon)\to L(t+(k-1)\epsilon)}$ for
$1\le k\le n$, it is a quasi-isomorphism in the localized category, hence
admits a quasi-inverse $f_{L(t)\to L(t+n\epsilon)}$; similarly for
$e_{L(t')\to L(t'-n\epsilon)}$, whose quasi-inverse we denote by
$f_{L(t'-n\epsilon)\to L(t')}$.
Then in $H^0\Wrap$ we have
\begin{eqnarray*}
e_{L(t')\to L(t)}\,\cdot\, (e_{L(t+n\epsilon)\to L(t')}\cdot f_{L(t)\to
L(t+n\epsilon)})&=&\id_{L(t)}\qquad
\text{and}\\
(f_{L(t'-n\epsilon)\to L(t')}\cdot e_{L(t)\to L(t'-n\epsilon)})\,\cdot\,
e_{L(t')\to L(t)}&=&\id_{L(t')},
\end{eqnarray*}
giving left and right inverses for $e_{L(t')\to L(t)}$ up to homotopy
and proving that it is a quasi-isomorphism.
Hence, localizing with respect to quasi-units for a fixed step size
$\epsilon$ actually inverts {\em all} quasi-units; and
$L(t)$ and $L(t')$
are quasi-isomorphic in the localized category whenever
they belong to the set of objects.
This implies that up to quasi-equivalence the category we construct does not depend on the
choice of $\epsilon$.
\end{remark}

\subsection{Fiberwise wrapped category via colimits}\label{ss:colimits}
Our goal in this section is to construct the fiberwise wrapped Fukaya category as a subcategory of the category of modules over $\O$. This approach is adapted from unpublished work \cite{A-S} of the first author with Seidel, where the starting point is the more abstract formalism of localisation of categories, and the point of view which we take here is used as a computational tool.

The basic idea is that we seek an $A_\infty$-category where morphism spaces between Lagrangians are taken after passing to a limit with respect to positive wrapping. We implement this by assigning to each Lagrangian $L$ an object of the category of modules over $\O$ given as a homotopy colimit
(or direct limit)
\begin{equation}
\cY_{L^\infty} \equiv  \hocolim_{k \to + \infty} \cY_{L^k}  
\end{equation}
where $\cY_{L^k}$ is the Yoneda module
\begin{equation}
 X \mapsto \O(X,L^k)
\end{equation}
and the connecting maps $\cY_{L^k}\to \cY_{L^{k+1}}$ are given by
composition with the quasi-units $e_{L^k}$.
We take as model for the homotopy colimit the mapping telescope
\begin{equation}
  \mathrm{Cone}\left( \bigoplus_{k =0}^{\infty} \cY_{L^k} \to  \bigoplus_{k =0}^{\infty} \cY_{L^k}\right)
\end{equation}
where the arrow is the direct sum of the differences $ \id - e_{L^k}$.

We write $\Wrap$ for the full subcategory of modules over $\O$ with these objects, i.e. objects are admissible 
Lagrangians in $\mathbf{L}$, and morphisms between $L_0$ and $L_1$ given by
\begin{equation}
  \Wrap(L_0,L_1) \equiv   \Hom_{\O}( \cY_{L^\infty_0} , \cY_{L_1^\infty} )
\end{equation}

The first computation we need is:
\begin{lemma}
  There is a natural quasi-isomorphism
  \begin{equation} \label{eq:morphisms-wrapped-inverse-limit-of-direct-limit}
    \holim_{k \to \infty} \hocolim_{j \to \infty} \O(L_0^k, L_1^j) \to   \Wrap(L_0,L_1).
  \end{equation}
\end{lemma}
\begin{proof}
The cone of the complex
  \begin{equation}
 \Hom_{\O}(\bigoplus_{k =0}^{\infty} \cY_{L_0^k}, \cY_{L_1^\infty} ) \to  \Hom_{\O}( \bigoplus_{k =0}^{\infty} \cY_{L_0^k},\cY_{L_1^\infty}),
\end{equation}
maps quasi-isomorphically to the space of morphisms from $\cY_{L_0^\infty}$ to $ \cY_{L_1^\infty} $, and is isomorphic to the cone of the map
\begin{equation}
  \prod_{k =0}^{\infty}  \Hom_{\O}( \cY_{L_0^k}, \cY_{L_1^\infty} ) \to    \prod_{k =0}^{\infty} \Hom_{\O}( \cY_{L_0^k},\cY_{L_1^\infty}), 
\end{equation}
which is a model for
\begin{equation}
  \holim_{k} \Hom_{\O}( \cY_{L_0^k}, \cY_{L_1^\infty} ).  
\end{equation}
On the other hand, the Yoneda map induces a quasi-isomorphism
  \begin{equation}
   \hocolim_{j \to \infty} \O(L_0^k, L_1^j) \to
   \hocolim_{j\to\infty} \Hom_{\O}(\cY_{L_0^k},\cY_{L_1^j})
    \cong \Hom_{\O}( \cY_{L_0^k}, \cY_{L_1^\infty} ).  
 \end{equation}
 The desired map follows by composition.
\end{proof}

The next result reduces the computation of morphisms in $\Wrap$ to a direct limit:
\begin{lemma} \label{lem:colimit-is-local-module}
  For all $L_0$, $L_1$ and $k$, the map
  \begin{equation}
     \Hom_{\O}( \cY_{L_0^{k+1}}, \cY_{L_1^\infty} ) \to   \Hom_{\O}( \cY_{L_0^k}, \cY_{L_1^\infty} )   
  \end{equation}
induced by multiplication by quasi-units is a quasi-isomorphism.
\end{lemma}
\begin{proof}
  The Yoneda Lemma reduces the problem to the statement that the map
  \begin{equation}
    \hocolim_{j \to \infty} \O(L_0^{k+1}, L_1^j) \to  \hocolim_{j \to \infty} \O(L_0^k, L_1^j)
  \end{equation}
  induces an isomorphism on cohomology. Since direct limits commute with passing to cohomology, it suffices to show that the map of cohomology groups
  \begin{equation} \label{eq:colimit_cohomology_k+1-to-k}
    \colim_{j \to \infty} HF^*(L_0^{k+1}, L_1^j) \to  \colim_{j \to \infty} HF^*(L_0^k, L_1^j)
  \end{equation}
  is an isomorphism, where we use the fact that the morphisms in $\O$ are given by Floer cochains whenever $j$ is sufficiently large. We claim that the continuation maps from Equation \eqref{eq:continuation_map} provide an inverse. Indeed, by taking the cohomology of Diagram \eqref{eq:quasiunit_natural} we obtain a commutative diagram
  \vskip3mm
  \begin{equation}
\begin{psmatrix}[colsep=2cm,rowsep=2cm]
HF^*(L_0^{k+1},L_1^j) & HF^*(L_0^{k+1},L_1^{j+1}) \\
HF^*(L_0^k,L_1^j) &  HF^*(L_0^k,L_1^{j+1}).
\end{psmatrix}
\psset{nodesep=5pt,arrows=->}
\ncline{2,1}{1,2} \tbput[tpos=0.7]{F_{L_0^k,L_1^j}}
\ncline{1,1}{2,1} \tlput{\mu^2(\cdot,e_{L_0^k})}
\ncline{1,1}{1,2} \taput{\mu^2(e_{L_1^j},\cdot)}
\ncline{2,1}{2,2} \tbput{\mu^2(e_{L_1^j},\cdot)}
\ncline{1,2}{2,2} \trput{\mu^2(\cdot,e_{L_0^k})}
\end{equation}
\vskip3mm
In this diagram the horizontal maps are those used to define the
direct limits, while the vertical maps assemble into the map \eqref{eq:colimit_cohomology_k+1-to-k}.

To show that \eqref{eq:colimit_cohomology_k+1-to-k} is injective, note that every element of the left hand side is represented by an element of $HF^*(L_0^{k+1},L_1^j)$ for some $j$. The above diagram implies that the image of this element in $ HF^*(L_0^{k+1},L_1^{j+1})$ agrees with the image under 
our proposed inverse (the continuation map $F_{L_0^k,L_1^j}$) 
of its image under the map of direct limits \eqref{eq:colimit_cohomology_k+1-to-k}.
By definition of the direct limit, this implies that the continuation map
is a left inverse to \eqref{eq:colimit_cohomology_k+1-to-k}, and 
injectivity follows. 

Considering the composition in the other order yields surjectivity: every
element of the right hand side of \eqref{eq:colimit_cohomology_k+1-to-k} is
represented by an element of $HF^*(L_0^k,L_1^j)$ for some~$j$, whose image
in $HF^*(L_0^k, L_1^{j+1})$ is also the image under \eqref{eq:colimit_cohomology_k+1-to-k} 
of its image under the continuation map, so the continuation map is a
right inverse.
\end{proof}

\begin{corollary}\label{cor:colimit}
  For each pair $L_0^k$ and $L_1$ of objects of $\O$, there is a natural isomorphism
  \begin{equation}
    \colim_{j\to \infty}    HF^*(L_0^k, L_1^j) \to H\Wrap(L_0,L_1).   
  \end{equation}
\end{corollary}
\begin{proof}
  The above Lemma implies that bonding maps in the inverse system appearing in Equation \eqref{eq:morphisms-wrapped-inverse-limit-of-direct-limit} are quasi-isomorphisms. In particular, the Mittag-Leffler condition  is satisfied,%
\footnote{An inverse system $A_1\leftarrow
A_2\leftarrow A_3\leftarrow \dots$ is said to satisfy the Mittag-Leffler condition
if for each $k$, there exists $j>k$ such that, for all $i>j$,
the maps $A_i\to A_{k}$ and $A_j\to A_k$ have the same image;
this condition implies vanishing of the first derived functor of the
inverse limit, and that inverse limits are well-behaved with respect to
cohomology (see e.g.\ \cite[Definition 3.5.6]{Weibel}).}
 and for each integer $k$ the projection map 
  \begin{equation}
       \holim_{k \to \infty} \hocolim_{j \to \infty} \O(L_0^k, L_1^j) \to   \hocolim_{j \to \infty} \O(L_0^k, L_1^j)
  \end{equation}
  induces an isomorphism on cohomology. Inverting this map, and  composing with the one induced by Equation \eqref{eq:morphisms-wrapped-inverse-limit-of-direct-limit} on cohomology yields the desired isomorphism.
\end{proof}
\begin{remark}
The most straightforward way to compare our construction with the approach of \cite{A-S} is to consider the localisation functor from $\O$-modules  to $\O[e^{-1}]$- modules. By the universal property of localisation, the images of the Yoneda objects $L^k$ are equivalent, hence the image of the colimit $\cY_{L^\infty}$ under localisation is equivalent to these Yoneda modules. Lemma \ref{lem:colimit-is-local-module} can be restated as the fact that the modules $ \cY_{L^\infty} $ lie in the $e$-local subcategory of $\O$-modules, which is quasi-isomorphic to the category of $\O[e^{-1}]$-modules. We therefore conclude that the category generated by the modules $\cY_{L^\infty} $ is equivalent to the localisation of $\O$, which is the point of view taken by \cite{A-S}.
\end{remark}

\section{K\"ahler forms and admissibility} \label{s:toric}

In this section, we study the geometry of parallel transport in toric
Landau-Ginzburg models, and construct suitable K\"ahler forms for which 
{\em fiberwise monomial admissibility} is preserved by
parallel transport; we then show that the technical assumptions
we have made in the previous Section follow from this property.

\begin{definition}\label{def:admissibility}
A {\em fiberwise monomial subdivision} for the toric Landau-Ginzburg model\/
$W:Y\to\C$ consists of a finite collection of toric monomials
$z^{\mathbf{v}}$, $\mathbf{v}\in \mathcal{V}\subset \Z^{n+1}$,
weights $d(\mathbf{v})\in \Z_{>0}$, open subsets $C_{\mathbf{v}}\subset
Y$, and a closed subset $\Omega\subset Y$, such that:
\begin{enumerate}
\item $z^{\mathbf{v}}\in \O(Y)$ for all $\mathbf{v}\in \mathcal{V}$, and 
$z\mapsto (z^{\mathbf{v}})_{\mathbf{v}\in \mathcal{V}}$ defines a 
proper map $Y\to \C^{|\mathcal{V}|}$;
\item the restriction of $W$ to $\Omega$ is a proper map;
\item $\Omega \cup \bigcup_{\mathbf{v}\in \mathcal{V}} C_{\mathbf{v}}=Y$;
\item for $z\in Y\setminus \Omega$, if 
$|z^{\mathbf{v}_0}|^{1/d(\mathbf{v}_0)}=
\max\{|z^{\mathbf{v}}|^{1/d(\mathbf{v})},\ \mathbf{v}\in \mathcal{V}\}$
then $z\in C_{\mathbf{v}_0}$.
\end{enumerate}
\end{definition}

\begin{definition} \label{def:flat_at_infinity} 
Given a fiberwise monomial subdivision,
a Lagrangian submanifold $\ell \subset W^{-1}(c)\cong (\C^*)^n$ is {\em monomially
admissible} with phase angles $\{\varphi_{\mathbf{v}}, \mathbf{v}\in
\mathcal{V}\}$ if, outside of the compact subset $W^{-1}(c)\cap \Omega$,
$\arg(z^{\mathbf{v}})=\varphi_{\mathbf{v}}$ at
every point of $\ell \cap C_{\mathbf{v}}$.

A Lagrangian submanifold $L\subset Y$ is {\em fiberwise monomially admissible}
with phase angles $\{\varphi_{\v}\}$ if, outside of $\Omega$, 
$\arg(z^{\v})=\varphi_{\v}$ at every point of $L\cap C_\v$.
\end{definition}

\begin{example}
We can define a fiberwise monomial subdivision for the toric
Landau-Ginzburg model $(\C^N, W_0=-\prod z_j)$ as follows (the 
construction below will be a slight modification of this example). Take the collection of 
monomials to be the coordinate
functions $z_j$, $1\leq j\leq N$ (i.e., the exponent vectors $\mathbf{v}_j$ are
the standard basis of $\Z^N$); take $d(\v_j)=1$ for all $j$,
and let $C_{\v_j}$ be the set of points of $\C^N$ where $|z_j|>
K\,|W_0|^{1/N}$ for some constant $K>1$, and 
$\Omega=\C^N\setminus \bigcup C_{\v_j}=\{z\in \C^N\,|\,\max(|z_j|)\leq K\,|W_0|^{1/N}\}$.
Condition~(2) 
holds since the coordinates of points of $\Omega$ are bounded by $K\,|W_0|^{1/N}$,
and condition (4) holds since if $|z_j|=\max(|z_1|,\dots,|z_N|)>K\,|W_0|^{1/N}$
then $z\in C_{\v_j}$.
A Lagrangian submanifold $L\subset \C^N$ is then fiberwise monomially admissible
with phase angles $\varphi_1,\dots,\varphi_N$ if, at every point of $L$
where $|z_j|>K|W_0|^{1/N}$, $\arg(z_j)=\varphi_j$. For instance, the real
positive locus $(\R_+)^N$ satisfies this condition with all phase angles
equal to zero. We shall see below how to build more interesting examples 
under the assumption that the toric K\"ahler form on $\C^N$ 
is chosen suitably; see Section \ref{s:toricCN}.
\end{example}

The notions of monomial subdivision and monomial admissibility for
Lagrangians in $(\C^*)^n$ already appear in Andrew Hanlon's thesis
\cite{Hanlon}. One technical difference is that we consider a fiberwise
version of monomial admissibility and its compatibility with
parallel transport between the fibers of $W$. The more important difference
is philosophical: we use monomial admissibility merely as a technical tool
to ensure the flatness condition of Definition \ref{def:admissibleLagrangian}~(ii),
rather than as a geometric way of restricting the
fiberwise wrapping by introducing additional stops (though we will do so in the sequel
\cite{AA2} for mirrors of hypersurfaces in toric varieties).

\subsection{A toric K\"ahler form on $\C^N$}\label{s:toricCN}

We first consider the case of $\C^N$ equipped with a complete toric K\"ahler form
$\omega=dd^c\Phi$ (for a $\T^N$-invariant K\"ahler potential $\Phi$) and the superpotential
$W_0=-\prod z_j$. 
Writing $z_j=\exp(\rho_j+i\theta_j)$, we have
$$\omega=dd^c\Phi=\sum_{i,j} \frac{\partial^2\Phi}{\partial \rho_i
\partial \rho_j} d\rho_i\wedge d\theta_j.$$
In particular, $\omega$ is a K\"ahler form if and only if
the potential is a strictly convex function of the $\rho$ coordinates,
i.e.\  the Hessian matrix $\Psi=(\partial^2\Phi/\partial \rho_i\partial \rho_j)_{ij}$ is positive definite.
The moment map $\mu=(\mu_1,\dots,\mu_N):\C^N\to \R^N$
is given by the partial derivatives of~$\Phi$: $$\mu_j=\partial \Phi/\partial \rho_j.$$
The horizontal distribution, i.e.\ the symplectic orthogonal to the level
sets of $W_0$, is spanned (over $\C$) by the Hamiltonian vector
field generated by $\log |W_0|=\sum \rho_j$. 
We can express $d\log |W_0|$ as a linear combination of the
differentials of the moment maps, 
\begin{equation}\label{eq:dlogWdmuj}
d\log |W_0|=\sum_j d\rho_j = \sum_j
\lambda_j d\mu_j,\quad \text{where}\quad
(\lambda_1,\dots,\lambda_N)=\Psi^{-1}(1,\dots,1).
\end{equation} Angular parallel transport (i.e., along circles
centered at the origin in the base of the fibration given by $W_0$)
is then given by
rotating each coordinate at a rate proportional to $\lambda_i$, so that the
horizontal lifts of the angular and radial vector fields are given by
\begin{equation}\label{eq:horiz_span}(\partial_\theta)^\#=
\frac{\sum \lambda_j \partial_{\theta_j}}{\sum \lambda_j}\qquad
\mathrm{and}\qquad (r\partial_r)^\#=-i(\partial_\theta)^\#=\frac{\sum \lambda_j \partial_{\rho_j}}{
\sum \lambda_j}.
\end{equation}
One checks that the quantities $\mu_j-\mu_i$ are conserved by parallel
transport, as expected (since parallel transport is equivariant with respect
to the standard Hamiltonian $\T^{N-1}$-action on the fibers of $W_0$).

\begin{example}
For the standard K\"ahler form on $\C^N$, with potential $\Phi=\frac14\sum
|z_j|^2=\frac14 \sum e^{2\rho_j}$, the moment map is given by
$\mu_j=\frac12 |z_j|^2$, and $\Psi$ is diagonal with entries $|z_j|^2$, so
that $\lambda_j=|z_j|^{-2}$, and $(\partial_\theta)^\#=\frac{1}{\sum
|z_j|^{-2}}\sum |z_j|^{-2}\partial_{\theta_j}.$ Thus, when $|z_j|\to
\infty$ for $|W_0|$ fixed, the rate of change of $\arg(z_j)$ under angular
parallel transport tends to zero.  
This in turn implies that a weaker form of 
asymptotic admissibility (only requiring arguments of monomials to converge to
prescribed limit values at infinity) is preserved under parallel transport,
and it should be possible to carry out the whole construction using
the standard K\"ahler form. However, the stronger admissibility requirement
that we impose is necessary for the maximum principle of 
Proposition \ref{prop:maxprinciple_fiber}; thus we will need to ensure that $\arg(z_j)$
remains strictly constant (rather than approximately constant) under parallel
transport, and this in turn motivates the introduction of a different K\"ahler form.
\end{example}

\noindent Our choice of K\"ahler form involves smooth approximations of the maximum
function:

\begin{definition}\label{def:Max}
Given a constant $\delta>0$, denote by $M:\R^2\to \R$ a smooth convex function such that:
\begin{enumerate}
\item $M(u,v)=\max(u,v)$ whenever $|u-v|\ge \delta$;
\item $M(u+a,v+a)=M(u,v)+a$ for all $u,v,a\in \R$; and 
\item $M(u,v)=M(v,u)$.
\end{enumerate}
These conditions imply that $M$ is monotonically increasing with either
variable, and
$$\max(u,v)\le M(u,v)\le
\max(u,v)+\delta,\qquad 0\le \frac{\partial M}{\partial u}\le 1,\quad
\mbox{and}\quad 0\le \frac{\partial M}{\partial v}\le 1.$$
We then define $\hat{M}:\R_{\ge 0}^2\to \R_{\ge 0}$ by 
\begin{align*}
\hat{M}(U,V)&=\exp M(\log U, \log V)\quad \text{for $U,V>0$,} \\
\hat{M}(U,0)&=\hat{M}(0,U)=U,
\end{align*} and note that
$\hat{M}$ is continuous, smooth everywhere except at the origin, and
$\hat{M}(U,V)=\max(U,V)$ whenever\/ $U/V \not\in
(e^{-\delta},e^\delta)$.
\end{definition}
In fact, the second condition above implies that $M$ is determined by a smoothing, near the origin, of the absolute value function on $\R$.

\begin{definition}\label{def:K_potential} Choosing some small
$\varepsilon>0$, we equip $\C^N$ with 
$\omega=dd^c\Phi$, where
\begin{equation}\label{eq:K_potential}
\Phi=\sum_{i=1}^N \hat{M}\Biggl(\varepsilon,\prod_{\substack{j=1\\ j\neq i}}^N
\hat{M}(|z_i|^2,|z_j|^2)\Biggr)|z_i|^2.
\end{equation}
\end{definition}

\begin{remark} \label{rmk:simplify_K_potential}
The only purpose of taking $\hat{M}(\varepsilon,...)$ is that otherwise
$\omega$ would be degenerate (and non-smooth) along the coordinate 
axes. In fact, $$\prod_{j\neq i}
\hat{M}(|z_i|^2,|z_j|^2) \geq \prod_{j\neq i} \max(|z_i|^2,|z_j|^2)\geq
\frac{|W_0|^2}{\min \{|z_1|^2,\dots,|z_N|^2\}}\geq 
|W_0|^{2(N-1)/N},$$ so we have the simpler expression
\begin{equation}\label{eq:K_potential_simpler}
\Phi=\sum_{i=1}^N \Bigl(\prod_{j\neq i}
\hat{M}(|z_i|^2,|z_j|^2)\Bigr)\,|z_i|^2\qquad
\text{whenever}\ |W_0|^2\ge (\varepsilon e^{\delta})^{\frac{N}{N-1}}.
\end{equation}
Since we will only consider Lagrangian submanifolds which stay away from
the preimage of a small disc under $W_0$, choosing $\varepsilon$ and $\delta$ sufficiently small we can
always work with the simpler formula \eqref{eq:K_potential_simpler} to study
the geometry of admissible Lagrangians.
\end{remark}

\begin{lemma}\label{l:K_form_CN} $\omega$ is a toric K\"ahler form on $\C^N$.
\end{lemma}

\proof $\Phi$ is obviously $\T^N$-invariant, and we will momentarily check
that outside of the coordinate axes it is strictly convex as a function of
the variables $\rho_j=\log |z_j|$. Meanwhile, smoothness and non-degeneracy 
of the Hessian
near $z_i=0$ follow from the observation that the coefficient 
$\hat{M}(\varepsilon,...)$ in the $i$-th term of \eqref{eq:K_potential} is bounded
below by $\varepsilon>0$.

To prove the strict convexity of $\Phi$ outside of the coordinate axes, we
observe that each term in the sum \eqref{eq:K_potential} is {\em log-convex} 
as a function of $\rho_j=\log |z_j|$, i.e.\  its logarithm is convex. 
Indeed, using the convexity of $M$ and 
the fact that the composition of a convex monotonically 
increasing function with a convex function is itself convex, 
we find that $$\varphi_i(\rho_1,\dots,\rho_N):=M\Bigl(\log\varepsilon,\sum_{j\neq i} M(2\rho_i,2\rho_j)\Bigr)+2\rho_i$$
is a convex function.  Since the exponential function is
strictly increasing and strictly convex, we conclude that
$$\Phi_i(\rho_1,\dots,\rho_N)=e^{\varphi_i(\rho_1,\dots,\rho_N)}=
\hat{M}\Bigl(\varepsilon,\prod_{j\neq i} \hat{M}(|z_i|^2,|z_j|^2)\Bigr)
|z_i|^2$$ is a convex function, and that its Hessian
is non-degenerate on all tangent vectors which are transverse to the level
sets of $\varphi_i$, i.e.\ $d^2 \Phi_i(v,v)>0$ whenever $d\varphi_i(v)\neq
0.$

Thus, in order to conclude that $\Phi=\sum \Phi_i$ is strictly convex, it
suffices to show that $d\varphi_1,\dots,d\varphi_N$ are everywhere linearly
independent. Equivalently, we need to show that the matrix $A$ with
entries $a_{ij}=\partial \varphi_i/\partial \rho_j$ is invertible.
For simplicity we only do this in the region where $\varphi_i=\sum_{j\neq i}
M(2\rho_i,2\rho_j)+2\rho_i$; in light of Remark
\ref{rmk:simplify_K_potential} this is the only case of genuine interest to us.

Let $\hat{A}=A+A^T$, with entries $\hat{a}_{ij}=a_{ij}+a_{ji} = \partial
\varphi_i/\partial \rho_j+\partial \varphi_j/\partial \rho_i$.
For $i\neq j$, it follows from property (2) of Definition \ref{def:Max}
that $$\hat{a}_{ij}=
\tfrac{\partial}{\partial \rho_i} M(2\rho_i,2\rho_j)+
\tfrac{\partial}{\partial \rho_j} M(2\rho_i,2\rho_j)=2.$$
Meanwhile, $\hat{a}_{ii}=2\,\partial \varphi_i/\partial \rho_i\ge 4$.
Thus, given any non-zero vector $v\in \R^N$, 
$$\langle v,\hat{A}v\rangle = \sum_{i,j=1}^N \hat{a}_{ij} v_iv_j=
2\,\Bigl(\sum_i v_i\Bigr)^2 + \sum_i (\hat{a}_{ii}-2)\,v_i^2>0,$$
and it follows that $\langle v, Av\rangle=\frac12 \langle v,\hat{A}v\rangle
$ is positive as well, which implies that $A$ is invertible, and hence
$\Phi$ is strictly convex.
\endproof

The key feature of the K\"ahler form  $\omega$ which makes it possible for
fibered Lagrangians to be fiberwise monomially admissible
is that all ``large'' coordinates are
preserved under parallel transport. We first make the notion of ``large''
coordinate more precise:

\begin{definition}
A partition $\{1,\dots,N\}=K\sqcup J$ into two non-empty subsets is called a
$\delta$-gap at a point $(z_1,\dots,z_N)\in \C^N$
 if $\inf\,\{|z_i|^2,\ i\in J\} \geq e^\delta \sup\, \{|z_i|^2,\ i\in K\}$.
We say that $z_j$ lies above a $\delta$-gap if there exists a
$\delta$-gap $\{1,\dots,N\}=K\sqcup J$ with $j\in J$.
\end{definition}

\begin{lemma}\label{l:abovegap}
If $|z_\ell|\geq e^{\frac14 (N-1)\delta}\,|W_0|^{1/N}$, or if $|z_\ell| \geq e^{\frac12
(N-1)\delta}\,\min |z_i|$, then $z_\ell$ lies
above a $\delta$-gap.
\end{lemma}
\proof Assume $z_\ell$ does not lie above any $\delta$-gap. Then listing all $|z_i|^2$ in
decreasing order, the entry just after $|z_\ell|^2$ (if there is one) is bounded below by
$e^{-\delta} |z_\ell|^2$, the next one is bounded below by $e^{-2\delta}
|z_\ell|^2$, and so on, whereas the entries preceding $|z_\ell|^2$ are bounded
below by $|z_\ell|^2$.  Thus, we conclude that $\min |z_i|^2 >
e^{-(N-1)\delta}\,|z_\ell|^2$, and $|W_0|^2=\prod_{i=1}^N |z_i|^2 > 
e^{-\frac12 N(N-1)\delta}|z_\ell|^{2N}$.  Taking the square root, resp.\ the $2N$-th root of both
sides of these inequalities, we obtain a contradiction.
\endproof

\begin{lemma}\label{l:invce_CN}
Assume that $|W_0|^2\ge (\varepsilon e^{\delta})^{\frac{N}{N-1}}$, and that
$z_\ell$ lies above a $\delta$-gap. Then 
the coordinate $z_\ell$ is (locally) invariant under parallel transport.
\end{lemma}

Before giving the proof, we provide some intuition by briefly considering
the case $N=2$: when $|z_2|^2\geq e^\delta |z_1|^2$, our K\"ahler potential is
$\Phi=|z_1|^2|z_2|^2+|z_2|^4$, and 
$\omega$ is locally a product K\"ahler
form when expressed in the coordinates $(W_0,z_2)$, which readily implies
that parallel transport for $W_0$ preserves $z_2$. Alternatively,
the first component of the moment map is 
$\mu_1=\partial\Phi/\partial\rho_1=2|W_0|^2$,
as is also the case more generally whenever $z_1$ is the smallest coordinate and
separated from $z_2,\dots,z_N$ by a $\delta$-gap. Since $d\log
|W_0|$ is proportional to $d\mu_1$, comparing \eqref{eq:dlogWdmuj} and
\eqref{eq:horiz_span} we conclude that only $z_1$ varies
along the horizontal distribution, while $z_2,\dots,z_N$ are preserved.
(However, as parallel transport towards
$|W_0|\to\infty$ proceeds by varying $z_1$ while $z_2,\dots,z_N$ remain fixed, 
eventually $|z_1|$ becomes large enough to ``close'' the $\delta$-gap and 
the statement no longer holds). The argument in the general case is less explicit but
similarly involves the vanishing of certain coefficients in \eqref{eq:dlogWdmuj}.

\proof
Let $\{1,\dots,N\}=K\sqcup J$ be a $\delta$-gap with $\ell\in J$.
Recall that the K\"ahler potential is given by
\eqref{eq:K_potential_simpler}, i.e.\ $\Phi=\sum_{i=1}^N e^{\varphi_i}$,
where $\varphi_i=\sum_{j\neq i} M(2\rho_i,2\rho_j)+2\rho_i$. 
Property (1) of Definition \ref{def:Max} implies that, for $i\in J$ and
$k\in K$, $\partial \varphi_i/\partial \rho_k\equiv 0$, whereas for $i\in K$ and
$j\in J$, $\partial \varphi_i/\partial \rho_j\equiv 2$.
Thus, for $k\in K$ and $j\in J$,
\begin{equation}\label{eq:Psikj}
\Psi_{jk}=\frac{\partial^2\Phi}{\partial \rho_j\partial \rho_k}=
\sum_{i=1}^N \left( \frac{\partial^2 \varphi_i}{\partial \rho_j \partial
\rho_k}+\frac{\partial \varphi_i}{\partial \rho_j}\frac{\partial\varphi_i}{
\partial \rho_k}\right)\,e^{\varphi_i}=\sum_{i\in K} 2 \frac{\partial
\varphi_i}{\partial \rho_k}\,e^{\varphi_i},
\end{equation}
which is {\em independent of $j$}. We denote this quantity by
$c_k$.  

Next, property (2) of Definition
\ref{def:Max} implies that for all $i$ we have
$\sum_{m=1}^N \partial \varphi_i/\partial \rho_m=2N$, and for $i\in K$
we have $\sum_{m\in K} \partial \varphi_i/\partial \rho_m=
2N-2|J|=2|K|$. Thus,
$$\sum_{m\in K} \frac{\partial \Phi}{\partial \rho_m}=
\sum_{i,m\in K} \frac{\partial \varphi_i}{\partial \rho_m} e^{\varphi_i}
=2|K|\sum_{i\in K} e^{\varphi_i}.$$
Differentiating with respect to $\rho_k$ for $k\in K$, we find that
\begin{equation}\label{eq:Psikk}
\frac{1}{|K|}\sum_{m\in K} \Psi_{mk}=\sum_{i\in K} 2 \frac{\partial
\varphi_i}{\partial \rho_k}\,e^{\varphi_i}=c_k.
\end{equation}

The non-degeneracy of $\omega$ implies that the symmetric matrix $\Psi$ is
positive definite, and in particular its restriction $\Psi_{|K}$ to the
coordinates labelled by elements of $K$ is also non-degenerate. For $k\in
K$, denote by $\lambda_k$ the components of $(\Psi_{|K})^{-1}(1,\dots,1)$,
i.e.\ by definition $\sum_{k\in K} \Psi_{ik} \lambda_k=1$
for all $i\in K$. Averaging over $i\in K$ and using \eqref{eq:Psikk},
we also have $\sum_{k\in K} c_k \lambda_k=1$. Thus, using \eqref{eq:Psikj}
we find that $$\sum_{k\in K} \Psi_{ik}\lambda_k=1\qquad \text{for all }
i=1,\dots,N.$$
Setting $\lambda_j=0$ for $j\in J$, we conclude that
that $\Psi^{-1}(1,\dots,1)=(\lambda_1,\dots,\lambda_N)$, i.e.\  
$(\lambda_1,\dots,\lambda_N)$ are the coefficients which
appear in \eqref{eq:dlogWdmuj} and \eqref{eq:horiz_span}.

For $j\in J$, the vanishing of $\lambda_j$ implies that the $j$-th
components of $(\partial_\theta)^\#$ and $(r\partial_r)^\#$ are zero,
and thus, parallel transport preserves $z_j$. This is in particular true
for $j=\ell$.
\endproof

We conclude this section with some estimates for the moment map, which
will be used to establish the analogue of Lemma \ref{l:invce_CN} for 
K\"ahler forms obtained from $\omega$ by symplectic reduction.
Since the formula for the moment map is
obviously equivariant under permutation of the variables, it suffices to
consider the case where  $|z_1|\leq |z_2|\leq \dots \leq |z_N|$.

\begin{lemma}\label{l:momentmapest}
Assume that $|W_0|^2\geq (\varepsilon e^{\delta})^{\frac{N}{N-1}}$, and that
$|z_1|\leq |z_2|\leq \dots \leq |z_N|$. Then:
\begin{enumerate}
\item $0<\mu_1\leq \mu_2\leq \dots \leq \mu_N$.\medskip
\item $2 \leq \dfrac{\mu_j}{|z_j|^{2j}\prod_{i=j+1}^N |z_{i}|^2}\leq 4N
e^{2N\delta}$
for all $1\leq j\leq N$.\medskip
\item $\displaystyle (2N)^{-\frac{1}{2N}}e^{-\delta} \Bigl(\frac{\mu_j}{\mu_k}\Bigr)^{\!\frac{1}{2N}} \leq
\frac{|z_j|}{|z_k|}\leq (2N)^{1/2} e^{N\delta}\,\Bigl(\frac{\mu_j}{\mu_k}\Bigr)^{\!\frac12}$ for
all $1\leq k<j\leq N$.
\end{enumerate}
\end{lemma}

\proof
Recall that, by \eqref{eq:K_potential_simpler}, $\Phi=\sum e^{\varphi_i}$ with $\varphi_i=\sum_{j\neq i}
M(2\rho_i,2\rho_j)+2\rho_i$. Thus,
\begin{equation}
\label{eq:mu_j_CN}
\mu_j=\frac{\partial \Phi}{\partial \rho_j}=\sum_{i=1}^N  \frac{\partial
\varphi_i}{\partial \rho_j} e^{\varphi_i}=
\sum_{i\neq j} \frac{\partial M(2\rho_i,2\rho_j)}{\partial
\rho_j}\,(e^{\varphi_i}+e^{\varphi_j}) \,+\, 2\,e^{\varphi_j}.
\end{equation}
We first establish the inequality (1). For $j<k$, we have
$|z_j|\leq |z_k|$ by assumption, and using the monotonicity of $M$ we
immediately deduce that $\varphi_j\leq \varphi_k$. Moreover, for
$i\not\in \{j,k\}$, the convexity of $M$ implies that
$0\leq \partial M(2\rho_i,2\rho_j)/\partial \rho_j
\leq \partial M(2\rho_i,2\rho_k)/\partial \rho_k$, and hence 
$$\frac{\partial M(2\rho_i,2\rho_j)}{\partial
\rho_j}\,(e^{\varphi_i}+e^{\varphi_j})\leq 
\frac{\partial M(2\rho_i,2\rho_k)}{\partial
\rho_k}\,(e^{\varphi_i}+e^{\varphi_k}).$$
Meanwhile, properties (2)(3) of Definition \ref{def:Max} and the convexity
of $M$ imply that
${\partial M(2\rho_j,2\rho_k)}/{\partial \rho_j} \leq 1 \leq
{\partial M(2\rho_j,2\rho_k)}/{\partial \rho_k},$ so
$$\frac{\partial M(2\rho_j,2\rho_k)}{\partial \rho_j}(e^{\varphi_j}+
e^{\varphi_k}) \,+\,2\,e^{\varphi_j} \leq
\frac{\partial M(2\rho_j,2\rho_k)}{\partial \rho_k}(e^{\varphi_j}+
e^{\varphi_k}) \,+\,2\,e^{\varphi_k}.$$
Combining these inequalities we conclude that $\mu_j\leq \mu_k$, which
proves (1).

To establish (2), we first observe that, for $i_1<i_2$,
$|z_{i_2}|^2\leq \hat{M}(|z_{i_1}|^2,|z_{i_2}|^2)\leq e^{\delta} |z_{i_2}|^2$.
Therefore,
$$|z_j|^{2j}\prod_{i>j} |z_i|^2 \leq e^{\varphi_j}=\Bigl(\prod_{i\neq j}
\hat{M}(|z_i|^2,|z_j|^2)\Bigr)\,|z_j|^2 \leq e^{N\delta} |z_j|^{2j}\prod_{i>j}
|z_i|^2.$$
Since $\mu_j\geq 2\,e^{\varphi_j}$ by \eqref{eq:mu_j_CN}, the lower bound on $e^{\varphi_j}$
immediately yields the lower bound in (2). 
Meanwhile, to obtain an upper bound on $\mu_j$ we observe that in the sum
\eqref{eq:mu_j_CN} the terms corresponding to $i$ such that $|z_i|^2\geq
e^{\delta} |z_j|^2$ vanish identically, since for such $i$ we have
$M(2\rho_i,2\rho_j)\equiv 2\rho_i$.  Otherwise, the inequality $2\rho_i\leq
2\rho_j+\delta$ implies that $\varphi_i\leq \varphi_j+N\delta$.
Meanwhile, $\partial M(2\rho_i,2\rho_j)/\partial \rho_j\leq 2$.
Thus,
$$\mu_j\leq \!\!\!\sum_{\substack{i\neq j\\ 2\rho_i\leq 2\rho_j+\delta}}\!\!\!
2(e^{\varphi_i}+e^{\varphi_j})\,+\, 2e^{\varphi_j}\leq (2N+2(N-1)e^{N\delta})\,
e^{\varphi_j}\leq 4Ne^{2N\delta}\,|z_j|^{2j}\prod_{i>j} |z_i|^2.$$
This establishes the upper bound in (2). Finally, (3) is a direct
consequence of (2) using the observation that
$$\biggl(\frac{|z_j|}{|z_k|}\biggr)^{\!2k}\leq \frac{|z_j|^{2j} \prod_{i=j+1}^N |z_i|^2}{|z_k|^{2k} \prod_{i=k+1}^N
|z_i|^2}=\biggl(\frac{|z_j|}{|z_k|}\biggr)^{\!2k} \prod_{i=k+1}^{j-1}
\frac{|z_j|^2}{|z_i|^2} \leq \biggl(\frac{|z_j|}{|z_k|}\biggr)^{\!2j-2},$$
which in turn implies that
$$\biggl(\frac{|z_j|^{2j} \prod_{i=j+1}^N |z_i|^2}{|z_k|^{2k} \prod_{i=k+1}^N
|z_i|^2}\biggr)^{\!\frac{1}{2N}}\leq \frac{|z_j|}{|z_k|}\leq 
\biggl(\frac{|z_j|^{2j} \prod_{i=j+1}^N |z_i|^2}{|z_k|^{2k} \prod_{i=k+1}^N
|z_i|^2}\biggr)^{\!\frac{1}{2}}.$$
\endproof

\subsection{Symplectic reduction and monomial admissibility}\label{ss:invce_Y}

Recall that the toric variety $Y$ described in \S \ref{s:LGmodel} is the
symplectic reduction of $\C^{P_\Z}$ by a subtorus $\T_M\subset
\T^{P_\Z}$, i.e.\ $Y=\mu^{-1}(\lambda)/\T_M$, and the superpotential
$W_0\in \O(\C^{P_\Z})$ descends to $W\in \O(Y)$. We equip $\C^{P_\Z}$ with
the toric K\"ahler form constructed in the previous section, and the reduced
space $Y$ with the induced toric K\"ahler form.

Our goal in this section is to show that symplectic reduction preserves the
compatibility of parallel transport with fiberwise monomial admissibility,
i.e.\ to establish an analogue of Lemma
\ref{l:invce_CN} for symplectic parallel transport between the fibers of
$W:Y\to \C$. Our starting point is the observation that ``parallel transport
commutes with reduction'':

\begin{lemma}\label{l:partransred} 
The horizontal vector fields
$(\partial_\theta)^\#$ and $(r\partial_r)^\#$ described by \eqref{eq:horiz_span},
which span the symplectic orthogonal to the
fibers of $W_0:\C^{P_\Z}\to \C$, are $\T_M$-equivariant and
tangent to $\mu^{-1}(\lambda)$. Their images under
the projection from $\mu^{-1}(\lambda)$ to $\mu^{-1}(\lambda)/\T_M=Y$ 
span the symplectic orthogonal to the fibers of $W:Y\to\C$ with respect to the reduced K\"ahler
form, and in fact they are the horizontal lifts to $Y$ of $\partial_\theta$ and
$r\partial_r$.  \qed
\end{lemma}

To take advantage of this property,
we need a criterion to determine when a $\T_M$-invariant monomial
on $\C^{P_\Z}$ involves only coordinates to which Lemma
\ref{l:invce_CN} applies.

Recall that the moment polytope $\Delta_Y$ of $Y$, given by \eqref{eq:Delta_Y}, arises
as the intersection of an affine linear subspace of $\R^{P_\Z}$ (expressing
the condition $\mu=\lambda$) with the non-negative orthant (the moment polytope
for $\C^{P_\Z}$).  Embedding $\Delta_Y$ into $\R^{P_\Z}$
in this way, the coordinate hyperplanes correspond to the facets of $\Delta_Y$,
and the ambient coordinates (i.e., the components of the moment map for $\C^{P_\Z}$) 
are given by the affine distances to the facets of~$\Delta_Y$.
Thus, in terms of the description \eqref{eq:Delta_Y}, the point
$(\xi,\eta)\in \Delta_Y\subset \R^n\oplus \R$ corresponds to a
$\T^{n+1}$-orbit in $Y$ whose preimage in $\mu^{-1}(\lambda)\subset
\C^{P_\Z}$ is the $\T^{P_\Z}$-orbit whose moment map coordinates 
$(\mu_\alpha)_{\alpha \in P_\Z}$ are given by
\begin{equation}\label{eq:mualpha}
\mu_\alpha=\eta-\langle \alpha,\xi\rangle+\nu(\alpha)\qquad \text{for all }
\alpha\in P_\Z.
\end{equation}

Given a vector $\mathbf{v}=(\vec{v},v^0)\in \Z^n\oplus \Z$, the toric monomial
$z^{\mathbf{v}}$ defines a regular function on $Y$ if and only it pairs
non-negatively with the inward normal vector to each facet of $\Delta_Y$,
i.e.\ \begin{equation}\label{eq:zvorder}
v^\alpha:=(-\alpha,1)\cdot \mathbf{v}=v^0-\alpha\cdot \vec{v}\geq 0\quad
\text{for all }\alpha\in P_\Z.\end{equation}
The monomial $z^{\mathbf{v}}$ vanishes to order $v^\alpha$ along the toric divisor of $Y$
corresponding to $\alpha\in P_\Z$. Moreover, the monomial 
$\prod_{\alpha\in P_\Z} z_\alpha^{v^\alpha}\in \O(\C^{P_\Z})$ is invariant under the
$\T_M$-action and descends to $z^{\mathbf{v}}\in \O(Y)$ under reduction.
With a slight abuse of notation, we will therefore write
\begin{equation}\label{eq:monomialorders}
z^{\mathbf{v}}=\prod_{\alpha\in P_\Z} z_\alpha^{v^\alpha}.
\end{equation}
The vectors $\mathbf{v}$ satisfying \eqref{eq:zvorder} are the integer
points of a polyhedral convex cone, whose extremal rays are in one-to-one
correspondence with the facets of $P$. 

\begin{definition}\label{def:extremalv} Given a facet of the polytope $P$
with primitive outward normal vector $\vec{v}$, contained in the affine
hyperplane $\langle \vec{v},\cdot\rangle=v^0$, the corresponding extremal
vector is $\mathbf{v}=(\vec{v},v^0)$; we denote the set of 
these vectors by $\mathcal{V}$.  
\end{definition}

The elements of $\mathcal{V}$ can be
characterized equivalently as the primitive inward normal vectors to the
$n$-dimensional cones which lie on the boundary of the fan $\Sigma_Y$, or as
the primitive tangent vectors to the unbounded edges of the moment polytope
$\Delta_Y$.

For $\mathbf{v}\in \mathcal{V}$ we denote
by $A_{\mathbf{v}}$ the set of all $\alpha \in P_\Z$ which lie on the
corresponding facet of $P$, i.e.\ those $\alpha$ for which
$\alpha\cdot \vec{v}=v^0$, or equivalently, the quantity $v^\alpha$ defined
by \eqref{eq:zvorder} vanishes.  These correspond exactly to the facets of
$\Delta_Y$ to which $\mathbf{v}$ is parallel.

Given a small positive constant $\gamma>0$ and $\mathbf{v}=(\vec{v},v^0)\in
\mathcal{V}$, we define
\begin{equation}\label{eq:Svgamma}
S_{\mathbf{v},\gamma}=\{\xi\in \R^n\,|\,\langle
\alpha,\xi\rangle-\nu(\alpha)< \varphi(\xi)-\gamma\|\xi\|\ \ \forall
\alpha \in P_\Z\setminus A_{\mathbf{v}}\},
\end{equation}
where $\|\cdot\|$ is an arbitrary norm (e.g.\ the Euclidean norm) on $\R^n$.
In other terms, recalling that $\varphi(\xi)=\max\{\langle
\alpha,\xi\rangle-\nu(\alpha)\,|\,\alpha\in P_\Z\}$, $S_{\mathbf{v},\gamma}$ 
is the set of points where the maximum is achieved by some $\alpha \in
A_{\mathbf{v}}$, and no $\alpha\not\in A_{\mathbf{v}}$ comes close to the
maximum. We also define $C_{\mathbf{v},\gamma}\subset Y$ to be the inverse
image of $S_{\mathbf{v},\gamma}\times \R$ under the moment map.

Denote by $\Delta_\alpha$ the polyhedral subset of $\R^n$ where $\alpha$ achieves the
maximum in $\varphi$ (which is also the projection to $\R^n$ of the
corresponding facet of $\Delta_Y$). Then $S_{\mathbf{v},\gamma}$ is 
a retract of $\bigcup_{\alpha \in A_{\mathbf{v}}} \Delta_\alpha$, obtained by
removing those points which are too close 
(within distance of the order of $\gamma\|\xi\|$) to some other $\Delta_\alpha$. 
Thus, for sufficiently small $\gamma$ the subsets $S_{\mathbf{v},\gamma}$,
$\mathbf{v}\in \mathcal{V}$ cover the complement of a compact subset of $\R^n$. 

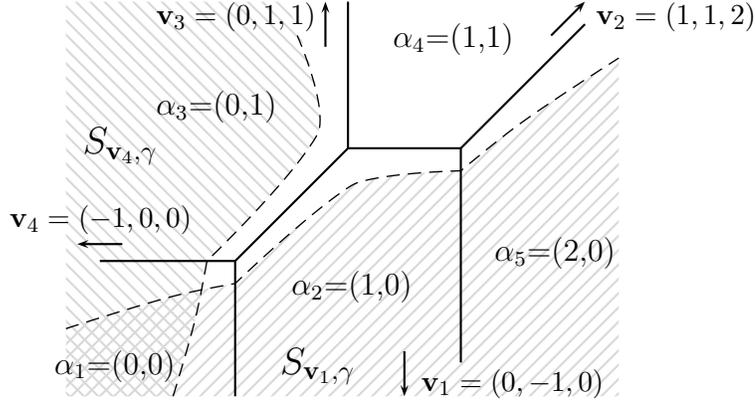
\begin{figure}[t]
\setlength{\unitlength}{1.5cm}
\begin{picture}(5,3.5)(-1.5,-1.2)
\psset{unit=\unitlength}
\newgray{gray15}{0.85}
\newgray{gray10}{0.9}
\newgray{gray5}{0.95}
\pscurve[fillstyle=hlines,hatchcolor=gray15,hatchsep=3pt,linestyle=none](-1.5,-1.2)%
(-1.5,-0.6)(-1.5,-0.6)(-0.5,-0.3)(0,-0.2)(0,-0.2)(1,0.65)(1.1,0.7)(2,0.8)(2,0.8)(2.5,1.2)(3.4,1.8)%
(3.4,1.8)(3.4,-1.2)(3.4,-1.2)(1,-1.2)(-1.5,-1.2)
\pscurve[fillstyle=vlines,hatchcolor=gray15,hatchsep=3pt,linestyle=none](-1.5,-1.2)(-0.55,-1.2)%
(-0.55,-1.2)(-0.35,-0.5)(-0.25,0)(-0.25,0)(0.65,1)(0.75,1.2)(0.45,2.3)(0.45,2.3)(-1.5,2.3)(-1.5,2.3)
\psline(0,-1.2)(0,0)
\psline(-1.2,0)(0,0)(1,1)(1,2.3)
\psline(1,1)(2,1)(2,-0.9)
\psline(2,1)(3.1,2.1)
\put(-1.6,-1){$\alpha_1$=(0,0)}
\put(-0.7,1.3){$\alpha_3$=(0,1)}
\put(0.5,-0.3){$\alpha_2$=(1,0)}
\put(2.3,0){$\alpha_5$=(2,0)}
\put(1.4,1.9){$\alpha_4$=(1,1)}
\psline{->}(1.5,-0.8)(1.5,-1.2)
\put(1.65,-1.15){\small $\mathbf{v}_1=(0,-1,0)$}
\psline{->}(2.8,2)(3.1,2.3)
\put(3.2,2.1){\small $\mathbf{v}_2=(1,1,2)$}
\psline{->}(0.8,1.9)(0.8,2.3)
\put(-0.7,2.1){\small $\mathbf{v}_3=(0,1,1)$}
\psline{->}(-1,0.15)(-1.4,0.15)
\put(-2,0.3){\small $\mathbf{v}_4=(-1,0,0)$}
\pscurve[linewidth=0.5pt,linestyle=dashed](-1.5,-0.6)(-0.5,-0.3)(0,-0.2)(0,-0.2)(1,0.65)(1.1,0.7)(2,0.8)(2,0.8)(2.5,1.2)(3.4,1.8)
\pscurve[linewidth=0.5pt,linestyle=dashed](-0.55,-1.2)(-0.35,-0.5)(-0.25,0)(-0.25,0)(0.65,1)(0.75,1.2)(0.55,2)
\put(-1.35,0.95){\large $S_{\mathbf{v}_4,\gamma}$}
\put(0.4,-1){\large $S_{\mathbf{v}_1,\gamma}$}
\end{picture}
\caption{The extremal vectors $\mathbf{v}\in\mathcal{V}$ and the regions $S_{\mathbf{v},\gamma}$,
for $f(x_1,x_2)=1+x_1+x_2+t^{2\pi}x_1x_2+t^{4\pi}x_1^2$ \ (cf.\ Example
\ref{ex:example_42}).}
\label{fig:Svgamma}
\end{figure}

\begin{example}\label{ex:example_42}
Consider $f(x_1,x_2)=1+x_1+x_2+t^{2\pi}x_1x_2+t^{4\pi}x_1^2$ (as in Example
\ref{ex:example_21}) and its
tropicalization
$\varphi(\xi_1,\xi_2)=\max\{0,\xi_1,\xi_2,\xi_1+\xi_2-1,2\xi_1-2\}$.
The convex hull $P$ of $P_\Z=\{(0,0),(1,0),(0,1),(1,1),(2,0)\}$ is a trapezoid 
with primitive outward normal vectors $\vec{v}_1=(0,-1)$, $\vec{v}_2=(1,1)$,
$\vec{v}_3=(0,1)$, and $\vec{v}_4=(-1,0)$, and we find that
$\mathcal{V}$ consists of the four elements $\mathbf{v}_1=(0,-1,0)$,
$\mathbf{v}_2=(1,1,2)$, $\mathbf{v}_3=(0,1,1)$, and $\mathbf{v}_4=(-1,0,0)$,
which are indeed the tangent vectors to the unbounded edges of the moment
polytope $\Delta_Y=\{(\xi_1,\xi_2,\eta)\,|\,\eta\geq \varphi(\xi_1,\xi_2)\}$,
shown ``from above'' on Figure \ref{fig:Svgamma}.

For $\v_1=(0,-1,0)$, the elements of $P_\Z$ which lie on the facet of $P$ with
outward normal vector $\vec{v}_1=(0,-1)$ are $\alpha_1=(0,0)$, $\alpha_2=(1,0)$,
and $\alpha_5=(2,0)$, whereas $P_\Z\setminus A_{\v_1}$ consists of
$\alpha_3=(0,1)$ and $\alpha_4=(1,1)$, so 
$$S_{\v_1,\gamma}=\{\xi=(\xi_1,\xi_2)\in
\R^2\,|\ \xi_2<\varphi(\xi)-\gamma\|\xi\|\text{ and }
\xi_1+\xi_2-1<\varphi(\xi)-\gamma\|\xi\|\}$$
is the set of points where the two terms $\xi_2$ and $\xi_1+\xi_2-1$ are
sufficiently far from achieving the maximum in $\varphi(\xi_1,\xi_2)$; see
Figure \ref{fig:Svgamma}. This is a retract of the region
$\Delta_{\alpha_1}\cup\Delta_{\alpha_2}\cup\Delta_{\alpha_5}$ where the
maximum is achieved by one of the three other terms. 
Similarly for the other regions $S_{\v_i,\gamma}$. \qed
\end{example}

For $c\in \C^*$, the fiber $W^{-1}(c)$ of $W:Y\to \C$ is $\T^n$-invariant, and its image under the
moment map is the graph $\{(\xi,\eta)\in \R^n\oplus
\R\,|\,\eta=f_c(\xi)\}$ of a piecewise smooth function $f_c:\R^n\to \R$
(with $f_c(\xi)>\varphi(\xi)$ everywhere).  We now show that, outside of a
bounded subset of $W^{-1}(c)$ (whose size depends on $c$), the monomial
$z^{\mathbf{v}}$ is locally preserved by parallel transport at all points of
$C_{\mathbf{v},\gamma}$.

\begin{proposition}\label{prop:invce_Y}
Let $z\in W^{-1}(c)\cap C_{\mathbf{v},\gamma}\subset Y$ for some
$\mathbf{v}\in \mathcal{V}$ and $\gamma>0$, with moment map coordinates
$(\xi,\eta)$, $\xi\in S_{\mathbf{v},\gamma}$. Assume that 
$|c|^2\geq (\varepsilon e^\delta)^{N/(N-1)}$ and $\|\xi\|\geq 
R=R(c,\gamma):=8N^2e^{N(N+3)\delta}\gamma^{-1}|c|^2$ (where $N=|P_\Z|$,
and $\varepsilon$ and $\delta$ are the same constants as in Section \ref{s:toricCN}).
Then the monomial $z^{\mathbf{v}}\in \O(Y)$ is locally invariant under
parallel transport.
\end{proposition}

\begin{example} Continuing with Example \ref{ex:example_42}, consider
the case of $\v_1=(0,-1,0)$, for which the quantities $v^\alpha$ defined by
\eqref{eq:zvorder} are $0,0,1,1,0$ for $\alpha_1,\dots,\alpha_5$
respectively. Thus, $z^{\v_1}\in\O(Y)$ arises by toric reduction from the monomial
$z_3z_4\in\O(\C^{P_\Z})$, which is indeed invariant under the action of
the 2-dimensional torus $\T_M$ described in Example \ref{ex:example_22}.
By Lemma \ref{l:invce_CN}, the monomial $z_3z_4$ is invariant under parallel
transport for $W_0:\C^{P_\Z}\to\C$ wherever $z_3$ and $z_4$ lie above a
$\delta$-gap. The main ingredient of the proof is therefore to prove that
such a gap exists whenever $\xi\in S_{\v_1,\gamma}$ and $\|\xi\|$ is
sufficiently large; the key point being that, by
\eqref{eq:mualpha}, when $\xi\in S_{\v_1,\gamma}$ the moment map coordinates
$\mu_{\alpha_3}$ and $\mu_{\alpha_4}$ are bounded below by 
$\min(\mu_{\alpha_i})+\gamma\|\xi\|$.
\end{example}

\proof Denote by $(z_\alpha)_{\alpha\in P_\Z}$ the coordinates of a lift of
$z\in Y$ to $\mu^{-1}(\lambda)\subset \C^{P_\Z}$, and let $\alpha_0\in P_\Z$ be
such that $\xi\in \Delta_{\alpha_0}$. Then by
\eqref{eq:mualpha} the smallest moment map coordinate is
$\min(\mu_\alpha)=\mu_{\alpha_0}=\eta-\varphi(\xi)=f_c(\xi)-\varphi(\xi)$. On the other hand, Lemma
\ref{l:momentmapest}~(2) gives a bound on the ratio between $\mu_{\alpha_0}$
and $|W_0(z)|^2=|c|^2$. We conclude that 
\begin{equation}\label{eq:mumin_bound}
2|c|^2\leq \mu_{\alpha_0}=f_c(\xi)-\varphi(\xi)\leq
4Ne^{2N\delta} |c|^2.
\end{equation}
On the other hand, since $\xi\in S_{\mathbf{v},\gamma}$ and $\|\xi\|\geq R$,
for all $\alpha \not \in A_\mathbf{v}$ we
have $$\mu_\alpha=f_c(\xi)-\langle
\alpha,\xi\rangle+\nu(\alpha)\geq \mu_{\alpha_0}+\gamma \|\xi\|\geq \gamma
R=8N^2 e^{N(N+3)\delta} |c|^2\geq 2N e^{N(N+1)\delta}\,\mu_{\alpha_0}.$$
Hence, by Lemma \ref{l:momentmapest}~(3),
$$\frac{|z_\alpha|}{|z_{\alpha_0}|}\geq (2N)^{-\frac{1}{2N}} e^{-\delta}
\left(\frac{\mu_\alpha}{\mu_{\alpha_0}}\right)^{\!\frac{1}{2N}}\geq e^{(N-1)\delta/2}.$$
By Lemma \ref{l:abovegap}, we conclude that $z_\alpha$ lies above a
$\delta$-gap for all $\alpha \not \in A_{\mathbf{v}}$. Hence, by Lemma
\ref{l:invce_CN} the coordinates $z_\alpha$ ($\alpha \not\in
A_{\mathbf{v}}$) are locally invariant under parallel
transport in $\C^{P_\Z}$. Using the fact that the exponents $v^\alpha$ in
\eqref{eq:monomialorders} vanish
for all $\alpha \in A_{\mathbf{v}}$ (by definition of $A_{\mathbf{v}}$) and
the compatibility of parallel transport with reduction (Lemma
\ref{l:partransred}), we conclude that $z^{\mathbf{v}}$ is locally preserved
under parallel transport in $Y$.
\endproof

Finally, we show that, 
at every point where $\|\xi\|$ is sufficiently large, Proposition
\ref{prop:invce_Y} applies to the largest (in a suitably renormalized sense) among
the monomials $z^{\mathbf{v}}$, $\mathbf{v}\in \mathcal{V}$.
More precisely, for $\mathbf{v}\in \mathcal{V}$ and $v^\alpha$ as in
\eqref{eq:zvorder}, we set
\begin{equation}\label{eq:dofv}
d(\mathbf{v})=\sum_{\alpha \in P_\Z} v^\alpha.
\end{equation}

\begin{proposition} \label{prop:largest_invce_Y}
There exist positive constants $\gamma_0$ and $K_0$, depending only on the
polytope $\Delta_Y$ $($and on $N$, $\varepsilon$, $\delta)$ with the
following property.
Let $z\in W^{-1}(c)\subset Y$ be a point with moment map coordinates $(\xi,\eta)$,
where $|c|^2 \geq (\varepsilon e^{\delta})^{N/(N-1)}$ and $\|\xi\| \geq
K_0|c|^2$. 
Let $\mathbf{v}_0\in\mathcal{V}$ be such that
$$|z^{\mathbf{v}_0}|^{1/d(\mathbf{v}_0)}=\max\{|z^{\mathbf{v}}|^{1/d(\mathbf{v})}\,|\,\mathbf{v}\in
\mathcal{V}\}.$$
Then $\xi\in S_{\mathbf{v}_0,\gamma_0}$ and $z\in C_{\mathbf{v}_0,\gamma_0}$.
\end{proposition}

\proof Let $(z_\alpha)_{\alpha \in P_\Z}$ be a lift of $z\in Y$ to
$\mu^{-1}(\lambda)\subset \C^{P_\Z}$. Recall from \eqref{eq:mumin_bound}
that the smallest moment map coordinate $\mu_{\alpha_0}$ corresponds to
$\alpha_0\in P_\Z$ such that $\xi\in \Delta_{\alpha_0}$, and
$\mu_{\alpha_0}$ is bounded by
$4N e^{2N\delta}|c|^2$. On the other hand, let $\alpha_1\in P_\Z$ be such
that $|z_{\alpha_1}|=\max \{|z_\alpha|,\ \alpha\in P_\Z\}$, or equivalently,
$\mu_{\alpha_1}=\max \{\mu_\alpha,\ \alpha\in P_\Z\}$. By
\eqref{eq:mualpha}, $\mu_{\alpha_1}-\mu_{\alpha_0}=\langle
\alpha_0-\alpha_1,\xi\rangle+\nu(\alpha_1)-\nu(\alpha_0)$, so there exist
positive constants $c_1,c_2$ depending only on $\Delta_Y$ such that
\begin{equation}\mu_{\alpha_1}\leq \mu_{\alpha_0}+c_1\|\xi\|+c_2.\end{equation}
On the other hand, since $P$ is assumed to have non-empty interior, 
the quantity 
$\max \{\langle \alpha-\alpha',\xi\rangle,\ \alpha,\alpha'\in P_\Z\}$ is
bounded below by a positive constant times $\|\xi\|$, and there are positive constants $c'_1,c'_2$
depending only on $\Delta_Y$ such that
\begin{equation}\label{eq:maxmu_lowerbound}
\mu_{\alpha_1}\geq \mu_{\alpha_0}+c'_1\|\xi\|-c'_2.
\end{equation}

Assume that $\xi \in S_{\mathbf{v},\gamma}$ for some $\mathbf{v}\in
\mathcal{V}$ and $\gamma>0$. Then for all $\alpha \not\in A_{\mathbf{v}}$ we
have \begin{equation}
\label{eq:mualpha_lowerbound}
\mu_\alpha=f_c(\xi)-\langle \alpha,\xi\rangle+\nu(\alpha)\geq
\mu_{\alpha_0}+\gamma\|\xi\|.
\end{equation}
Thus, assuming some lower bound on $\|\xi\|$ (e.g.\ $\|\xi\|\geq 1$), 
the upper bound on $\mu_{\alpha_1}$ implies the existence of
a constant $c_3>0$ (still depending only on $\Delta_Y$) such that, for all
$\alpha\not \in A_{\mathbf{v}}$, $\mu_\alpha \geq c_3 \gamma \mu_{\alpha_1}$.
Using Lemma \ref{l:momentmapest}~(3), this in turn yields the inequality
\begin{equation}\label{eq:zalpha_lowerbound}
|z_\alpha|\geq (2N)^{-1/2} e^{-N\delta} c_3^{1/2}\gamma^{1/2}\,|z_{\alpha_1}|\quad
\text{for all } \alpha\not\in A_{\mathbf{v}}.
\end{equation}
Taking a weighted geometric mean (and recalling that $v^{\alpha}=0$ for $\alpha\in A_{\mathbf{v}}$), we
get:
\begin{equation}\label{eq:zv_lowerbound}
|z^{\mathbf{v}}|^{1/d(\mathbf{v})}\geq (2N)^{-1/2} e^{-N\delta}
c_3^{1/2}\gamma^{1/2}\,|z_{\alpha_1}|.
\end{equation}

Conversely, if $\xi\not \in S_{\mathbf{v},\gamma}$, then 
 $\langle \alpha,\xi\rangle-
\nu(\alpha)\geq \varphi(\xi)-\gamma\|\xi\|$ for some $\alpha\not\in A_{\mathbf{v}}$, hence
$$\mu_\alpha=f_c(\xi)-\langle \alpha,\xi\rangle+\nu(\alpha)\leq
\mu_{\alpha_0}+\gamma\|\xi\|.$$
When $\|\xi\|$ is sufficiently large, namely
$\|\xi\|\geq \max(2c'_2/c'_1,4Ne^{2N\delta}\gamma^{-1}|c|^2)$, we have
$\mu_\alpha \leq 2\gamma\|\xi\|$ and $\mu_{\alpha_1}\geq \frac12 c'_1 \|\xi\|$.
Therefore, $\mu_\alpha \leq c'_3
\gamma \mu_{\alpha_1}$, where $c'_3=4(c'_1)^{-1}$.
Using Lemma \ref{l:momentmapest}~(3), this in turn yields the inequality
\begin{equation}\label{eq:zalpha_upperbound}
|z_\alpha|\leq (2N)^{1/2N} e^\delta (c'_3)^{1/2N} \gamma^{1/2N}
|z_{\alpha_1}|.
\end{equation}
Since $\alpha\not\in A_{\mathbf{v}}$, by definition the exponent $v^\alpha$
of $z_\alpha$ in the expression for $z^{\mathbf{v}}$ is at least 1.
Since the other coordinates which appear in the expression for $z^{\mathbf{v}}$
are all bounded by $|z_{\alpha_1}|$, we obtain:
\begin{equation}\label{eq:zv_upperbound}
|z^{\mathbf{v}}|^{1/d(\mathbf{v})}\leq
e^{\delta/d(\mathbf{v})}\,(2N\,c'_3\gamma)^{\frac{1}{2Nd(\mathbf{v})}}\,|z_{\alpha_1}|.
\end{equation}

With the necessary estimates in hand, we now proceed with the proof. First,
there exists $\gamma_1>0$ depending only on $\Delta_Y$ such that the subsets
$S_{\mathbf{v},\gamma_1}$ cover all but a bounded subset of $\R^n$, i.e.\
for some constant $K_1>0$ (depending only on $\Delta_Y$), 
every point with $\|\xi\|\geq K_1$ belongs to some
$S_{\mathbf{v},\gamma_1}$. Thus, whenever $\|\xi\|\geq K_1$, the estimate
\eqref{eq:zv_lowerbound} implies that
\begin{equation}\label{eq:maxzv_lowerbound}
\max\{|z^{\mathbf{v}}|^{1/d(\mathbf{v})}\,|\,\mathbf{v}\in
\mathcal{V}\}\geq (2N)^{-1/2} e^{-N\delta} c_3^{1/2} \gamma_1^{1/2}
\max\{|z_\alpha|,\ \alpha\in P_\Z\}.
\end{equation}
Let $\mathcal{D}=\{d(\mathbf{v}),\ \mathbf{v}\in \mathcal{V}\}$ (a finite set
of positive integers). We now choose $\gamma_0$ so that
\begin{equation}\label{eq:gamma0_choice}
e^{\delta/d} (2N c'_3\gamma_0)^{\frac{1}{2Nd}} < (2N)^{-1/2}
e^{-N\delta} c_3^{1/2} \gamma_1^{1/2} \qquad \text{for all }d\in \mathcal{D},
\end{equation}
and $K_0$ so that
$$K_0 \geq 4N e^{2N\delta} \gamma_0^{-1}\quad \text{and} \quad
K_0 (\varepsilon e^{\delta})^{N/(N-1)} \geq \max(K_1, 2c'_2/c'_1).$$
Assume $|c|^2\geq (\varepsilon e^{\delta})^{N/(N-1)}$ and $\|\xi\|\geq
K_0 |c|^2$, and let $\mathbf{v}_0$ be such that
$|z^{\mathbf{v}_0}|^{1/d(\mathbf{v}_0)}=\max\{|z^{\mathbf{v}}|^{1/d(\mathbf{v})}\,|\,\mathbf{v}\in
\mathcal{V}\}$. If $\xi \not\in S_{\mathbf{v}_0,\gamma_0}$, then
\eqref{eq:zv_upperbound} and \eqref{eq:gamma0_choice} give
\begin{eqnarray*}|z^{\mathbf{v}_0}|^{1/d(\mathbf{v}_0)}&\leq& 
e^{\delta/d(\mathbf{v}_0)}\,(2N\,c'_3\gamma_0)^{\frac{1}{2Nd(\mathbf{v}_0)}}\,
\max\{|z_\alpha|,\ \alpha\in P_\Z\} \\
&<& (2N)^{-1/2} e^{-N\delta} c_3^{1/2}\gamma_1^{1/2}\,\max\{|z_\alpha|,\ \alpha\in P_\Z\},
\end{eqnarray*}
which contradicts \eqref{eq:maxzv_lowerbound}.  Thus $\xi\in
S_{\mathbf{v}_0,\gamma_0}$, or equivalently, $z\in
C_{\mathbf{v}_0,\gamma_0}$.
\endproof

Propositions \ref{prop:invce_Y} and \ref{prop:largest_invce_Y} imply the
following:

\begin{corollary}\label{cor:invce_Y}
The extremal monomials $z^{\mathbf{v}}$, $\mathbf{v}\in \mathcal{V}$
introduced in Definition \ref{def:extremalv}, the weights $d(\mathbf{v})$ defined in
\eqref{eq:dofv}, the open subsets $C_{\mathbf{v}}=C_{\mathbf{v},\gamma_0}$,
and the closed subset $\Omega=\{z\in Y\,|\ \|\xi\|\leq K'_0
\max(1,|W|^2)\}$, where $K'_0=\max(8N^2 e^{N(N+3)\delta} \gamma_0^{-1}, K_0)$, define a
fiberwise monomial subdivision on the toric Landau-Ginzburg model $(Y,W,\omega)$.

Moreover, with respect to this subdivision,
fiberwise monomial admissibility (with fixed phase angles)
is preserved by parallel
transport between the fibers of\/ $W$ over any path $\gamma:[0,1]\to \C$
such that $|\gamma(0)|^2\geq (\varepsilon e^\delta)^{N/(N-1)}$ and
$|\gamma(t)|$ is non-decreasing.
\end{corollary}

\proof
The fact that the collection of extremal monomials
$(z^\mathbf{v})_{\mathbf{v}\in \mathcal{V}}$ defines a proper map is a
classical fact of toric geometry, but can also be seen directly from the
lower bound \eqref{eq:maxzv_lowerbound}. Properties (2) and (3) of 
Definition \ref{def:admissibility} are clear from the construction, and property (4) follows from 
Proposition \ref{prop:largest_invce_Y}. 

When $|W|^2\geq (\varepsilon
e^{\delta})^{N/(N-1)}$, Proposition \ref{prop:invce_Y} implies that
$z^{\mathbf{v}}$ is invariant under parallel transport at every point of
$C_{\mathbf{v}}$ which lies outside of $\Omega$. Thus, the property that
$\arg(z^{\mathbf{v}})=\varphi_{\mathbf{v}}$ is preserved under parallel
transport. The reason why we require $|\gamma(t)|$ to be non-decreasing
with respect to $t$ is to ensure that $C_{\mathbf{v}}\setminus
(C_{\mathbf{v}}\cap \Omega)$ is preserved under parallel transport
(using the fact that $\xi$ is preserved under parallel transport).
\endproof

\subsection{The wrapping Hamiltonian}\label{ss:wrappingHam}

We now define a Hamiltonian $H:Y\to \R$ whose flow preserves both the
fibers of $W$ and monomial admissibility within them.
This Hamiltonian is constructed by reduction from the case of $\C^N$.
The construction involves a smooth approximation of the minimum
function, conceptually similar to Definition \ref{def:Max} but with $N$
variables.

\begin{definition}\label{def:Min}
Given a constant $\delta'>0$, denote by $m:\R^N\to \R$ a smooth concave
function with the following properties:
\begin{enumerate}
\item letting $I=\{i\,|\,u_i< \min(u_1,\dots,u_N)+\delta'\}$, locally
$m(u_1,\dots,u_N)$ depends only on $(u_i)_{i\in I}$, and if $I=\{i_0\}$ then
$m(u_1,\dots,u_N)=u_{i_0}=\min(u_1,\dots,u_N)$;
\item $m(u_1+a,\dots,u_N+a)=m(u_1,\dots,u_N)+a$ for all $a\in \R$;
\item $m$ is symmetric, i.e.\
$m(u_{\sigma(1)},\dots,u_{\sigma(N)})=m(u_1,\dots,u_N)$ for all $\sigma \in
\mathfrak{S}_N$.
\end{enumerate}
These conditions imply that $m$ is monotonically increasing in
all  variables, and $$\min(u_1,\dots,u_N)-\delta' \leq m(u_1,\dots,u_N)\leq
\min(u_1,\dots,u_N).$$
\end{definition}
\noindent For instance, for $\delta'\geq N\delta$ we can take
$$\textstyle
m(u_1,\dots,u_N)=-\frac{1}{N!} \sum_{\sigma\in \mathfrak{S}_N}
M(-u_{\sigma(1)},M(\dots,M(-u_{\sigma(N-1)},-u_{\sigma(N)})\dots)).$$

Denoting $\mu_1,\dots,\mu_N$ the moment map coordinates for the
chosen toric K\"ahler form on $\C^N$, the Hamiltonian we consider is
\begin{equation}\label{eq:wrapping_ham_CN}
H_0=\sum_{i=1}^N \mu_i-N m(\mu_1,\dots,\mu_N).
\end{equation}
Setting $N=|P_\Z|$ and viewing $Y$ as a symplectic reduction of $\C^{P_\Z}$,
recall that the moment map coordinates $\mu_1,\dots,\mu_N$ descend to functions
$(\mu_\alpha)_{\alpha\in P_\Z}$ on the moment polytope $\Delta_Y$, given by
\eqref{eq:mualpha}. We then define the Hamiltonian $H$ on $Y$ via reduction:

\begin{definition} \label{def:wrapping_ham}
Given a point of\/ $Y$ with moment map
coordinates $(\xi,\eta)\in \Delta_Y$, set
$\mu_\alpha=\eta-\langle \alpha,\xi\rangle+\nu(\alpha)$ for all $\alpha\in
P_\Z$ as before. Then we define $H:Y\to\R$ by
\begin{equation}\label{eq:wrapping_ham}
H=\sum_{\alpha\in P_\Z} \mu_\alpha\,-\,|P_\Z|\, m(\{\mu_\alpha\}_{\alpha\in P_\Z}).
\end{equation}
\end{definition}

\begin{proposition}\label{prop:wrapping_ham}
$H$ only depends on the moment map coordinates $(\xi_1,\dots,\xi_n)$, and as
a function of these variables it is proper, convex, and grows linearly at
infinity.
In particular, the flow of $H$ preserves the fibers of $W$, and the 
restriction of $H$ to every fiber of $W$ is proper.
\end{proposition}

\proof Clearly $H$ is a function of the moment map coordinates $(\xi_1\dots,\xi_n,\eta)$. 
Since $\partial \mu_\alpha/\partial \eta=1$ for all $\alpha\in
P_\Z$, property (2) of Definition \ref{def:Min} implies that $\partial
H/\partial \eta=0$, i.e.\ $H$ only depends on $(\xi_1,\dots,\xi_n)$.
This in turns implies that $X_H$ is everywhere in the linear span of the
generators of the $\T^n$-action and preserves the fibers of $W$.

Since the coordinates $\mu_\alpha$ are affine linear functions of
$(\xi_1,\dots,\xi_n,\eta)$, the convexity of $H$ as a function of these
variables (and hence of $(\xi_1,\dots,\xi_n)$) follows from the concavity of $m$.
Meanwhile, the properness of $H$ follows from our assumption that $P$ has non-empty
interior, which yields the lower bound \eqref{eq:maxmu_lowerbound} on $\max
\{\mu_\alpha\}-\min \{\mu_\alpha\}$; the linear growth is manifest.
\endproof

\begin{proposition}\label{prop:wrapping_flatness}
The flow of $H$ preserves monomial admissibility with respect to the
fiberwise monomial subdivision of Corollary \ref{cor:invce_Y}. More precisely, if
$\ell\subset W^{-1}(c)$ is monomially admissible with phase angles
$\{\varphi_{\mathbf{v}},\ \mathbf{v}\in \mathcal{V}\}$, then its image under the
time $t$ flow is monomially admissible at infinity with phase angles
$\{\varphi_{\mathbf{v}}+t\,d(\mathbf{v}),\ \mathbf{v}\in \mathcal{V}\}$, where $d(\mathbf{v})\in \Z_+$ is given
by \eqref{eq:dofv}.
\end{proposition}

\proof On $\C^N$, the Hamiltonian $H_0$ defined by \eqref{eq:wrapping_ham_CN} is a
function of the moment map coordinates $\mu_1,\dots,\mu_N$. Letting
$I=\{i\,|\,\mu_i<\min(\mu_1,\dots,\mu_N)+\delta'\}$ as in Definition
\ref{def:Min}~(1), we observe that $\partial H_0/\partial \mu_i\equiv 1$ for all
$i\not \in I$. Thus, the flow of $H_0$ rotates the coordinates $z_i$
uniformly at unit speed for all $i\not\in I$. Moreover, this flow is $\T_M$-equivariant,
preserves $\mu^{-1}(\lambda)\subset \C^N$, and descends to
$Y=\mu^{-1}(\lambda)/\T_M$ as the Hamiltonian flow generated by $H$.

Using the same notations as in the previous section,
fix $\mathbf{v}\in \mathcal{V}$, and consider a point
of $C_{\mathbf{v}}=C_{\mathbf{v},\gamma_0}\subset Y$ with moment map coordinates
$(\xi,\eta)$ such that $\|\xi\|\geq \gamma_0^{-1}
\delta'$. (Choosing $\delta'$ sufficiently small in Definition \ref{def:Min},
we can ensure that every point outside of $\Omega$ satisfies this
inequality.)  Denote $\mu_{\alpha_0}=\min \{\mu_\alpha\}$.
By \eqref{eq:mualpha_lowerbound}, for $\alpha \not\in A_{\mathbf{v}}$
we have $\mu_\alpha\geq \mu_{\alpha_0}+\gamma_0\|\xi\|\geq
\mu_{\alpha_0}+\delta'$. Thus, $m(\{\mu_\alpha\}_{\alpha\in P_\Z})$ only
depends on $(\mu_\alpha)_{\alpha \in A_{\mathbf{v}}}$, and
the flow generated by $H$ rotates all the other coordinates ($z_\alpha$,
$\alpha\not\in A_{\mathbf{v}}$) at unit speed. Recalling that
$z^{\mathbf{v}}=\prod_\alpha z_\alpha^{v^\alpha}$ with $v^\alpha=0$ whenever
$\alpha\in A_{\mathbf{v}}$, we conclude that the flow of $X_H$ rotates
$z^{\mathbf{v}}$ uniformly at a rate of $\sum_{\alpha}
v^\alpha=d(\mathbf{v})$ at every point of $C_{\mathbf{v}}$ which lies
outside of $\Omega$. The result follows.
\endproof

\begin{remark}
Essentially any Hamiltonian satisfying the conditions of Proposition
\ref{prop:wrapping_ham} and Proposition \ref{prop:wrapping_flatness}
(possibly with different values of the phase shifts $d(\v)$, as long
as these remain positive) would be equally suitable for our purposes;
see e.g.\ Hanlon's work \cite{Hanlon} for another construction.
The Hamiltonian of Definition \ref{def:wrapping_ham} is particularly natural from the perspective 
of symplectic reduction from $\C^N$ to $Y$, but the category $\Wrap(Y,W)$ is, up to equivalence, independent of the
choice, as will be clear from the arguments in Section \ref{s:calculation}.
\end{remark}



\section{Computation of fiberwise wrapped Floer cohomology}
\label{s:calculation}

\subsection{Geometric setup}\label{ss:L0setup}

We now fix the geometric data needed for our construction of the admissible
Lagrangian $L_0\in \Wrap(Y,W)$, besides the
K\"ahler forms and wrapping Hamiltonians defined in Section \ref{s:toric},
and check that the various conditions imposed in Section \ref{s:Fukayacat}
are satisfied.

Let $(Y,W=-z^{(0,\dots,0,1)})$ be the Landau-Ginzburg model constructed in Section \ref{s:LGmodel}, 
equipped with the toric K\"ahler form $\omega$ which is the result of symplectic
reduction by $\T_M$ of the K\"ahler form on $\C^{P_\Z}$ introduced in
Definition \ref{def:K_potential}. Let $\mathcal{V}\subset \Z^{n+1}$ be the
set of extremal vectors of Definition \ref{def:extremalv}, $d(\v)$ the
positive integers given by \eqref{eq:dofv}, and the subsets $C_\v$
and $\Omega$ of $Y$ as in Corollary \ref{cor:invce_Y}.
We consider the height function
\begin{equation}\label{eq:height_fn}
h=\max\{h_\v,\ \v\in \mathcal{V}\}:Y\to [0,\infty),
\quad \text{where}\quad h_\v=|z^\v|^{1/d(\v)},
\end{equation}
and the wrapping Hamiltonian $H$ introduced in Definition
\ref{def:wrapping_ham}. 

We fix a properly embedded U-shaped arc $\gamma_0:\R\to\C$ such that 
$\gamma_0(0)=-1$; $|\gamma_0(s)|$ passes through a minimum at $s=0$ and
increases monotonically as a function of $|s|$; $\arg \gamma_0(s)$
increases monotonically as a function of $s$;
$\arg \gamma_0(s)=\theta_0$ for $s\ll 0$ and $\arg \gamma_0(s)=2\pi-\theta_0$
for $s\gg 0$, for some positive angle $0<\theta_0<\frac{\pi}{2}$.
(Thus, $\gamma_0$ intersects the negative real axis transversely
at $-1$, remains at distance at least 1 from the origin, and outside of a compact subset it coincides with the
rays $e^{\pm i\theta_0}\R_+$.)

Given a monomially admissible Lagrangian submanifold $\ell\subset
W^{-1}(-1)\cong (\C^*)^n$ (in the sense of Definition
\ref{def:flat_at_infinity}), with all phase angles equal to zero, 
we denote by $L=\cup\ell$ the fibered Lagrangian submanifold of $Y$ 
obtained from $\ell$ by parallel transport in the fibers of $W$ over the
arc $\gamma_0$. It follows from Corollary \ref{cor:invce_Y} that $L$ is
fiberwise monomially admissible, with all phase angles equal to zero.
We will in particular consider the case where $\ell=\ell_0$ is the real
positive locus of $W^{-1}(-1)$, i.e.\ the set of points where
all toric monomials are real positive and $z^{(0,\dots,0,1)}=1$, and denote
its parallel transport by $L_0=\cup\ell_0$.

As in Section \ref{sec:defin-direct-categ}, we choose an autonomous flow $\rho^t$
in the complex plane which fixes the 
negative real axis pointwise as well as a small neighborhood of the origin,
specifically the disc $\Delta'$ of radius $(\varepsilon e^\delta)^{N/(2N-2)}$
(in particular $\rho^t$ fixes the points $-1$ and $0$), maps radial lines to radial
lines outside of a compact subset, and moves all radial lines other than
the negative real axis in the counterclockwise direction. We will furthermore assume
that the flow rotates the tangent vector to $\gamma_0$ at $-1$ (the
imaginary axis) counterclockwise, so that 
\begin{enumerate}
\item for $t\neq 0$,
$\gamma_t=\rho^t(\gamma_0)$ intersects $\gamma_0$ transversely at $-1$,
\item $\gamma_0\cap \gamma_t=\{-1\}$ for $|t|\in (0,t_0)$, where $t_0$ 
is the value of $t$ for which $\rho^{t}$ pushes the ray
$e^{-i\theta_0}\R_+$ past $e^{i\theta_0}\R_+$, 
and 
\item for $|t|>t_0$, $\gamma_0$ and $\gamma_t$
intersect transversely in exactly two points ($-1$ and one other intersection). 
\end{enumerate}
(These requirements on $\gamma_0\cap \gamma_t$ are natural and easy to 
achieve given the other requirements on $\rho^t$; see Figure \ref{fig:L0}).

Since the arcs $\gamma_t$ are strictly radial outside of a bounded subset,
their mutual intersections, and the bounded polygonal 
regions they delimit in the complex plane are all contained within a
bounded subset, say the disc of radius $R_0$.
For $R\in \R_{\geq 0}$, let $r(R)$ be the maximum of $h$ on the compact
subset $\Omega\cap \{|W|\leq \max(R,R_0)\}$ of $Y$. Then $r$ is a non-decreasing
function, constant over $[0,R_0]$, and we take the closed subset $Y^{in}\subset Y$ appearing in
Section \ref{sec:LGmodel-setup} to be the set of points of $Y$ where 
$h\leq r(|W|)$. This ensures that $Y^{in}$ contains $\Omega$.

Finally, we take the almost-complex structure $J$ to be the standard
complex structure of $Y$ outside of the bounded subset 
\begin{equation}\label{eq:W<epsilon}
Y^{in}\cap \{|W|<\epsilon\}
\end{equation} for some
$\epsilon\ll 1$ (smaller than the radius of $\Delta'$), and a generic small perturbation of the standard complex
structure inside that subset. This ensures that simple $J$-holomorphic spheres which
intersect this subset are regular, and evaluation maps for rigid somewhere
injective discs and spheres are mutually transverse, as explained in Remark 
\ref{rmk:genericJtransverse}.

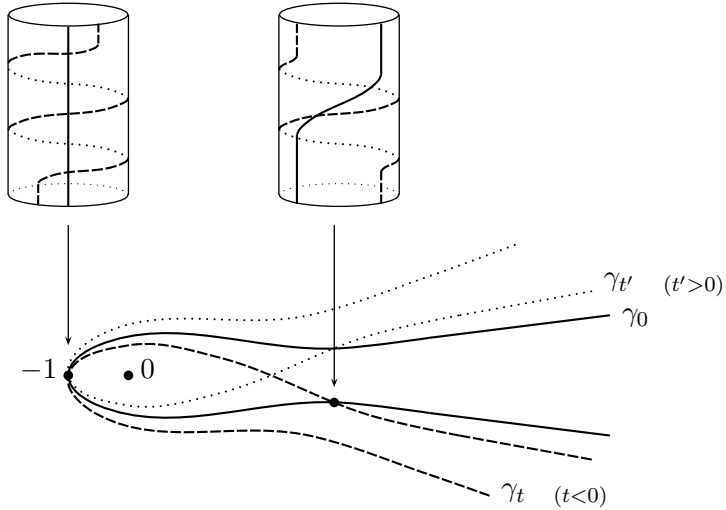
\begin{figure}[t]
\setlength{\unitlength}{8mm}
\begin{picture}(10,8.5)(-2,-2)
\psset{unit=\unitlength,linewidth=0.8pt,dotsep=2pt,dash=4pt 1pt}
\pscircle*(0,0){0.08}
\pscircle*(-1,0){0.08}
\pscurve(5,0.625)(3.5,0.45)(-0.5,0.5)(-1,0)(-0.5,-0.5)(3.5,-0.45)(5,-0.625)
\psline(5,-0.625)(8,-1)
\psline(5,0.625)(8,1)
\pscurve[linestyle=dotted](5.5,1)(4,0.66)(1.5,-0.3)(0,-0.5)(-1,0)(-0.6,0.6)(3,1)(5,1.65)(6.5,2.2)
\psline[linestyle=dotted](5.5,1)(7.7,1.4)
\pscurve[linestyle=dashed](5.5,-1)(4,-0.66)(1.5,0.3)(0,0.5)(-1,0)(-0.6,-0.6)(3,-1)(5,-1.65)(6,-2)
\psline[linestyle=dashed](5.5,-1)(7.7,-1.4)
\put(8.2,0.9){\small $\gamma_0$}
\put(7.9,1.5){\small $\gamma_{t'\quad (t'>0)}$}
\put(6.2,-2){\small $\gamma_{t\quad (t<0)}$}
\put(-1.8,-0.05){\small $-1$}
\put(0.2,-0.05){\small $0$}
\pscircle*(3.42,-0.45){0.08}
\psline[linewidth=0.5pt]{->}(-1,2.5)(-1,0.5)
\psline[linewidth=0.5pt]{->}(3.42,2.5)(3.42,-0.2)
\psellipticarc[linewidth=0.5pt](-1,3)(1,0.2){180}{360}
\psellipticarc[linewidth=0.5pt,linestyle=dotted](-1,3)(1,0.2){0}{180}
\psellipse[linewidth=0.5pt](-1,6)(1,0.2)
\psline[linewidth=0.5pt](-2,6)(-2,3)
\psline[linewidth=0.5pt](0,6)(0,3)
\psellipticarc[linewidth=0.5pt](3.5,3)(1,0.2){180}{360}
\psellipticarc[linewidth=0.5pt,linestyle=dotted](3.5,3)(1,0.2){0}{180}
\psellipse[linewidth=0.5pt](3.5,6)(1,0.2)
\psline[linewidth=0.5pt](2.5,6)(2.5,3)
\psline[linewidth=0.5pt](4.5,6)(4.5,3)
\psline(-1,5.8)(-1,2.8)
\psline[linestyle=dashed](-0.5,5.85)(-0.5,5.52)
\pscurve[linestyle=dashed](-0.5,5.5)(-0.55,5.45)(-1.95,5.25)(-2,5.2)
\pscurve[linestyle=dotted](-2,5.1)(-1.95,5.05)(-0.05,4.65)(0,4.6)
\pscurve[linestyle=dashed](-2,4.1)(-1.95,4.15)(-0.05,4.55)(0,4.6)
\pscurve[linestyle=dotted](-2,4.1)(-1.95,4.05)(-0.05,3.65)(0,3.6)
\pscurve[linestyle=dashed](0,3.6)(-0.05,3.55)(-1.45,3.25)(-1.5,3.2)
\psline[linestyle=dashed](-1.5,2.85)(-1.5,3.18)
\psline(4.2,5.85)(4.2,5)
\pscurve(4.2,5)(4.15,4.9)(2.85,4.1)(2.8,4)
\psline(2.8,4)(2.8,2.85)
\psline[linestyle=dashed](2.8,5.85)(2.8,5.3)
\pscurve[linestyle=dashed](2.8,5.3)(2.75,5.25)(2.55,5.15)(2.5,5.1)
\pscurve[linestyle=dotted](2.5,5.1)(2.55,5.05)(4.45,4.65)(4.5,4.6)
\pscurve[linestyle=dashed](2.5,4.1)(2.55,4.15)(4.45,4.55)(4.5,4.6)
\pscurve[linestyle=dotted](2.5,4.1)(2.55,4.05)(4.45,3.65)(4.5,3.6)
\pscurve[linestyle=dashed](4.5,3.6)(4.45,3.55)(4.25,3.45)(4.2,3.4)
\psline[linestyle=dashed](4.2,2.85)(4.2,3.4)
\end{picture}
\caption{The Lagrangians $L_0$ and $L_0(t)=\phi^{t}\rho^{t}(L_0)$ ($t<0$).}
\label{fig:L0}
\end{figure}

\begin{proposition}
The above geometric data on $Y$\! satisfy the requirements listed in Section
\ref{sec:LGmodel-setup}, and the Lagrangian submanifolds
$L_0(t)=\phi^t\rho^t(L_0)$ are admissible in the sense of Definition
\ref{def:admissibleLagrangian}.
\end{proposition}

\proof
We start with the geometric conditions in Section \ref{sec:LGmodel-setup}.
First, the properness of $h=\max\{|z^\v|^{1/d(\v)}\}$ follows from that of
the map $(z^\v)_{\v\in\mathcal{V}}:Y\to \C^{|\mathcal{V}|}$ (item~(1) in
Definition~\ref{def:admissibility}).
Next, we have already seen in Proposition \ref{prop:wrapping_ham} that $H$ is
proper on every fiber of $W$, and its Hamiltonian flow preserves the
fibers of $W$, i.e.\ $dW(X_H)=0$. Thus, $H$ Poisson
commutes with the real and imaginary parts of $W$, whose Hamiltonian vector
fields span the horizontal distribution; it follows that $dH$ vanishes on
horizontal vector fields. Moreover, since $H$ is a function of the moment map coordinates only,
$X_H$ is in the span of the vector fields generating the toric action, hence
its flow preserves the norms of all toric monomials, and so $dh(X_H)=0$.

Next we consider the behavior of $h$ along the horizontal distribution --
or more precisely, by Remark \ref{rmk:maxprinciple_maxhv}, the behavior of
the term(s) $h_\v$ that achieve the maximum in $h=\max\{h_\v\}$.
By Proposition \ref{prop:invce_Y}, for each $\v\in \mathcal{V}$, and at 
every point of $C_\v$ which lies outside of $Y^{in}\cup W^{-1}(\Delta')$, 
the monomial $z^\v$ is invariant under parallel transport. Therefore,
$dh_\v=\frac{1}{d(\v)}h_\v\, d\/\log|z^\v|$ and 
$d^c h_\v=\frac{1}{d(\v)}h_v\, d\/\arg(z^\v)$ both vanish on horizontal
vectors, and their Lie derivatives along horizontal vector fields also
vanish. It then follows from Proposition \ref{prop:largest_invce_Y} that,
everywhere outside of $Y^{in}\cup W^{-1}(\Delta')$,
these properties hold for any $h_\v$ that achieves the maximum in
$h=\max\{h_\v\}$. 

Finally, Proposition \ref{prop:wrapping_flatness}
implies that the flow of $X_H$ rotates $z^\v$ uniformly at a rate of $d(\v)$
at every point of $C_\v$ which lies outside of $Y^{in}$. Therefore, 
$dh_\v(X_H)=0$, $d^c h_\v(X_H)=\frac{1}{d(\v)} h_v\, d\arg(z^\v)(X_H)=h_\v\geq 0$,
and $\cL_{X_H}(d^c h_\v)=0$. As before, these properties hold everywhere in
$Y\setminus Y^{in}$ for any $h_\v$ that achieves the maximum in $h$.
This completes the verification of the requirements listed in Section
\ref{sec:LGmodel-setup}.

Next we prove the admissibility of $L_0$ in the sense of Definition
\ref{def:admissibleLagrangian}. The construction of the U-shaped arc
$\gamma_0$ ensures that its two halves connecting $-1$ to infinity are
admissible arcs in the sense of Definition \ref{def:admissiblearc}. 
The monomial admissibility of $\ell_0=(\R_+)^n\subset
W^{-1}(-1)$ and the compatibility of parallel transport with fiberwise monomial
admissibility (Corollary \ref{cor:invce_Y}) imply that $L_0$ is fiberwise
monomially admissible; therefore, $\arg(z^\v)$ vanishes identically on the
portion of $L_0$ which lies in $C_\v\setminus (C_\v\cap Y^{in})$, which in
turn implies the vanishing of $d^c h_\v=\frac{1}{d(\v)}h_\v d\arg(z^\v)$.
It follows that the restriction of $d^c h$ to $L_0$ vanishes outside of
$Y^{in}$ (wherever $h$ is differentiable, and otherwise in the sense of
Remark \ref{rmk:maxprinciple_maxhv}).

Since $L_0(t)=\phi^t\rho^t(L_0)$ is obtained from the admissible Lagrangian
$L_0$ by the admissible lifted isotopy $\rho^t$ and the flow of the wrapping 
Hamiltonian $H$, it is also admissible by Lemma
\ref{l:admissible_invariant}. (Alternatively,
$\ell_0(t)=\phi^t(\ell_0)\subset W^{-1}(-1)$ is monomially admissible by
Proposition \ref{prop:wrapping_flatness}, and the two portions of the arc
$\gamma_t=\rho^t(\gamma_0)$ connecting $-1$ to infinity are admissible in
the sense of Definition \ref{def:admissiblearc}; since $L_0(t)$ is obtained
by parallel transport of $\ell_0(t)$ over $\gamma_t$, its admissibility
follows from the same argument as above.)
\endproof

\subsection{The Floer complex $CF^*(L_0(t'),L_0(t))$}\label{ss:CFL0prelim}

Recall that $L_0(t)$ is fibered over $\gamma_t$, and fiberwise monomially
admissible with phase angles $\varphi_\v=d(\v)t$ (by Proposition
\ref{prop:wrapping_flatness}). Thus, the asymptotic directions of the
noncompact ends of $L_0(t)$ and $L_0(t')$ are disjoint whenever
$t'-t\in U=\R\setminus (\{\pm t_0\} \cup 
\frac{2\pi}{d_0}\Z)$, where we denote by $d_0$ the least common
multiple of the positive integers $d_\v$, $\v\in \mathcal{V}$. Since the arcs
$\gamma_t$ are strictly radial outside of the disc of radius $R_0$, and
monomial admissibility precludes the existence of intersections outside of $Y^{in}$ when the
phase angles are distinct,
for $t'-t\in U$ all the intersections of $L_0(t)$ and $L_0(t')$ lie within the
compact subset $Y^{in}\cap \{|W|\leq R_0\}$. 

The intersections of $L_0(t)$ and $L_0(t')$ are concentrated in the fibers
of $W$ above the intersection points of $\gamma_t$ and $\gamma_{t'}$;
we will now see that Lagrangian Floer theory for these submanifolds can be 
expressed in terms of the fiberwise Floer complexes in those fibers and 
counts of holomorphic sections of $W:Y\to \C$ over regions of the complex plane
delimited by the arcs $\gamma_t$ and $\gamma_{t'}$.

Because our construction of the wrapping Hamiltonian does not guarantee that
$L_0(t')$ and $L_0(t)$ intersect transversely, we will allow ourselves to
modify our Lagrangians by small Hamiltonian isotopies supported inside $Y^{in}$ 
(and preserving the fibers of $W$, so that admissibility is not affected)
in order to achieve transversality of intersections; we will see in the next sections that
our main calculation reduces to a cohomology-level argument, so we do not
specify the exact choice of perturbation involved in the definition of the
Floer complex.

For $t'-t\in \R_+\cap U$, we denote by $C_0(t',t)$ the portion of the Floer
complex $CF^*(L_0(t'),L_0(t))$ generated by intersection points which lie in the 
fiber $W^{-1}(-1)$, i.e.\ the Floer complex of the monomially admissible 
Lagrangian submanifolds
$\ell_0(t')=\phi^{t'}(\ell_0)$ and $\ell_0(t)=\phi^t(\ell_0)$ inside 
$W^{-1}(-1)\simeq (\C^*)^n$ (possibly after a small compactly supported
perturbation to achieve transversality). We similarly denote by $C_1(t',t)$ 
the portion of the Floer complex generated by intersection points which lie in the fiber above
the other intersection point $c_{t',t}$ of $\gamma_{t'}$ and $\gamma_t$ for
$t'-t>t_0$; this amounts to the Floer complex of the monomially admissible
Lagrangian submanifolds $\ell_-(t')$ and $\ell_+(t)$ of $W^{-1}(c_{t',t})$
obtained by parallel transport of $\ell_0(t')$ and $\ell_0(t)$ along the
portions of $\gamma_{t'}$ and $\gamma_t$ which run from $-1$ to $c_{t',t}$
(clockwise on $\gamma_{t}$, and counterclockwise on $\gamma_{t'}$).
For $t'-t<t_0$ we set $C_1(t',t)=0$. 

The choice of a grading (for instance the usual one) on 
$\ell_0=(\R_+)^n\subset (\C^*)^n$ and on the arc $\gamma_0$ in the complex
plane induces a grading on the admissible Lagrangian $L_0$, and also, by
following the various isotopies, on the monomially admissible Lagrangians 
$\ell_0(t)$ and their images under parallel transport, as well as $L_0(t)$.
We view $C_0(t',t)$ and $C_1(t',t)$ as the Floer complexes of the monomially
admissible Lagrangian submanifolds $\ell_0(t'),\ell_0(t)$ (resp.\
$\ell_-(t'),\ell_+(t)$) of $W^{-1}(-1)$ and $W^{-1}(c_{t',t})$,
with the grading induced by that of $\ell_0$; in the case of $C_0(t',t)$
this coincides with the grading of $CF^*(L_0(t'),L_0(t))$, but in the case
of $C_1(t',t)$ the grading in $CF^*(L_0(t'),L_0(t))$ is one less than the
fiberwise degree, due to the phase angles of the arcs $\gamma_t,\gamma_{t'}$ 
at $c_{t',t}$ differing by an amount in the interval $(\pi,2\pi)$ for
$t'-t>t_0$ (see Figure \ref{fig:L0}). Thus,
\begin{align}
\nonumber CF^*(L_0(t'),L_0(t))&=C_0(t',t)\oplus C_1(t',t)[1]\\
&=\begin{cases} 
CF^*(\ell_0(t'),\ell_0(t))\oplus CF^*(\ell_-(t'),\ell_+(t))[1] & (t'-t>t_0)\\
CF^*(\ell_0(t'),\ell_0(t)) & (0<t'-t<t_0).
\end{cases}
\end{align}

Because the almost-complex structure $J$ coincides with the standard one
outside of the subset $Y^{in}\cap \{|W|<\epsilon\}$ introduced in
\eqref{eq:W<epsilon}, $J$-holomorphic curves satisfy the open mapping principle
with respect to the projection $W:Y\to\C$ and intersect positively with the
fibers of $W$ outside of the disc of radius $\epsilon$. (However this fails
near the origin.) This implies immediately that $J$-holomorphic discs with
boundary on a union of fibered Lagrangian submanifolds (disjoint from
the region where $|W|<\epsilon$) are either contained in the fibers of $W$,
or behave (away from the zero fiber) like sections or multisections of $W:Y\to\C$ over 
regions of the complex plane delimited by the arcs over which the
Lagrangians fiber. By abuse of terminology, we call such $J$-holomorphic discs 
``sections'' when their intersection number with the fibers is one, even though they need not be 
genuine sections over the disc of radius $\epsilon$.

The fibers of $W$ outside of the origin are isomorphic to $(\C^*)^n$,
and the monomially admissible Lagrangians $\ell_0(t)$ and their images under
parallel transport do not bound any holomorphic discs inside the fibers of
$W$ (e.g.\ because they are contractible and hence exact). It follows
that $L_0(t)$ does not bound any $J$-holomorphic discs. 

Moreover, our choice of $J$ ensures that we can also avoid sphere bubbling
by arguing as in Remark \ref{rmk:genericJtransverse}. Since the 
intersections of $L_0(t')$ and $L_0(t)$ lie within the region of $Y$ where
$|W|\leq R_0$ and $h\leq r(R_0)$, the maximum principles for $W$ and $h$
(Propositions \ref{prop:maxprinciple_base} and \ref{prop:maxprinciple_fiber})
imply that the $J$-holomorphic discs contributing to the Floer differential
(and later on, to continuation maps or product operations) also remain within $Y^{in}\cap
\{|W|\leq R_0\}$. Since the fibers of $W$ away from the origin are
aspherical, the only possible sphere bubbles are configurations contained in
the region where $|W|\leq \epsilon$, at least one component of which must
pass within $Y^{in}\cap \{|W|<\epsilon\}$.  The choice of a generic
perturbation of the standard complex structure within this subset ensures that 
the underlying simple spheres are disjoint from all $J$-holomorphic discs
in the 0- or 1-dimensional moduli spaces we consider, and hence that
no sphere bubbles can form.

We can now state and prove the main result of this section, which describes the
structure of the Floer differential on $CF^*(L_0(t'),L_0(t))$.

\begin{proposition}\label{prop:CF_is_cone}
For $0<t'-t<t_0$, the Floer complex $CF^*(L_0(t'),L_0(t))$ in $Y$ is isomorphic
to the Floer complex $CF^*(\ell_0(t'),\ell_0(t))$ in $W^{-1}(-1)\simeq
(\C^*)^n$. 

For $t'-t>t_0$, $CF^*(L_0(t'),L_0(t))$ is isomorphic to the
mapping cone 
\begin{equation}\label{eq:CF_is_cone}
CF^*(\ell_0(t'),\ell_0(t))\oplus
CF^*(\ell_-(t'),\ell_+(t))[1],\qquad \partial=\begin{pmatrix}
\partial_0 & s \\ 0 & \partial_1\end{pmatrix}
\end{equation}
where the diagonal entries are the Floer differentials on the fiberwise
Floer complexes, and the off-diagonal term
\begin{equation}\label{eq:s_t't}
s=s^0_{\ell_0,t',t}:CF^*(\ell_-(t'),\ell_+(t))\to
CF^*(\ell_0(t'),\ell_0(t))
\end{equation} is a chain map defined by a (weighted) count
of $J$-holomorphic sections of $W:Y\to\C$ over 
the bounded region of the complex plane delimited by $\gamma_t$ and
$\gamma_{t'}$ (cf.\ Figure \ref{fig:L0}).
\end{proposition}

\proof 
The open mapping principle implies that the $J$-holomorphic discs that 
contribute to the Floer differential on $CF^*(L_0(t'),L_0(t))$ are either
contained within the fibers of $W$, or (for $t'-t>t_0$) sections of $W$ over
the bounded region of the complex plane delimited by $\gamma_t$ and
$\gamma_{t'}$.
The contributions of discs contained within $W^{-1}(-1)$ and $W^{-1}(c_{t',t})$
correspond exactly to the Floer differentials on the
fiberwise Floer complexes $C_0(t',t)=CF^*(\ell_0(t'),\ell_0(t))$ and 
$C_1(t',t)=CF^*(\ell_-(t'),\ell_+(t))$, while the sections
contribute the off-diagonal term $s$. The fact that $s$ is a chain map
follows directly from the vanishing of the square of the Floer differential.
\endproof

It follows that the Floer cohomology group $HF^*(L_0(t'),L_0(t))$ is
isomorphic to $HF^*(\ell_0(t'),\ell_0(t))$ for $0<t'-t<t_0$, while for
$t'-t>t_0$ it is determined by the map induced by $s$ on cohomology,
which we again denote by
\begin{equation}
\label{eq:s_t't_hf}
s=s^0_{\ell_0,t',t}:HF^*(\ell_-(t'),\ell_+(t))\to HF^*(\ell_0(t'),\ell_0(t)).
\end{equation}

Even though the Floer complexes and the chain map \eqref{eq:s_t't} depend
on the choices made in the construction, the maps constructed from different 
choices are related by homotopies, so that the cohomology-level map \eqref{eq:s_t't_hf} is
independent of choices.

Indeed, deformations of Floer data among the set of choices which
satisfy our technical requirements (e.g.\ compactly
supported fiberwise Hamiltonian isotopies, modifications of $J$ near
$W^{-1}(0)$, or even admissible isotopies of the arcs
$\gamma_t$, $\gamma_{t'}$ which preserve transversality at all times)
induce continuation quasi-isomorphisms
on the Floer complexes \eqref{eq:CF_is_cone}. In every instance,
by considering the projection $W:Y\to\C$ one shows that
continuation trajectories, just like contributions to the Floer differential, can map 
generators in $W^{-1}(c_{t',t})$ to generators in $W^{-1}(-1)$ but
not vice-versa. Thus, our continuation homomorphisms are upper-triangular with respect to the decomposition
\eqref{eq:CF_is_cone} and induce quasi-isomorphisms on the summands
$C_0=CF^*(\ell_0(t'),\ell_0(t))$ and
$C_1=CF^*(\ell_-(t'),\ell_+(t))$.
Denoting by $C_0$ and $C_1$ the two summands in
\eqref{eq:CF_is_cone} with respect to one set of choices, and $C'_0$ and
$C'_1$ the two summands for the other set of choices, we obtain a diagram

\begin{equation}\label{eq:continuationcone}
\begin{psmatrix}[colsep=1.5cm,rowsep=1.5cm]
C_1 & C_0 \\
C'_1 & C'_0
\end{psmatrix}
\psset{nodesep=5pt,arrows=->}
\ncline{1,1}{2,1} \tlput{f_1}
\ncline{1,1}{1,2} \taput{s}
\ncline{2,1}{2,2} \tbput{s'}
\ncline{1,1}{2,2} \taput{\makebox(0,0)[lc]{\ $h$}}
\ncline{1,2}{2,2} \trput{f_0}
\end{equation}
\vskip3mm
\noindent
where $f_0,f_1,h$ are the components of the continuation homomorphism,
and $f_0:(C_0,\partial_0)\to (C'_0,\partial'_0)$ and $f_1:(C_1,\partial_1)\to 
(C'_1,\partial'_1)$ are quasi-isomorphisms. 

The fact that the continuation homomorphism is a chain map implies that
$$f_0\circ s + h\circ \partial_1=s'\circ f_1+\partial'_0\circ h.$$
Therefore $f_0\circ s$ and $s'\circ f_1$ are homotopic, and so
the cohomology level maps induced by $s$ and $s'$ coincide under the
isomorphisms $H^*(C_1,\partial_1)\simeq H^*(C'_1,\partial'_1)$ and
$H^*(C_0,\partial_0)\simeq H^*(C'_0,\partial'_0)$ induced by $f_1$ and
$f_0$. In this sense, the map \eqref{eq:s_t't_hf} is independent of the
choices made in the construction and invariant under admissible isotopies.

To put this in proper context, the map $s$ is part of the ``Seidel TQFT''
(cf.\ \cite{SeBook}) associated to the symplectic fibration $W:Y\to\C$.
As a general principle,
counts of $J$-holomorphic sections over given 
domains in the complex plane with boundary on given fibered Lagrangian submanifolds 
give rise to maps between the respective fiberwise Floer complexes which are independent of choices up to homotopy
and satisfy algebraic relations that can be understood in terms of gluing
axioms (we shall not elaborate on the latter point here; see \cite{SeBook}
for details).

\subsection{Floer cohomology for monomially admissible Lagrangians in $(\C^*)^n$}
\label{ss:fiberwise-floer}

To proceed further, we need to discuss Floer theory for monomially
admissible Lagrangian submanifolds in the fibers of $W$, which we identify
with $(\C^*)^n$ by considering the toric monomials $z_1,\dots,z_n$ on the
open stratum of $Y$ whose weights correspond to first $n$ basis vectors.
(So, for each $\v=(\vec{v},v^0)=(v_1,\dots,v_n,v^0)\in \mathcal{V}$, the
monomial $z^\v$ restricts to $W^{-1}(c)$ as $(-c)^{v^0}z_1^{v_1}\dots z_n^{v_n}$).
The material in this section closely parallels Hanlon's work \cite[Section
3.4]{Hanlon}.

The moment map for the standard $\T^n$-action on $W^{-1}(c)\simeq (\C^*)^n$
is given by the first $n$ coordinates $(\xi_1,\dots,\xi_n)$ of the moment
map of $Y$, and for each $\v\in \mathcal{V}$
the intersection of $C_\v$ with $W^{-1}(c)$ 
is the inverse image under the moment map of the subset 
$S_\v=S_{\v,\gamma}\subset \R^n$ defined by \eqref{eq:Svgamma}
(for an appropriate value of the constant $\gamma>0$, matching that used
for $C_\v$ at the beginning of \S \ref{ss:L0setup}).

We consider Lagrangian submanifolds of $(\C^*)^n$ which are sections over the moment map
projection; any such Lagrangian is the graph of the differential of a function
$K=K(\xi):\R^n\to \R$, i.e.\ the arguments $\arg(z_j)=\theta_j$ are
determined as functions of the moment map variables $(\xi_1,\dots,\xi_n)$
by $\theta_j=\partial K/\partial \xi_j$. (For a given Lagrangian, $K$ is
unique up to an affine function whose gradient is $2\pi$ times an integer
vector.) The monomial admissibility
condition can then be expressed in terms of the gradient of~$K$: 
the graph $\ell=\Gamma_{dK}\subset W^{-1}(c)\simeq (\C^*)^n$ is monomially admissible with phase angles
$\{\varphi_\v\}$ if and only if, outside of a compact subset,
\begin{equation}\label{eq:mon_adm_gradient}
\langle \nabla K(\xi), \vec{v}\rangle \equiv \varphi_\v-v^0\arg(-c) \ \mod\ 2\pi\Z
\quad \forall \xi \in S_\v,\quad \forall \v=(\vec{v},v^0)\in \mathcal{V}.
\end{equation}

\begin{definition}\label{def:mon_adm_slope}
The {\em slope} of the monomially admissible Lagrangian section
$\ell=\Gamma_{dK}$ is
the tuple $\sigma(K)=(\sigma_\v(K))_{\v\in \mathcal{V}}\in
\R^{|\mathcal{V}|}$, where 
$\sigma_\v(K)=\langle \nabla K(\xi),\vec{v}\rangle_{|S_\v}$.

When $K$ is a convex function, we associate to its slope $\sigma=\sigma(K)$ the polytope
\begin{equation}\label{eq:mon_adm_polytope}
P(\sigma)=\{u\in \R^n\,|\,\langle u,\vec{v}\rangle\leq \sigma_\v
\ \ \forall \v=(\vec{v},v^0)\in \mathcal{V}\}.
\end{equation}
\end{definition}

Recall that the vectors $\vec{v}$ appearing in \eqref{eq:mon_adm_polytope} are
the primitive normal vectors to the facets of the Newton polytope $P$ associated to the Laurent
polynomial $f$ (cf.\ Definition \ref{def:extremalv}). Given any vertex $\alpha \in
P_\Z$ of $P$, the subsets $S_\v$ 
associated to the various facets of $P$ which meet at $\alpha$ have a non-empty
and unbounded intersection $U_\alpha$ (comprising most of the region of $\R^n$ where 
$\alpha$ achieves the maximum in the tropicalization of $f$, cf.\ Figure
\ref{fig:Svgamma}). Over $U_\alpha$, the value of $\nabla K$ is 
fully constrained by the slope $\sigma=\sigma(K)$, since
$\langle \nabla K,\vec{v}\rangle =\sigma_\v$
whenever $\vec{v}$ is the normal vector to a facet of $P$ 
containing $\alpha$. This corresponds to the equality
case in the inequalities \eqref{eq:mon_adm_polytope} for a maximal
collection of linearly independent $\vec{v}$, i.e.\ a vertex of the polytope
$P(\sigma)$. From this and standard facts about convex functions we deduce:

\begin{lemma}\label{l:mon_adm_polytope}
If $K$ is convex with slope $\sigma$, then $P(\sigma)$ is a
convex polytope with the same normal vectors and normal fan as $P$, and
the range of values taken by the gradient $\nabla K$ is exactly
$P(\sigma)$.
\end{lemma}

\begin{example}\label{ex:slopes_L0}
The monomially admissible section $\ell_0(t)=\phi^t(\ell_0)\subset
W^{-1}(-1)$ is the graph
of $d(tH)$, so by Proposition \ref{prop:wrapping_flatness} and
\eqref{eq:mon_adm_gradient} its
slope is
\begin{equation}\label{eq:slope_l0}
\sigma_0(t):=\sigma(tH)=(\sigma_\v(tH))_{\v\in \mathcal{V}}=(t\,d(\v))_{\v\in
\mathcal{V}}.
\end{equation}
Moreover, for $t'-t>t_0$, parallel transport of $\ell_0(t)$ and $\ell_0(t')$ 
from $W^{-1}(-1)$ to $W^{-1}(c_{t',t})$ along the relevant portions of
$\gamma_t$ and $\gamma_{t'}$ preserves the phase angles
$\varphi_\v=t\,d(\v)$, so by \eqref{eq:mon_adm_gradient} the monomially
admissible Lagrangian sections $\ell_-(t')$ and $\ell_+(t)$ in $W^{-1}(c_{t',t})$ 
have slopes
\begin{eqnarray}
\label{eq:slope_l-}
\sigma_-(t')&=&(t'd(\v)-v^0\,(\arg(c_{t',t})+\pi))_{\v\in \mathcal{V}}
\quad \text{and}\\
\label{eq:slope_l+}
\sigma_+(t)&=&(t\,\,d(\v)-v^0\,(\arg(c_{t',t})-\pi))_{\v\in\mathcal{V}},
\end{eqnarray}
where we take $\arg(c_{t',t})\in (-\pi,\pi)$; the values of 
$\arg(-c_{t',t})$ in these two formulas differ by $2\pi$ because
we consider parallel transport from $-1$ to $c_{t',t}$ clockwise around the
origin for $\ell_+(t)$ and counterclockwise for $\ell_-(t')$.
\end{example}

Let $\ell$ and $\ell'$ be two monomially admissible Lagrangian sections,
expressed as the graphs of $dK$ and $dK'$. If the slopes of $K$ and $K'$
differ by amounts that aren't multiples of $2\pi$, 
then the intersections of $\ell$ and $\ell'$ remain within a compact
subset of $(\C^*)^n$, and their Floer cohomology is well-defined.
We claim that $HF^*(\ell',\ell)$ only depends on the slopes. (As we shall
see in the argument below, this is an instance of the invariance of Floer cohomology under 
Hamiltonian isotopies which preserve admissibility and disjointness at infinity, and follows from the
existence of well-defined continuation maps;
see \cite[Lemma 3.21]{GPS1} for the analogous result in the setting
of Liouville sectors.)

\begin{proposition}\label{prop:mon_adm_hf}
Let $\ell=\Gamma_{dK}$ and $\ell'=\Gamma_{dK'}$ be two monomially admissible 
Lagrangian sections, with slopes $\sigma(K)=\sigma$ and $\sigma(K')=\sigma'$, 
and assume that $\sigma_\v-\sigma'_\v\not\in 2\pi\Z$ $\forall \v\in
\mathcal{V}$. Then $HF^*(\ell',\ell)$ only depends on the slopes
$\sigma$ and $\sigma'$ of $K$ and $K'$. Moreover, if $K'-K$ is convex then the Floer
cohomology is concentrated in degree zero and 
\begin{equation}\label{eq:mon_adm_hf}
HF^0(\ell',\ell)\cong \bigoplus_{p\in P(\sigma'-\sigma)\cap (2\pi \Z)^n} \K\cdot
\vartheta_p.
\end{equation}
\end{proposition}

\proof First we prove invariance. Given any two Hamiltonians $K_0,K_1$ with
$\sigma(K_0)=\sigma(K_1)=\sigma$, the convex combinations
$K_s=(1-s)K_0+sK_1$ also have slope $\sigma$, and the graphs
$\ell_s=\Gamma_{dK_s}$ are monomially admissible Lagrangian sections.
Similarly, given $K'_0,K'_1$ with $\sigma(K'_0)=\sigma(K'_1)=\sigma'$,
we set $K'_t=(1-s)K'_0+sK'_1$ and $\ell'_s=\Gamma_{dK'_s}$.  We 
then define continuation maps 
$$\Phi_{01}:CF^*(\ell'_0,\ell_0)\to CF^*(\ell'_1,\ell_1)\quad
\text{and}\quad \Phi_{10}:CF^*(\ell'_1,\ell_1)\to CF^*(\ell'_0,\ell_0)$$
by counting 
index zero $J$-holomorphic strips $u:\R\times [0,1]\to (\C^*)^n$ with moving 
boundary conditions given by $\ell'_s$ (for $s$ a suitable function of the
real coordinate) along $\R\times 0$ and $\ell_s$ along $\R\times 1$.

Since the slopes of $K_1-K_0$ and $K'_1-K'_0$ are all zero, the 
Hamiltonian vector fields $X=X_{K_1-K_0}$ and $X'=X_{K'_1-K'_0}$ (which generate the isotopies of the
moving boundary conditions $\ell_s$ and $\ell'_s$) satisfy
$d^c h(X)=d^c h(X')=0$ outside of a compact subset (where we recall that
$h=\max\{h_\v\}=\max\{|z^\v|^{1/d(\v)}\}$). More precisely, 
the vanishing of $\langle \nabla(K_1-K_0),\vec{v}\rangle$
and $\langle \nabla(K'_1-K'_0),\vec{v}\rangle$ implies the invariance of
the monomial $z^\v$ under the flows of $X$ and $X'$ at all points of $C_\v\cap W^{-1}(c)$ which lie outside of $Y^{in}$,
and hence the vanishing of $d^c h_\v(X)=\frac{1}{d(\v)}h_v
d\arg(z^\v)(X)$ and $d^c h_\v(X')$.

This in turn implies that $J$-holomorphic strips with moving boundary
conditions $\ell_s$ and $\ell'_s$ satisfy the maximum principle with 
respect to the proper function $h$ outside of a compact subset of
$(\C^*)^n$, and hence that the continuation maps $\Phi_{01}$ and $\Phi_{10}$ are well-defined. The argument is similar to 
the last part of the proof of Proposition 
\ref{prop:maxprinciple_fiber}: the vanishing of $d^c h$ on the tangent spaces to the
monomially admissible Lagrangians $\ell_s,\ell'_s$ and on the vector fields $X$ and $X'$ along
which these boundary conditions move implies that the restriction of
$d^c(h\circ u)$ to the boundary of the strip $\R\times [0,1]$ vanishes 
identically (outside of $u^{-1}(Y^{in})$), and the result then follows from
the maximum principle with Neumann boundary conditions. 

The usual argument for Floer continuation maps then shows that $\Phi_{01}$
and $\Phi_{10}$ are chain maps, and that $\Phi_{01}\circ \Phi_{10}$ and
$\Phi_{10}\circ \Phi_{01}$ are homotopic to identity; it follows that
$HF^*(\ell'_0,\ell_0)\simeq HF^*(\ell'_1,\ell_1)$.

We now turn to the second part of the statement. Assume that $K'-K$ is
convex, and observe that the generators of $CF^*(\ell',\ell)$ correspond to
points where $dK'$ and $dK$ differ by an integer multiple of $2\pi$, i.e.\ 
$\nabla(K'-K)\in (2\pi \Z)^n$. By Lemma \ref{l:mon_adm_polytope}, the
set of possible values of $\nabla(K'-K)$ is $P(\sigma'-\sigma)$. For
each $p\in P(\sigma'-\sigma)\cap (2\pi\Z)^n$, the function $K'(\xi)-K(\xi)-\langle
p,\xi\rangle$ is convex; up to a small perturbation (preserving convexity)
we can assume that its critical points are non-degenerate. Convexity then
ensures that the critical point (guaranteed to exist
by Lemma \ref{l:mon_adm_polytope}) is unique and a minimum, so that it
contributes a single generator to $CF^0(\ell',\ell)$, which (up to a
suitable rescaling, see below) we denote by
$\vartheta_p$. Taking the direct sum over all $p$, we find that the Floer 
complex $CF^*(\ell',\ell)=\mathrm{span}\,\{\vartheta_p,\ p\in P(\sigma'-\sigma)\cap
(2\pi\Z)^n\}$ is concentrated in degree zero, which in turn implies the
vanishing of the Floer differential, and \eqref{eq:mon_adm_hf} follows.
\endproof

As a general convention, we rescale all generators of the Floer complexes
for monomially admissible Lagrangian sections by their action (suitably
defined, see below), using the exactness of these Lagrangians to
eliminate geometrically irrelevant powers of the Novikov variable and ensure
that continuation isomorphisms map generators to generators. In the
setting of Proposition \ref{prop:mon_adm_hf}, given $\ell=\Gamma_{dK}$ and
$\ell'=\Gamma_{dK'}$ with $K'-K$ convex and $p\in P(\sigma'-\sigma)\cap
(2\pi\Z)^n$, and denoting by $\xi_p$ the critical
point of $K'-K-\langle p,\cdot\rangle$, we define the {\em action} of this
intersection point to be the associated critical value of $K'-K-\langle
p,\cdot\rangle$, and the generator we denote by
$\vartheta_p$ is
actually $t^{K'(\xi_p)-K(\xi_p)-\langle p,\xi_p\rangle}$ times the standard
generator associated to the intersection point. (Of note: our basis depends
not only on the Lagrangians $\ell$ and $\ell'$ but also on the 
normalizations of $K$ and $K'$; different choices yield
differently scaled bases, which can be related explicitly by isomorphisms
mapping each generator to a power of $t$ times a generator.)

\begin{proposition}\label{prop:mon_adm_hf_product}
Let $\ell=\Gamma_{dK}$, $\ell'=\Gamma_{dK'}$, and $\ell''=\Gamma_{dK''}$ be
three monomially admissible Lagrangian sections, such that none of the pairwise
differences of their slopes $\sigma,\sigma',\sigma''$ is a multiple of
\,$2\pi$.  Assume moreover that $K''-K'$ and $K'-K$ are convex.
Then for any
$p\in P(\sigma'-\sigma)\cap (2\pi\Z)^n$ and $p'\in P(\sigma''-\sigma')\cap
(2\pi\Z)^n$, the Floer product of $\vartheta_p\in HF^0(\ell',\ell)$ and
$\vartheta_{p'}\in HF^0(\ell'',\ell')$ is given by
\begin{equation}\label{eq:mon_adm_hf_product}
\vartheta_{p}\cdot \vartheta_{p'}=\vartheta_{p+p'}\in HF^0(\ell'',\ell).
\end{equation}
\end{proposition}

\proof 
We lift $\ell,\ell',\ell''$ to the universal cover $T^*\R^n$ of $(\C^*)^n$ by
considering the graphs $\tilde\ell$, $\tilde\ell'$ and $\tilde\ell''$ of
$d(K+\langle p,\cdot\rangle)$, $dK'$, and $d(K''-\langle p',\cdot\rangle)$
respectively. By construction, the generator $\vartheta_p$ lifts to an
intersection point of $\tilde\ell$ and $\tilde\ell'$, and similarly
$\vartheta_{p'}$ lifts to an intersection of $\tilde\ell'$ with
$\tilde\ell''$.
Thus, any holomorphic disc in $(\C^*)^n$ contributing to
the Floer product of $\vartheta_p$ and $\vartheta_{p'}$ lifts to a
disc in the universal cover with boundary on $\tilde\ell$, $\tilde\ell'$ and
$\tilde\ell''$. It follows that the output of the disc corresponds to an intersection
of $\tilde\ell$ with $\tilde\ell''$, i.e.\ a critical point of $K''-K-\langle
p+p',\cdot\rangle$; hence
$\vartheta_p\cdot \vartheta_{p'}$ is a multiple of $\vartheta_{p+p'}$.

Denote by $\xi_p$, $\xi_{p'}$, and $\xi_{p+p'}\in \R^n$ the critical
points of the convex functions \hbox{$K'-K-\langle p,\cdot\rangle$,} $K''-K'-\langle p',\cdot\rangle$,
and $K''-K-\langle p+p',\cdot\rangle$ respectively. By Stokes' theorem, the
symplectic area of any holomorphic triangle contributing to the coefficient of
$\vartheta_{p+p'}$ in $\vartheta_p\cdot \vartheta_{p'}$ is equal to the difference of
the actions of the input and output generators, i.e.\ 
\begin{multline}
\label{eq:actiontriangle}
(K''(\xi_{p+p'})-K(\xi_{p+p'})-\langle p+p',\xi_{p+p'}\rangle)\\ 
-(K'(\xi_p)-K(\xi_p)-\langle p,\xi_p\rangle)
-(K''(\xi_{p'})-K'(\xi_{p'})-\langle p',\xi_{p'}\rangle).
\end{multline}
Thus, since our chosen bases of the Floer complexes are already rescaled by action, the
powers of $t$ cancel out and each holomorphic disc contributes $\pm 1$.

It remains to show that the overall count of discs is $+1$. Since our calculation
is at the level of Floer cohomology, the count we consider is homotopy invariant
and we can deform the Lagrangian submanifolds $\tilde\ell$, $\tilde\ell'$ and $\tilde\ell''$ 
to simplify the problem. We use the same trick as \cite[Proposition 3.22]{Hanlon},
and replace $K$ and $K''$ by modified functions $\hat{K}$ and $\hat{K}''$ such that
$$(K'-\hat{K})(\xi)=(K'-K)(\xi+\xi_p)\quad\text{and}\quad
(\hat{K}''-K')(\xi)=(K''-K)(\xi+\xi_{p'}).$$ This modification ensures that $K'-\hat{K}$ and
$\hat{K}''-K'$ remain convex and have the same slopes at infinity as
$K'-K$ and $K''-K$, but the critical points of
$K'-\hat{K}-\langle p,\cdot\rangle$ and
$\hat{K}''-K'-\langle p',\cdot\rangle$ now lie at the origin; considering
their sum, the critical point of $\hat{K}''-\hat{K}-\langle p+p',\cdot\rangle$
also lies at the origin. 
(A note of caution: modifying $K'-K$ and $K''-K'$ by 
translations in the $\xi$ coordinate in this manner doesn't quite preserve monomial admissibility, as
the control over $\arg(z^\v)$ is now achieved over a slightly smaller subset of 
$(\C^*)^n$; since the collection of these modified subsets still covers the
complement of a compact subset, this does not affect in any significant manner the maximum principle arguments we use
to control holomorphic curves.)
Thus we have reduced the problem to the case where 
$\tilde\ell$, $\tilde\ell'$ and $\tilde\ell''$ all intersect (transversely)
in a single point (near which they are the graphs of the differentials of functions
whose differences have non-degenerate minima). The formula \eqref{eq:actiontriangle} now shows that
any holomorphic disc contributing to the Floer product must have area zero,
i.e.\ the only contribution is from the constant map. By linearization and
reduction to a product setting, the constant disc
is easily checked to be regular and contribute $+1$ to the count (using
the preferred trivializations of the orientation lines at even degree
generators and the sign conventions from \cite[Section 13c]{SeBook}).
\endproof

Next, we consider continuation elements (quasi-units) for the action of the wrapping
Hamiltonian $H$ on monomially admissible Lagrangian sections in
$W^{-1}(-1)\simeq (\C^*)^n$. 
Recall that $H$ is proper and convex by Proposition
\ref{prop:wrapping_ham}; to simplify normalizations, we assume that its
minimum value is zero (otherwise the formula below should be corrected by
a factor of $t^{\tau\min H}$).

\begin{proposition}\label{prop:mon_adm_hf_quasiunit} 
Let $\ell=\Gamma_{dK}$ be a monomially admissible Lagrangian
section in $(\C^*)^n$, and $\ell'=\phi^\tau(\ell)=\Gamma_{d(K+\tau H)}$ its image
under the
time $\tau$ flow of the wrapping Hamiltonian $H$ for $\tau>0$ chosen so that $\tau d(\v)\not\in 2\pi
\Z$ $\forall\v\in\mathcal{V}$.  Then the quasi-unit $e=e_{\ell',\ell}\in
HF^0(\ell',\ell)$ is the generator $e=\vartheta_0$ corresponding to the
minimum of $H$.
\end{proposition}

\proof
As in Section \ref{ss:quasiunits} (now working in $(\C^*)^n$ rather than in
$Y$), the quasi-unit $e_{\ell',\ell}$ is defined by counting solutions to a
Cauchy-Riemann equation whose domain $\Sigma$ is a disc with a single output boundary puncture,
with moving boundary condition along $\partial \Sigma$ given by the images by $\ell$ under the
flow generated by $H$. Such a disc lifts to the universal cover $T^*\R^n$
as a disc whose output marked point maps to an intersection of the graphs
of $dK$ and $d(K+\tau H)$; it follows that $e$ is a multiple of $\vartheta_0$. 

The count of solutions to the Cauchy-Riemann 
equation is homotopy invariant, so we modify the setting slightly from
Section \ref{ss:quasiunits} in order to make it apparent that the
only contribution is from the constant solution at the point of 
$\ell$ where $H$ reaches its minimum. Denote by $\eta$ the
1-form on $\partial\Sigma$ (vanishing near the puncture) such that the 
variation of the boundary condition 
along $\partial\Sigma$ is induced by the flow of $X_H\otimes \eta$.
Then we consider the perturbed
Cauchy-Riemann equation 
\begin{equation}\label{eq:CRquasiunit}
(du-X_H\otimes \alpha)^{0,1}=0,
\end{equation}
where $\alpha$ is a sub-closed 1-form on $\Sigma$
($d\alpha\leq 0$) which vanishes in the output strip-like end and satisfies
$\alpha_{|\partial\Sigma}=\eta$.

As in \cite[Appendix B]{AbGen}, the geometric energy $$E_{geo}(u)=\int_\Sigma \|du-X_H\otimes \alpha\|^2=
\int_\Sigma u^*\omega-u^*(dH)\wedge \alpha$$ of a solution to
\eqref{eq:CRquasiunit} and the topological energy
$$E_{top}(u)=\int_\Sigma u^*\omega-d(u^*(H)\alpha)=E_{geo}(u)-\int_\Sigma u^*(H)\,d\alpha$$
satisfy $0\leq E_{geo}(u)\leq E_{top}(u)$ (since $H\geq 0$ and
$d\alpha\leq 0$). Denoting by $s$ a coordinate along $\partial\Sigma$ and
by $t(s)$ the function such that the boundary condition at $s$ is given
by $\phi^{t(s)}(\ell)=\Gamma_{K+t(s)H}$ (so $t(s)$ decreases from $\tau$ to
zero along the boundary, and its differential coincides with $\eta$), Stokes' theorem gives
$$E_{top}(u)=\int_{\partial\Sigma} -(u^*(dK)+t(s)u^*(dH))-u^*(H)\,\eta=
\int_{\partial\Sigma} -d(u^*K+t(s)u^*H)=\tau\,H_{out},$$
where $H_{out}$ is the value of $H$ at the output marked point, i.e.\ zero.
Thus any solution has vanishing geometric and topological
energies, i.e.\ it is a constant map at the point where $H$ reaches its
minimum. Moreover, the constant map is regular (using the fact that its
index equals the degree of the output generator, i.e.\ zero, and the
linearized Cauchy-Riemann operator is injective since essentially the same
argument as above shows that the energy of any element of the kernel must
be zero); thus the count of solutions is $\pm 1$. Since
the sign is independent of $\ell$ and $\tau$, it follows from the multiplicativity of
quasi-units ($e_{\ell'',\ell}=e_{\ell',\ell}\cdot e_{\ell'',\ell'}$, see
e.g.\ \cite[Proposition 3.15]{Hanlon})
that the sign is $+1$, and thus $e=\vartheta_0$. 
\endproof

Finally, we consider the Floer theory of admissible sections with Lagrangian tori, which will allow us in the next part to reduce Floer-theoretic computations involving non-compact Lagrangians to computations involving only tori. Given  $x=(x_1,\dots,x_n)\in (\K^*)^n$, we denote by
$\mathfrak{t}_{x}$ the 
Lagrangian torus $\{\xi\}\times T^n$ consisting of those points of $(\C^*)^n$ whose moment map
coordinates satisfy
$\xi_i=-\frac{1}{2\pi}\mathrm{val}(x_i)$ for all $i=1,\dots,n$, equipped with a rank one
unitary local system over $\K$ whose holonomy $y_i$ around the $i$-th
$S^1$ factor satisfies $x_i=t^{-2\pi\xi_i} y_i^{-1}$.
Given a Lagrangian section $\ell=\Gamma_{dK}$, the Floer complex
$CF^*(\ell,\mathfrak{t}_x)$ has rank one, and we denote by $\varepsilon_x$
a suitably rescaled generator: namely, we define $\varepsilon_x$ to be
$t^{K(\xi)}$ times
the element of the local system at the intersection point $(\xi,dK(\xi))$
obtained by parallel transport of a fixed element at $(\xi,0)$ from the
origin to $dK(\xi)$ along $\mathfrak{t}_x$.

\begin{proposition}\label{prop:mon_adm_hf_eval}
Let $\ell=\Gamma_{dK}$ and $\ell'=\Gamma_{dK'}$ be two monomially admissible
Lagrangian sections, whose slopes $\sigma$ and $\sigma'$ do not differ by a
multiple of $2\pi$ and such that $K'-K$ is convex, and let
$\mathfrak{t}_{x}$ be the Lagrangian torus with local system associated
to the point $x\in (\K^*)^n$ as above. For $p\in P(\sigma'-\sigma)\cap (2\pi
\Z)^n$, the Floer product of the generators $\vartheta_p\in
HF^0(\ell',\ell)$ and $\varepsilon_x \in HF^0(\ell, \mathfrak{t}_x)$ is
given by
\begin{equation}
\varepsilon_x\cdot \vartheta_p = x^{\bar{p}}\, \varepsilon'_x,
\end{equation}
where $\bar{p}=p/2\pi\in \Z^n$, $x^{\bar{p}}=\prod x_i^{\bar{p}_i}\in \K^*$, and $\varepsilon'_x$ is the generator of $HF^0(\ell',\mathfrak{t}_x)$
rescaled in the same manner as $\varepsilon_x$.
\end{proposition}

\proof The argument is similar to the proof of Proposition \ref{prop:mon_adm_hf_product}.
We lift $\ell$ and $\ell'$ to $T^*\R^n$ by considering the graphs $\tilde\ell$ and
$\tilde\ell'$ of $d(K+\langle p,\cdot\rangle)$ and $dK'$, which intersect
at a lift of $\vartheta_p$, and lift $\mathfrak{t}_x$ to the cotangent fiber
at $\xi=-\frac{1}{2\pi}\mathrm{val}(x)$. Any holomorphic disc contributing to the
Floer product of $\vartheta_p$ and $\varepsilon_x$ lifts to $T^*\R^n$, and
its symplectic area can be calculated by integrating $d(K'-K-\langle
p,\cdot\rangle)$ from $\xi_p$ to $\xi$, where $\xi_p$ is the
critical point of $K'-K-\langle p,\cdot\rangle$, which gives $$(K'(\xi)-K(\xi)-\langle p,\xi\rangle)-
(K'(\xi_p)-K(\xi_p)-\langle p,\xi_p\rangle).$$ The contribution to the Floer
product also involves a holonomy factor, given by the ratio between
the parallel transport of $\varepsilon_x$ along
$\mathfrak{t}_x$ from $(\xi,dK(\xi)+p)$ to $(\xi,dK'(\xi))$ and $\varepsilon'_x$.
Given the above choices of normalizations of the generators $\vartheta_p$,
$\varepsilon_x$, and $\varepsilon'_x$, we find that the contribution of each holomorphic
disc to the coefficient of $\varepsilon'_x$ in the product of
$\varepsilon_x$ and $\vartheta_p$ is, up to sign,
$t^{-\langle p,\xi\rangle}$ times the holonomy of $\mathfrak{t}_x$ along
a closed loop whose lift to the universal cover runs from
$(\xi,dK(\xi)+p)$ to $(\xi,dK(\xi))$. This loop represents the homotopy class
$-\bar{p}\in \Z^n\simeq \pi_1(T^n)$; hence, the holonomy can be expressed
as $y^{-\bar{p}}$, and one ends up with
$$t^{-\langle p,\xi\rangle}\,y^{-\bar{p}}=x^{\bar{p}}.$$
It only remains to show that the signed count of holomorphic discs
contributing to the Floer product of $\varepsilon_x$ and $\vartheta_p$ is
$+1$. Since this count is invariant under deformations, it does not depend
on the value of $\xi$ (the position of the cotangent fiber), and it suffices
to determine it for a particular
value of $\xi$. We take $\xi=\xi_p$, when all three
intersection points coincide and the only contribution is from the constant
map, which is regular and contributes $+1$.
\endproof


\subsection{Floer products on $HF^*(L_0(t'),L_0(t))$} \label{ss:CFL0product}

We now return to our main topic, namely the
calculation of the Floer cohomology $HF^*(L_0(t'),L_0(t))$ for
\hbox{$t'>t$} and its product operations. As seen in Example
\ref{ex:slopes_L0}, the slopes of the monomially
admissible Lagrangian sections $\ell_0(t'),\ell_0(t)\subset W^{-1}(-1)$ and
$\ell_-(t'),\ell_+(t)\subset W^{-1}(c_{t',t})$ (for $t'-t>t_0$) are given by
\eqref{eq:slope_l0}--\eqref{eq:slope_l+}.

\begin{definition}
For $\tau>0$, we define
\begin{equation}\label{eq:slopes_diff}
\sigma_0(\tau)=(\tau\,d(\v))_{\v\in \mathcal{V}}
\quad \text{and}\quad \sigma_1(\tau)=(\tau\,d(\v)-2\pi v^0)_{\v\in \mathcal{V}}
\end{equation}
and denote by $P_0(\tau),P_1(\tau)$ the corresponding polytopes defined by
\eqref{eq:mon_adm_polytope}.
\end{definition}

Since $H$ is convex by Proposition \ref{prop:wrapping_ham},
the results of Section \ref{ss:fiberwise-floer} apply to the pair
$(\ell_0(t'),\ell_0(t))$ whenever $t'-t>0$.  However, because the clockwise monodromy
of $W:Y\to\C$ does not act by a convex Hamiltonian, there
is no similar guarantee for the pair $(\ell_-(t'),\ell_+(t))$; nonetheless,
$\sigma_1(\tau)$ is the slope of a convex Hamiltonian for
$\tau=t'-t$ sufficiently large (larger than some constant $t_1\geq t_0$),
so Propositions \ref{prop:mon_adm_hf}--\ref{prop:mon_adm_hf_eval} apply
to the Floer cohomology $HF^*(\ell_-(t'),\ell_+(t))$ whenever $t'-t>t_1$.

\begin{proposition}\label{prop:CFL0explicit}
For $\tau=t'-t\in (0,t_0)\cap U$, the Floer complex $CF^*(L_0(t'),L_0(t))$ is concentrated
in degree zero, the Floer differential vanishes, and 
\begin{equation}\label{eq:CFL0_explicit_simple}
HF^0(L_0(t'),L_0(t))\simeq HF^0(\ell_0(t'),\ell_0(t))\simeq \bigoplus_{p\in
P_0(t'-t)\cap (2\pi \Z)^n} \K\cdot \vartheta_{p}^{t'\to t},
\end{equation}
where the generators $\vartheta_{p}^{t'\to t}$ correspond to the intersections of
$\ell_0(t')$ and $\ell_0(t)$ inside $W^{-1}(-1)$, rescaled by action
as explained in Section \ref{ss:fiberwise-floer}.

For $\tau=t'-t\in (t_1,\infty)\cap U$, the Floer cohomology
$HF^*(L_0(t'),L_0(t))$ is isomorphic to the cohomology of the complex
\begin{multline}\label{eq:CFL0_explicit}
\Bigl\{
HF^0(\ell_-(t'),\ell_+(t))\stackrel{s}{\longrightarrow}
HF^0(\ell_0(t'),\ell_0(t))\Bigr\}\simeq\\
\Biggl\{
\bigoplus_{p\in P_1(t'-t)\cap (2\pi \Z)^n} \!\!\K\cdot \zeta_p^{t'\to t}\,\,
\stackrel{s}{\longrightarrow}
\bigoplus_{p\in P_0(t'-t)\cap (2\pi \Z)^n} \!\!\K\cdot \vartheta_p^{t'\to t}
\Biggr\}
\end{multline}
where the generators $\zeta_p^{t'\to t}$ $($in degree $-1)$ and
$\vartheta_p^{t'\to t}$
(in degree zero) correspond to intersections of $\ell_-(t')$ and $\ell_+(t)$
inside $W^{-1}(c_{t',t})$ and to intersections of $\ell_0(t')$ and
$\ell_0(t)$ inside $W^{-1}(-1)$, rescaled by action within the fibers of \,$W$;
and $s=s^0_{\ell_0,t',t}$ is defined by a weighted count of $J$-holomorphic
sections of\, $W:Y\to\C$ over the bounded region of the complex plane
delimited by $\gamma_t$ and $\gamma_{t'}$.
\end{proposition}

\proof This follows immediately from Propositions \ref{prop:CF_is_cone} and
\ref{prop:mon_adm_hf}. \endproof

\begin{remark}
There are two ways to understand the complex \eqref{eq:CFL0_explicit}  and its relation to the Floer complex $CF^*(L_0(t'),L_0(t))$
for $t'-t>t_1$.  

(1)
Perturbing $L_0(t')$ or $L_0(t)$ by
an admissible Hamiltonian isotopy
(preserving the fibers of $W$, and preserving fiberwise monomial
admissibility) if necessary, we can assume that (suitably perturbed versions
of) the monomially admissible Lagrangian sections $\ell_-(t')$ and $\ell_+(t)$ 
differ by a convex Hamiltonian. After such a perturbation, both of the Floer complexes
$CF^*(\ell_-(t'),\ell_+(t))$ and $CF^*(\ell_0(t'),\ell_0(t))$ are 
concentrated in degree $0$ and their differentials vanish, so that
$CF^*(L_0(t'),L_0(t))$ is given by \eqref{eq:CFL0_explicit}.

(2) Alternatively, consider the filtration
$0\subset CF^*(\ell_0(t'),\ell_0(t))\subset
CF^*(L_0(t'),L_0(t))$, which is compatible with the Floer differential
and products, 
as any holomorphic disc contributes in a manner that
decreases the filtration index by 
its intersection number with the fibers of $W$ near the origin.%
\footnote{Reinterpreting Floer generators as Hamiltonian chords
on $L_0$, their filtration index is their intersection number 
with the preimage under $W$ of the real positive axis, making this an
instance of the filtration associated to a stop (and its removal) in partially
wrapped Floer theory \cite{Sylvan}.}
This filtration gives rise to a spectral
sequence computing $HF^*(L_0(t'),L_0(t))$, in which the second page 
(after taking the cohomology of the portion of the differential which
preserves the filtration index, i.e.\ the contributions of holomorphic discs contained
in the fibers of $W$)
is precisely \eqref{eq:CFL0_explicit}.
\end{remark}

\begin{definition}
We call the complex \eqref{eq:CFL0_explicit} $($or
\eqref{eq:CFL0_explicit_simple} for $t'-t\in (0,t_0))$ the {\em vertical Floer
complex} of $L_0(t')$ and $L_0(t)$, and denote it by
$CF^*_{vert}(L_0(t'),L_0(t))$.
\end{definition}

The vertical Floer complex carries a Floer product operation 
\begin{equation}\label{eq:HFvert_product}
CF_{vert}^*(L_0(t'),L_0(t))\otimes CF^*_{vert}(L_0(t''),L_0(t'))\to
CF^*_{vert}(L_0(t''),L_0(t))
\end{equation}
for $t''>t'>t$; this can
be understood either as the chain-level product $\mu^2$ after suitable
fiberwise perturbations, or as an induced product on the second page of
the spectral sequence computing the Floer cohomology (using the fact that
the product operation is compatible with the filtration). It follows from
the algebraic properties of the Floer product that this operation is
associative and satisfies the Leibniz rule with respect to the
section-counting differential $s$.

\begin{proposition}\label{prop:CFL0_product} 
Assume $t''>t'>t$, and $t'-t,t''-t,t''-t'\in U$, and label the generators
as in Proposition \ref{prop:CFL0explicit}. Then
the Floer product \eqref{eq:HFvert_product} is given by:
\begin{itemize}
\item for $p\in P_0(t'-t)\cap (2\pi\Z)^n$ and $p'\in P_0(t''-t')\cap
(2\pi\Z)^n$, $$\vartheta_p^{t'\to t}\cdot \vartheta_{p'}^{t''\to
t'}=\vartheta_{p+p'}^{t''\to t}\in HF^0(\ell_0(t''),\ell_0(t));$$
\item when $t'-t>t_1$, for $p\in P_1(t'-t)\cap (2\pi\Z)^n$ and $p'\in
P_0(t''-t')\cap(2\pi\Z)^n$,
$$\zeta_p^{t'\to t}\cdot \vartheta_{p'}^{t''\to t'}=C_{t''\to t',t}\,\zeta_{p+p'}^{t''\to t}\in HF^0(\ell_-(t''),\ell_+(t)),$$
where $C_{t''\to t',t}$ is a nonzero constant (independent of $p$ and $p'$);
\smallskip
\item when $t''-t'>t_1$, for $p\in P_0(t'-t)\cap (2\pi\Z)^n$ and $p'\in
P_1(t''-t')\cap(2\pi\Z)^n$,
$$\vartheta_p^{t'\to t}\cdot \zeta_{p'}^{t''\to t'}=C_{t'',t'\to t}\,\zeta_{p+p'}^{t''\to t}\in
HF^0(\ell_-(t''),\ell_+(t)),$$
where $C_{t'',t'\to t}$ is a nonzero constant (independent of $p$ and $p'$);
\smallskip
\item when $t'-t>t_1$ and $t''-t'>t_1$, for all $p$ and $p'$, $\zeta_p^{t'\to
t}\cdot \zeta_{p'}^{t''\to t'}=0$.
\end{itemize}
\end{proposition}

\proof
Since the projection $W:Y\to\C$ is holomorphic away from a
neighborhood of the zero fiber, it follows from the open mapping principle
and from degree constraints
that all the holomorphic discs contributing to the Floer product are either
contained in the fiber $W^{-1}(-1)$ or sections over a triangular region
of the complex plane delimited by the arcs $\gamma_{t''}$, $\gamma_{t'}$ and
$\gamma_t$ (see Figure \ref{fig:L0}). 

When both inputs lie in $W^{-1}(-1)$, the output must also lie in
$W^{-1}(-1)$ for degree reasons, and the only contributions come from
discs contained inside $W^{-1}(-1)$. Given the relative positions of
the tangent lines to $\gamma_{t''}$, $\gamma_{t'}$ and $\gamma_t$ at $-1$,
the base of the fibration $W:Y\to\C$ doesn't contribute anything to the index of
the Cauchy-Riemann operator, so the product operation agrees with the
product on the Floer complexes of the monomially admissible sections
$\ell_0(t'')$, $\ell_0(t')$ and $\ell_0(t)$ within $W^{-1}(-1)\simeq (\C^*)^n$. Hence,
using the same normalization of the generators as in Section
\ref{ss:fiberwise-floer}, it follows from Proposition \ref{prop:mon_adm_hf_product} 
that $\vartheta_p^{t'\to t}\cdot \vartheta_{p'}^{t''\to t'}=\vartheta_{p+p'}^{t''\to t}$.

Next we consider the case where one input lies in $W^{-1}(c_{t',t})$ (with
$t'-t>t_1$) and the other one is in $W^{-1}(-1)$. The output then
necessarily lies in $W^{-1}(c_{t'',t})$ for degree reasons, and the
contributions to the Floer product come from holomorphic sections over the
triangle $\mathcal{T}_{t''\to t',t}$ delimited by $\gamma_{t''}$, $\gamma_{t'}$ and $\gamma_t$ with
vertices at $-1$, $c_{t',t}$, and $c_{t'',t}$. Since we are considering
cohomology-level operations on the fiberwise Floer complexes, the count
we consider is homotopy invariant under deformations; it is in fact one
of the operations of the cohomology-level ``Seidel TQFT'' \cite{SeBook} associated to the
fibration $W:Y\to\C$ (in a fairly simple case,
since the region over which we count sections
does not contain the critical value $0$). Thus, we can simplify the
counting problem either by trivializing the fibration and deforming the symplectic and complex structures
to product ones over $\mathcal{T}_{t''\to t',t}$, or more simply, by
deforming the arc $\gamma_t$ (without crossing the origin) by a
compactly-supported isotopy in order to bring the intersection points
$c_{t',t}$ and $c_{t'',t}$ to $-1$ and shrink the triangular region
$\mathcal{T}_{t''\to t',t}$ to a single point. 
After this deformation, we are once again reduced to a calculation of the
Floer product for the admissible Lagrangian sections within a fiber of
$W$, as the horizontal direction does not contribute to the index of the
Cauchy-Riemann operator. Since the slopes of the relevant admissible 
Lagrangian sections differ by $\sigma_1(t'-t)$ at one input and by
$\sigma_0(t''-t')$ at the other, it
follows again from Proposition \ref{prop:mon_adm_hf_product} that, for
all $p\in P_1(t'-t)\cap (2\pi \Z)^n$ and $p'\in P_0(t''-t')\cap (2\pi\Z)^n$, 
the product of $\zeta_p^{t'\to t}$ and $\vartheta_{p'}^{t''\to t'}$ is equal to 
$\zeta_{p+p'}^{t''\to t}$ up to a scaling factor (some power of the Novikov
parameter) coming from the amount of
symplectic area swept in the deformation to a single fiber.

Next we show that, when all the generators are normalized by action within the
fibers of $W$, the coefficient of $\zeta_{p+p'}^{t''\to t}$ in the product
of  $\zeta_p^{t'\to t}$ and $\vartheta_{p'}^{t''\to t'}$ depends only on $t'',t',t$ but not on $p$ and $p'$.
Let $K_c:\R^n\to\R$ (resp.\ $K'_c,K''_c$) 
be such that the intersection of $L_0(t)$ (resp.\ $L_0(t'),L_0(t'')$)
with $W^{-1}(c)$ is the graph of $dK_c$ (resp.\ $dK'_c,dK''_c$) for each $c\in \gamma_t$ (resp.\ $\gamma_{t'},\gamma_{t''}$). Normalizing
$K_c,K'_c,K''_c$ suitably, we can ensure that they vanish at $\xi=0$, and that 
a holomorphic section $u$ of $W\co Y\to\C$ over $\mathcal{T}_{t''\to t',t}$
which contributes to 
the product of $\zeta_p^{t'\to t}$ and $\vartheta_{p'}^{t''\to t'}$ 
lifts to the universal cover of $W^{-1}(\mathcal{T}_{t''\to t',t})$ as
a section with boundary values on the graphs of $dK_c+p$, $dK'_c$, and
$dK''_c-p'$ for each $c\in \partial \mathcal{T}_{t''\to t',t}$.
With this understood, the holomorphic section $u$ represents the same relative homology class as the chain obtained by adding together:
\begin{enumerate}
\item the ``zero section'' of $W$ over $\mathcal{T}_{t''\to t',t}$, 
consisting of the points with moment map coordinates $\xi=0$
and angular coordinates $\theta_i=\arg(z_i)=0$ in each fiber;\smallskip
\item over each edge of $\mathcal{T}_{t''\to t',t}$, a path in each fiber
$W^{-1}(c)$, $c\in \partial \mathcal{T}_{t''\to t',t}$, connecting the
zero section to the boundary value $u(c)$ of the holomorphic section $u$ 
by running first along $\xi=0$ from the origin to $dK_c(0)+p$, $dK'_c(0)$, or
$dK''_c(0)-p'$, and then
along the graph of $dK_c+p$, $dK'_c$, or $dK''_c-p'$ from $\xi=0$ in a 
straight line to the $\xi$-coordinate of $u(c)$;\smallskip
\item over each vertex of $\mathcal{T}_{t''\to t',t}$, a chain in
$W^{-1}(c)$ ($c\in \{-1,c_{t',t}$, $c_{t'',t}\}$) which lies over a
straight line path from $\xi=0$ to the $\xi$-coordinate of $u(c)$,
and for each $\xi$-value
runs in a straight line from $dK'_c(\xi)$ to $dK''_c(\xi)-p'$ (for $c=-1$), resp.\
$dK_c(\xi)+p$ to $dK'_c(\xi)$ ($c=c_{t',t}$), resp.\ $dK_c(\xi)+p$ to $dK''_c(\xi)-p'$
($c=c_{t'',t}$).
\end{enumerate}

Denote by $A_{t''\to t',t}$ the symplectic area of the first part of our chain 
(the ``zero section''), which manifestly does not depend on $p$ and $p'$.
The second portion of our chain (over the edges of $\mathcal{T}_{t''\to
t',t}$) runs partly along the Lagrangians obtained by
parallel transport of the torus $\{\xi=0\}$ over $\gamma_t,\gamma_{t'},\gamma_{t''}$, 
and partly along the Lagrangians $L_0(t),L_0(t'),L_0(t'')$, so its
symplectic area vanishes. Finally, the third piece (over the vertices)
contributes at each vertex an area equal to the fiberwise action of the 
corresponding Floer generator, given that we have normalized the Hamiltonians
$K_c,K'_c,K''_c$ so that they vanish at $\xi=0$.
For instance, the portion
which lies in $W^{-1}(c_{t',t})$, over the path from $0$ to $\xi=\xi_p$ and
between the graphs of $dK_c+p$ and $dK'_c$, has symplectic area given by
the integral of $dK'_c-dK_c-p$ from zero to $\xi_p$, i.e.\ 
$(K'_c(\xi_p)-K_c(\xi_p)-\langle p,\xi_p\rangle)-(K'_c(0)-K_c(0))$,
which coincides with the fiberwise action for the generator $\zeta_p^{t'\to
t}$ within $W^{-1}(c_{t',t})$ since the last term vanishes. 
Similarly at the two other vertices. Because a rescaling by action is built
into the definition of our Floer generators, this implies that the coefficient of
$\zeta^{t''\to t}_{p+p'}$ in the product $\zeta_p^{t'\to t}\cdot
\vartheta_{p'}^{t''\to t'}$ is $C_{t''\to t',t}=t^{A_{t''\to t',t}}$.

The case of the product $\vartheta_p^{t'\to t}\cdot \zeta_{p'}^{t''\to t'}$
is handled by exactly the same argument, deforming the problem from a count
of sections over a triangular region of the complex plane to a fiberwise
Floer product and appealing to Proposition \ref{prop:mon_adm_hf_product}.
Finally, the product of two degree $-1$ generators vanishes for degree reasons.
\endproof

For $x=(x_1,\dots,x_n)\in (\K^*)^n$ and $t\in \R$, we denote by $T_x(t)$
the admissible Lagrangian with local system obtained by parallel
transport over the arc $\gamma_t$ of the Lagrangian torus with local system 
$\mathfrak{t}_x$ introduced in Section \ref{ss:fiberwise-floer}.
To be more specific, we fix a $T^n$-equivariant structure on the local
system of $\mathfrak{t}_x$, i.e.\ a family of isomorphisms between the local
system and its pullbacks under rotations by elements of $T^n$. (This can be
done e.g.\ by thinking of the local system as a trivial complex
line bundle equipped with a translation-invariant connection.) 
With this understood, $\mathfrak{t}_x$ is invariant under both parallel
transport between the fibers of $W$ and the action of the
wrapping Hamiltonian, and the restriction of $T_x(t)$ to
the fiber of $W$ over any point of $\gamma_t$ can be identified
(as a Lagrangian submanifold with local system) with $\mathfrak{t}_x$.

For $t'-t>t_0$, $L_0(t')$ and
$T_x(t)$ intersect transversely once in $W^{-1}(-1)$ and once
in $W^{-1}(c_{t',t})$; we denote by $\varepsilon_x^{t'\to t}\in
HF^0(\ell_0(t'),\mathfrak{t}_x)$ and $\eta_x^{t'\to t}\in
HF^0(\ell_-(t'),\mathfrak{t}_x)$ the corresponding Floer generators,
rescaled by action as in Section \ref{ss:fiberwise-floer}. We now consider
the Floer product
\begin{equation}\label{eq:HFvert_eval}
CF_{vert}^*(L_0(t'),T_x(t))\otimes CF_{vert}^*(L_0(t''),L_0(t')) \to
CF_{vert}^*(L_0(t''),T_x(t)).
\end{equation}

\begin{proposition}\label{prop:CFL0_eval}
For $t'-t>t_0$, $CF^*_{vert}(L_0(t'),T_x(t))=CF^*(L_0(t'),T_x(t))$ is given by
\begin{equation}\label{eq:CFL0_eval}
\Bigl\{\K\cdot \eta_x^{t'\to t} \stackrel{s_x}{\longrightarrow}
\K\cdot \varepsilon_x^{t'\to t}\Bigr\},
\end{equation}
where the generators $\eta_x^{t'\to t}$ $($in degree $-1)$ and 
$\varepsilon_x^{t'\to t}$ (in degree zero) correspond to intersections of
$\ell_-(t')$ and $\ell_0(t')$ with $\mathfrak{t}_x$ inside $W^{-1}(c_{t',t})$
and $W^{-1}(-1)$ respectively, rescaled by action, and $s_x$ is defined by
a weighted count of $J$-holomorphic sections of $W:Y\to\C$ over the bounded
region of the complex plane delimited by $\gamma_t$ and $\gamma_{t'}$.

Moreover, given $t''>t'>t$ with $t'-t>t_0$, the Floer product
\eqref{eq:HFvert_eval} is given by:

\begin{itemize}
\item for $p=2\pi \bar{p}\in P_0(t''-t')\cap (2\pi \Z)^n$,
\begin{eqnarray*}
\varepsilon_x^{t'\to t}\cdot \vartheta_p^{t''\to t'}
&=& x^{\bar{p}}\,\varepsilon_x^{t''\to t}\in HF^0(\ell_0(t''),\mathfrak{t}_x)
\quad \text{and}\\
\eta_x^{t'\to t}\cdot \vartheta_p^{t''\to t'}
&=& C_{\xi;t''\to t',t}\,x^{\bar{p}}\,\eta_x^{t''\to t}\in HF^0(\ell_-(t''),\mathfrak{t}_x);
\end{eqnarray*}
\item if moreover $t''-t'>t_1$, then for $p=2\pi \bar{p}\in P_1(t''-t')\cap
(2\pi \Z)^n$,
$$\varepsilon_x^{t'\to t}\cdot \zeta_p^{t''\to t'}\ =\ 
C_{\xi;t'',t'\to t}\,x^{\bar{p}}\,\eta_x^{t''\to t}\in HF^0(\ell_-(t''),\mathfrak{t}_x)$$
and $\eta_x^{t'\to t}\cdot \zeta_p^{t''\to t'}=0$.
\end{itemize}
Here $C_{\xi;t''\to t',t}$ and $C_{\xi;t'',t'\to t}$ are non-zero constants
which depend on $t'',t',t$ and possibly on $\xi=-\frac{1}{2\pi}\mathrm{val}(x)$ but not 
on $p$.
\end{proposition}

\proof
The proof is identical to that of Proposition \ref{prop:CFL0_product}, 
except after reduction to a Floer product within the fiber of $W$ we now
appeal to Proposition \ref{prop:mon_adm_hf_eval}. The other difference with
our previous argument is that the scaling constant
$C_{\xi;t''\to t',t}$ is now determined by the symplectic area of
a reference section of $W$ over $\mathcal{T}_{t''\to t',t}$ whose edge
along $\gamma_t$ lies
at the $\xi$-value of $\mathfrak{t}_x$, i.e.\ 
$\xi=-\frac{1}{2\pi}\mathrm{val}(x)$, rather than at $\xi=0$, hence it
generally depends on $\xi$; similarly for $C_{\xi;t'',t'\to t}$.
\endproof

Our next result concerns the quasi-units induced by continuation:

\begin{proposition}\label{prop:CFL0_quasiunit}
For $t'>t$, the quasi-unit $e^{t'\to t}\in HF^0(L_0(t'),L_0(t))$ is
given by $e^{t'\to t}=\vartheta_0^{t'\to t}$.
\end{proposition}

\proof
It suffices to prove the result for
$t'-t\in (0,t_0)$, as the general case follows using the multiplicative
property of quasi-units ($e^{t''\to t}=e^{t'\to t}\cdot e^{t''\to t'}$ 
for $t''>t'>t$) and Proposition \ref{prop:CFL0_product}.

Recall that the quasi-unit is defined by counting solutions to a
Cauchy-Riemann equation whose domain $\Sigma$ is a disc with a single output
boundary puncture, with moving boundary condition given by the Lagrangians
$L_0(\tau)$ for $\tau$ varying between $t$ and $t'$. 
Along $\partial\Sigma$, the boundary condition is obtained from
the flow of $X_K\otimes \eta$ for
some 1-form $\eta$ on $\partial\Sigma$ and some Hamiltonian $K$, namely
the sum of a Hamiltonian generating the admissible lifted isotopy
$\rho^\tau$, cf.\ Lemma \ref{l:ham_lifted_isotopy}, which we assume to be
supported over a neighborhood $V$ of $\bigcup_{\tau \in [t,t']} \gamma_\tau$, 
and the wrapping Hamiltonian $H$. The restriction of $K$ to $L_0(\tau)$
is proper and achieves its minimum at the point of $W^{-1}(-1)$ where $H$ has its minimum;
we normalize $K$ so that this minimum value is zero.

As in the proof of
Proposition \ref{prop:mon_adm_hf_quasiunit}, we consider solutions to the
perturbed Cauchy-Riemann equation
$(du-X_K\otimes \alpha)^{0,1}=0$,  and $\alpha$ is a sub-closed 1-form on $\Sigma$ whose
restriction to $\partial\Sigma$ agrees with $\eta$.
Solutions to this equation satisfy the
open mapping principle with respect to the projection $W:Y\to\C$ everywhere
outside of $V$ (where $X_K$ is
not purely vertical) and a neighborhood of the origin (where $W$ isn't
necessarily $J$-holomorphic); this implies that solutions remain within
$W^{-1}(V)$, where the K\"ahler form is exact and the same energy argument
as in the proof of Proposition \ref{prop:mon_adm_hf_quasiunit} shows
that the only solution is the constant map at the point of $W^{-1}(-1)$
where $H$ reaches its minimum.  It follows that $e^{t'\to
t}=\vartheta_0^{t'\to t}$.
\endproof

\subsection{The Floer differential}

Propositions \ref{prop:CFL0explicit}--\ref{prop:CFL0_quasiunit} give all
the information needed to determine the fiberwise wrapped Floer cohomology
$H\Wrap^*(L_0,L_0)$ and its ring structure, except for one key piece of data: the
differential of the complex \eqref{eq:CFL0_explicit}, i.e.\ the
section-counting map $s=s^0_{\ell_0,t',t}:HF^0(\ell_-(t'),\ell_+(t))\to
HF^0(\ell_0(t'),\ell_0(t))$. We will first show that this map is given by multiplication with a Laurent polynomial, then show that this polynomial also controls the section-counting map for the parallel transport of the tori $\mathfrak{t}_x$.

Fix $t_+>t_-$ with $t_+-t_->t_1$, and for  $p=2\pi\bar{p}\in P_0(t_+-t_-)\cap (2\pi\Z)^n$, denote by
$c_{\bar{p}}\in \K$ the coefficients such that 
\begin{equation}\label{eq:s0zeta0ref}
s^0_{\ell_0,t_+,t_-}(\zeta_0^{t_+\to t_-})=
\sum_p c_{\bar{p}}\, \vartheta_p^{t_+\to t_-}.
\end{equation}

\begin{lemma}\label{l:s0zeta}
For all $t'>t$ such that $t'-t>t_1$ and all $p'\in P_1(t'-t)\cap (2\pi\Z)^n$, 
\begin{equation}\label{eq:s0zeta}
s^0_{\ell_0,t',t}(\zeta_{p'}^{t'\to t})=C(t',t)\,\sum_p c_{\bar{p}}\,\vartheta_{p+p'}^{t'\to t}
\end{equation}
where $C(t',t)$ is a nonzero constant depending only on $t$ and $t'$.
Moreover, if $c_{\bar{p}}\neq 0$ then $\bar{p}\in P_\Z$.
\end{lemma}

\proof
The compatibility of the Floer product with the differential (i.e.\ the Leibniz
rule), together with the product formulas of Proposition
\ref{prop:CFL0_product}, implies that
\begin{equation}\label{eq:s0leibniz1}
s^0_{\ell_0,t',t}(\zeta_{p_1}^{t'\to t})\cdot \vartheta_{p_2}^{t''\to t'}=
C_{t''\to t',t}\,s^0_{\ell_0,t'',t}(\zeta_{p_1+p_2}^{t''\to t})
\end{equation}
for all $(t''>t'>t)$ with $t'-t>t_1$, $p_1\in P_1(t'-t)\cap (2\pi \Z)^n$, $p_2\in P_0(t''-t')\cap (2\pi\Z)^n$; and
\begin{equation}\label{eq:s0leibniz2}
\vartheta_{p_1}^{t'\to t}\cdot s^0_{\ell_0,t'',t'}(\zeta_{p_2}^{t''\to t'})=
C_{t'', t'\to t}\,s^0_{\ell_0,t'',t}(\zeta_{p_1+p_2}^{t''\to t})
\end{equation}
for all $(t''>t'>t)$ with $t''-t'>t_1$, $p_1\in P_0(t'-t)\cap (2\pi\Z)^n$, $p_2\in P_1(t''-t')\cap (2\pi\Z)^n$.

We now deduce the lemma from these two identities. First, choose $t''>\max(t',t_+)$ 
such that $P_1(t'-t)\subset P_0(t''-t_+)$.
It follows from \eqref{eq:s0leibniz1} for $(t''>t_+>t_-)$, $p_1=0$ and $p_2=p'$
that, for all $p'\in
P_1(t'-t)\cap (2\pi\Z)^n\subset P_0(t''-t_+)\cap (2\pi\Z)^n$,
$$
s^0_{\ell_0,t'',t_-}(\zeta^{t''\to t_-}_{p'})=C_{t''\to t_+,t_-}^{-1}\,
s^0_{\ell_0,t_+,t_-}(\zeta_0^{t_+\to t_-})\cdot \vartheta_{p'}^{t''\to t_+}
=C_{t''\to t_+,t_-}^{-1}\,\sum_p c_{\bar{p}}\,\vartheta_{p+p'}^{t''\to t_-}.
$$
Next, considering \eqref{eq:s0leibniz2} for either $(t''>t_->t)$ or
$(t''>t>t_-)$, with $p_1=0$ and $p_2=p'$ again, yields
$$s^0_{\ell_0,t'',t}(\zeta_{p'}^{t''\to t})=C(t'',t)\sum_p c_{\bar{p}}\,
\vartheta^{t''\to t}_{p+p'}$$
for all $p'\in P_1(t'-t)\cap (2\pi\Z)^n$, where
$C(t'',t)$ equals $C_{t''\to t_+,t_-}^{-1} C_{t'',t_-\to t}^{-1}$ if $t<t_-$, or $C_{t''\to t_+,t_-}^{-1} C_{t'',t\to t_-}$
if $t>t_-$.
This is precisely \eqref{eq:s0zeta}, except with $t''$ everywhere instead of
$t'$. Finally, we use \eqref{eq:s0leibniz1}, now for $(t''>t'>t)$, $p_1=p'$,
and $p_2=0$, to conclude that $$s^0_{\ell_0,t',t}(\zeta^{t'\to t}_{p'})=
C_{t''\to t',t} C(t'',t)\,\sum_p c_{\bar{p}}\,\vartheta_{p+p'}^{t'\to t},$$
which is the desired result. 

Moreover, the final step of the calculation
implies that $p+p'\in P_0(t'-t)\cap (2\pi \Z)^n$ for all $p=2\pi\bar{p}$ such that
$c_{\bar{p}}\neq 0$ and for all $p'\in P_1(t'-t)\cap (2\pi \Z)^n$.
Recall that $P_0(t'-t)$ is defined by the inequalities 
\begin{equation}\label{eq:P0ineq}
\langle \vec{v},\cdot\rangle \leq (t'-t)\,d(\v)\end{equation} 
for all $\v=(\vec{v},v^0)\in\mathcal{V}$, while $P_1(t'-t)$ is defined
by \begin{equation}\label{eq:P1ineq}
\langle \vec{v},\cdot\rangle \leq (t'-t)\,d(\v)-2\pi v^0\end{equation}
for all $\v\in
\mathcal{V}$, and $P$ is defined by
the inequalities $\langle \vec{v},\cdot \rangle \leq v^0$ for all
$\v\in\mathcal{V}$ (cf.\ Definition \ref{def:extremalv}). For every $\v\in\mathcal{V}$, we can choose 
$t$ and $t'$ such that $P_1(t'-t)\cap (2\pi\Z)^n$ contains some $p'$
which realizes the equality in \eqref{eq:P1ineq}. Thus, since $p+p'$
satisfies \eqref{eq:P0ineq} whenever $c_{\bar{p}}\neq 0$, it follows
that $\langle p,\vec{v}\rangle\leq 2\pi v^0$, i.e.\ $\langle
\bar{p},\vec{v}\rangle \leq v^0$, whenever $c_{\bar{p}}\neq 0$.
Since this holds for all $\v\in\mathcal{V}$, it follows that $\bar{p}\in
P\cap \Z^n=P_\Z$.
\endproof

Lemma \ref{l:s0zeta} implies that the coefficients $c_{\bar{p}}\in \K$ 
($\bar{p}\in P_\Z$) suffice to determine the fiberwise wrapped Floer cohomology of
$L_0$. More explicitly:

\begin{proposition}\label{prop:HWL0}
Let $g(x)=\sum_p c_{\bar{p}} x^{\bar{p}}\in \K[x_1^{\pm 1},\dots,x_n^{\pm
1}]$, and assume that $g$ is not identically zero.
Then $H\Wrap^*(L_0,L_0)$ is isomorphic to the quotient $\K[x_1^{\pm
1},\dots,x_n^{\pm 1}]/(g)$ of the ring of Laurent polynomials by the ideal
generated by $g$.
\end{proposition}

\proof By Corollary \ref{cor:colimit}, we can calculate $H\Wrap^*(L_0,L_0)$ as
a colimit of Floer cohomology groups $HF^*(L_0(t'),L_0(t))$ for $t'-t\to
\infty$. For $t'-t>t_1$, we use Proposition \ref{prop:CFL0explicit} and
Lemma \ref{l:s0zeta} to identify $CF^*_{vert}(L_0(t'),L_0(t))$ 
with a subcomplex of the chain complex
\begin{equation}\label{eq:Laurent_g}\K[x_1^{\pm 1},\dots,x_n^{\pm 1}]\stackrel{g}{\longrightarrow}
\K[x_1^{\pm 1},\dots,x_n^{\pm 1}]\end{equation}
where in degree $0$ we identify $\vartheta_p^{t'\to t}$ with the
monomial $x^{\bar{p}}$ for all $p\in P_0(t'-t)\cap (2\pi\Z)^n$, and
in degree $-1$ we identify $\zeta_p^{t'\to t}$ with $C(t',t)\,x^{\bar{p}}$
for all $p\in P_1(t'-t)\cap (2\pi\Z)^n$, and the subcomplex corresponds to
those Laurent polynomials whose Newton polytopes are contained inside
$\frac{1}{2\pi}P_0(t'-t)$ resp.\ $\frac{1}{2\pi}P_1(t'-t)$.

It follows from Proposition \ref{prop:CFL0_product} that, with these
identifications, the product operations on these Floer complexes are given
by multiplication of Laurent polynomials; and 
Proposition \ref{prop:CFL0_quasiunit} implies that the
continuation maps as $t'-t$ increases to infinity are given by inclusion.
Thus, the naive limit of the complexes \eqref{eq:CFL0_explicit} as $t'-t\to
\infty$ is given by \eqref{eq:Laurent_g}.

Since by assumption $g$ is not zero, multiplication by $g$ is injective,
and the cohomology of \eqref{eq:CFL0_explicit} is 
concentrated in degree zero; specifically, $HF^0(L_0(t'),L_0(t))$ is
the quotient of the space of
Laurent polynomials whose Newton polytope is contained in $\frac1{2\pi}
P_0(t'-t)$ by the subspace of those which are $g$ times a Laurent polynomial
with Newton polytope contained in $\frac{1}{2\pi} P_1(t'-t)$. Taking
the colimit under inclusion maps as $t'-t\to \infty$, we conclude that
$H\Wrap^*(L_0,L_0)$ is also concentrated
in degree zero, and we have an isomorphism of $\K$-vector spaces
$$H\Wrap^0(L_0,L_0)\simeq \K[x_1^{\pm 1},\dots,x_n^{\pm 1}]/(g).$$ 
This isomorphism is compatible with the ring structure,
since by Proposition \ref{prop:CFL0_product} the Floer product operation
corresponds to multiplication of Laurent polynomials.
\endproof

Given Proposition \ref{prop:HWL0}, the proof of Theorem \ref{thm:main}
reduces to the determination of the Laurent polynomial $g$. More
precisely, we need to show that, after equipping $Y$ with a suitable bulk
deformation class, $g$ can be assumed to coincide with the Laurent
polynomial $f$ defining the hypersurface $H$ up to an overall scaling
factor. To this end, we first reinterpret
$g$ as a count of holomorphic sections with boundary on the objects
$T_x(t)$ obtained by parallel transport of product tori with rank one local
systems. Recalling the calculation of the vertical Floer complex
$CF^*_{vert}(L_0(t'),T_x(t))$ from Proposition \ref{prop:CFL0_eval}, we have:

\begin{proposition}\label{prop:s0zeta_eval}
For $t'-t>t_1$, and for $x\in (\K^*)^n$, the differential on the complex 
$CF^*_{vert}(L_0(t'),T_x(t))$ is given by $$s_x(\eta_x^{t'\to
t})=C_\xi(t',t)\,g(x)\,\varepsilon_x^{t'\to t},$$
where $C_\xi(t',t)$ is a nonzero constant depending only on $t$, $t'$, and
$\xi=-\frac{1}{2\pi}\mathrm{val}(x)$.
\end{proposition}

\proof For $t''>t'+t_1$, the compatibility of the Floer product
\eqref{eq:HFvert_eval} with the differentials on the vertical Floer
complexes implies that
$$s_x(\eta_x^{t'\to t})\cdot \zeta_{0}^{t''\to t'}-\eta_x^{t'\to t}\cdot
s^0_{\ell_0,t'',t}(\zeta_{0}^{t''\to t'})=0.$$
Using Lemma \ref{l:s0zeta} and Proposition \ref{prop:CFL0_eval}, this yields:
$$s_x(\eta_x^{t'\to t})\cdot \zeta_{0}^{t''\to t'}=C(t'',t')\sum_p
c_{\bar{p}}\,\eta_x^{t'\to t}\cdot \vartheta^{t''\to t'}_{p}=C(t'',t')\,
C_{\xi;t''\to t',t}\,g(x)\,\eta_x^{t''\to t}.$$
Since $s_x(\eta_x^{t'\to t})$ is a multiple of $\varepsilon_x^{t'\to t}$,
comparing with the formula for $\varepsilon_x^{t'\to t}\cdot \zeta_{0}^{t''\to
t}$ given by Proposition \ref{prop:CFL0_eval} we conclude that
$$s_x(\eta_x^{t'\to t})=C_{\xi;t'',t'\to t}^{-1}\,C(t'',t')\,C_{\xi;t''\to t',t}\,
g(x)\,\varepsilon_x^{t'\to t}.$$
The result follows, setting $C_{t'',t'\to t}^{-1}\,C(t'',t')\,C_{t''\to t',t}=C_\xi(t',t)$.
\endproof

\begin{remark} Another way to prove Proposition \ref{prop:s0zeta_eval}, still
using the Leibniz rule, Lemma \ref{l:s0zeta}, and
Proposition \ref{prop:CFL0_eval}, is to argue that, for $t''>t'>t$ with $t''-t'>t_1$, 
\begin{align*}
s_x(\eta_x^{t''\to t})&=C_{\xi;t'',t'\to t}^{-1}\,s_x(\varepsilon_x^{t'\to t}\cdot
\zeta_0^{t''\to t'})=C_{\xi;t'',t'\to t}^{-1}\,\varepsilon_x^{t'\to t}\cdot
s^0_{\ell_0,t'',t'}(\zeta_0^{t''\to t'})\\
&=C_{\xi;t'',t'\to t}^{-1}\,C(t'',t')\,\sum_p c_{\bar{p}}\,\varepsilon_x^{t'\to
t}\cdot \vartheta_p^{t''\to t'}=C_\xi(t'',t)\,g(x)\,\varepsilon_x^{t''\to t}.
\end{align*}
\end{remark}

Next we consider the Floer complex $CF_{vert}^*(T_x(t'),T_x(t))$ for $t'-t>t_0$. 
The Lagrangian submanifolds (with local systems) $T_x(t')$ and $T_x(t)$ 
obtained by parallel transport of $\mathfrak{t}_x$ over the arcs
$\gamma_{t'}$ and $\gamma_t$ intersect cleanly along tori within the fibers
$W^{-1}(-1)$ and $W^{-1}(c_{t',t})$, rather than
transversely, so the definition of their Floer complex requires a bit of
care. One approach is to use a small Hamiltonian perturbation to achieve
transversality within the fibers of $W$; another approach that is better
suited to computations is to use a ``Morse-Bott'' model. Namely, we
choose a Morse function on the 
$n$-torus, and consider holomorphic discs with boundary in
$T_x(t')\cup T_x(t)$ together with Morse
flow lines (within a component of $T_x(t')\cap T_x(t)$) from the boundary marked points of the disc to
critical points of the Morse function; see for example Section 4 of \cite{Sheridan}
(with the difference that we only use Morse theory within the fibers of $W$,
while in the base direction we have usual strip-like ends). Equivalently,
instead of involving Morse flow lines, one could simply
require the boundary marked points of the holomorphic discs to lie on
the stable or unstable manifolds of the Morse critical points.

Regardless of the chosen approach, the Floer complex is built from two
copies of the fiberwise Floer complex
$CF^*(\mathfrak{t}_x,\mathfrak{t}_x)$, corresponding to generators and
Floer trajectories which lie entirely within each of the two fibers
$W^{-1}(-1)$ and $W^{-1}(c_{t',t})$, together with a connecting differential
which counts $J$-holomorphic sections of $W:Y\to \C$ over the region delimited
by $\gamma_t$ and $\gamma_{t'}$ (with the usual caveat regarding our use of
the word ``section'' since $J$ differs from the standard complex structure
near $W^{-1}(0)$), with boundary on $\mathfrak{t}_x$, and satisfying incidence
conditions at $-1$ and at $c_{t',t}$.

As before, we denote by $CF^*_{vert}(T_x(t'),T_x(t))$ the ``vertical Floer
complex'' obtained by taking the cohomology with respect to the
contributions to the Floer differential which lie entirely within a fiber of $W$.
Since $\mathfrak{t}_x\subset (\C^*)^n$ does not bound any holomorphic discs,
the Floer differential on $CF^*(\mathfrak{t}_x,\mathfrak{t}_x)$ only involves a classical part, and reduces to the
usual cohomology of $T^n$ (with coefficients in endomorphisms of the local
system, which are canonically isomorphic to the ground field $\K$).
We claim:

\begin{proposition}\label{prop:s0torus}
For $t'-t>t_0$, and for $x\in (\K^*)^n$, 
the vertical Floer complex $CF^*_{vert}(T_x(t'),T_x(t))$ is given by
\begin{equation}\label{eq:s0torus}
\Bigl\{ H^*(T^n,\K) \stackrel{s_x}{\longrightarrow}
H^*(T^n,\K)\Bigr\},
\end{equation} where the connecting differential $s_x$, defined by a
weighted count of $J$-holomorphic sections of $W:Y\to\C$ over the
region delimited by $\gamma_t$ and $\gamma_{t'}$, with 
incidence conditions on cycles in $\mathfrak{t}_x$ at $-1$ and $c_{t',t}$,
is given by multiplication by $C'_\xi(t',t)\,g(x)\in\K$ for some non-zero
constant $C'_\xi(t',t)$ depending only on $t,t'$ and
$\xi=-\frac{1}{2\pi}\mathrm{val}(x)$.
\end{proposition}

The first part of the statement is clear from the above description of the
Floer complex $CF^*(T_x(t'),T_x(t))$; the remaining part, namely showing
that the differential $s_x$ is given by multiplication by $g(x)$, relies
on an algebraic argument similar to the proof of Proposition
\ref{prop:s0zeta_eval} using the Leibniz rule. Thus, we first need to 
establish a couple of lemmas (analogous to Propositions
\ref{prop:CFL0_product} and \ref{prop:CFL0_eval}), before providing the proof.

We denote respectively by $\delta_x^{t'\to t}$ and $1_x^{t'\to t}$ the
elements which correspond to $1\in H^0(T^n,\K)$ in
the left and right summands of \eqref{eq:s0torus};
given $\alpha\in H^*(T^n,\K)$, the corresponding elements of the
left and right summands of
\eqref{eq:s0torus} are denoted by $\alpha\,\delta_x^{t'\to t}$ 
and $\alpha\,1_x^{t'\to t}$. With this notation, we have:

\begin{lemma}\label{l:s0torusproduct}
Assuming $t'-t>t_0$ and $t''-t'>t_0$, the Floer product
$$CF^*_{vert}(T_x(t'),T_x(t))\otimes CF^*_{vert}(T_x(t''),T_x(t'))\to
CF^*_{vert}(T_x(t''),T_x(t))$$ is as follows: for
all $\alpha,\alpha'\in H^*(T^n,\K)$,
\begin{eqnarray*}
(\alpha\,1_x^{t'\to t})\cdot (\alpha'\,1_x^{t''\to t'})&=&(\alpha\cupprod
\alpha')\,1_x^{t''\to t},\\
(\alpha\,\delta_x^{t'\to t})\cdot (\alpha'\,1_x^{t''\to t'})&=&
C'_{\xi;t''\to t',t}\,(\alpha\cupprod \alpha')\,\delta_x^{t''\to t},\\
(\alpha\,1_x^{t'\to t})\cdot (\alpha'\,\delta_x^{t''\to t'})&=&
C'_{\xi;t'',t'\to t}\,(\alpha\cupprod \alpha')\,\delta_x^{t''\to t},\\
(\alpha\,\delta_x^{t'\to t})\cdot (\alpha'\,\delta_x^{t''\to t'})&=&0
\end{eqnarray*}
where $C'_{\xi;t''\to t',t},C'_{\xi;t'',t'\to t}\in \K^*$
depend only on $t,t',t''$ and $\xi=-\frac{1}{2\pi}\mathrm{val}(x)$.
\end{lemma}

\proof The proof is essentially the same as for Proposition
\ref{prop:CFL0_product}: by considering the projection under $W:Y\to \C$,
we find that the only holomorphic discs contributing to the Floer product
are either contained in $W^{-1}(-1)$, or sections over one of the two 
triangular regions delimited by $\gamma_{t''}$, $\gamma_{t'}$ and
$\gamma_t$; in the latter case, we use a deformation argument to shrink
the triangular region to a single point and reduce to a count within the
fiber of $W$. Either way, things reduce to the Floer product on
$HF^*(\mathfrak{t}_x,\mathfrak{t}_x)\simeq H^*(T^n,\K)$, which coincides
with the ordinary cup product since there are no non-constant holomorphic
discs in $(\C^*)^n$ with boundary on $\mathfrak{t}_x$.
As in the proof of Proposition \ref{prop:CFL0_product}, the constant factors
$C'_{\xi;t''\to t',t}$ and $C'_{\xi;t'',t'\to t}$ account for the symplectic area
of a reference section (now chosen to lie at the same $\xi$-value as
$\mathfrak{t}_x$, i.e.\ $\xi=-\frac{1}{2\pi}\mathrm{val}(x)$) over the appropriate triangular region of the complex
plane, which turns out to coincide with the amount of area swept in the 
deformation used to reduce to a single fiber.
\endproof

\begin{lemma}\label{l:s0toruseval}
Assume $t'-t>t_0$ and $t''-t'>t_0$. The Floer product
$$CF^*_{vert}(T_x(t'),T_x(t))\otimes CF^*_{vert}(L_0(t''),T_x(t'))\to
CF^*_{vert}(L_0(t''),T_x(t))$$ vanishes identically on elements 
of the form $(\alpha\,1_{x}^{t'\to t})$ or $(\alpha\,\delta_x^{t'\to t})$
whenever $\alpha$ is a cohomology class of positive degree, whereas
\begin{eqnarray*}
1_x^{t'\to t}\cdot \eta_x^{t''\to t'}=C''_{\xi;t'',t'\to t}\,\eta_x^{t''\to t},
&\quad& 
1_x^{t'\to t}\cdot \varepsilon_x^{t''\to t'}=\varepsilon_x^{t''\to t},\\
\delta_x^{t'\to t}\cdot \varepsilon_x^{t''\to t'}=C''_{\xi;t''\to
t',t}\,\eta_x^{t''\to t},
&\quad& 
\delta_x^{t'\to t}\cdot \eta_x^{t''\to t'}=0,
\end{eqnarray*}
where $C''_{\xi;t''\to t',t},C''_{\xi;t'',t'\to t}\in \K^*$
depend only on $t,t',t''$ and $\xi=-\frac{1}{2\pi}\mathrm{val}(x)$.
\end{lemma}

\proof The argument is again similar, reducing to the calculation of
Floer products within the fiber $W^{-1}(-1)\simeq (\C^*)^n$, specifically
the product
$$HF^*(\mathfrak{t}_x,\mathfrak{t}_x)\otimes
HF^*(\ell_0(t''),\mathfrak{t}_x)\to HF^*(\ell_0(t''),\mathfrak{t}_x).$$
The vanishing for elements of $HF^*(\mathfrak{t}_x,\mathfrak{t}_x)\simeq H^*(T^n,\K)$ of positive degree then
follows from the fact that $HF^*(\ell_0(t''),\mathfrak{t}_x)$ has rank one
and is concentrated in a single degree; whereas $1\in H^0(T^n,\K)\simeq
HF^0(\mathfrak{t}_x,\mathfrak{t}_x)$ acts by identity by cohomological
unitality.
\endproof

\proof[Proof of Proposition \ref{prop:s0torus}]
Given $t,t'$ with $t'-t>t_0$, choose $t''$ so that $t''>t'+t_1$. 
The compatibility of Floer products and differentials on vertical Floer
complexes (the Leibniz rule) implies that
$$s_x(\delta_x^{t'\to t})\cdot \eta_x^{t''\to t'}-
\delta_x^{t'\to t}\cdot s_x(\eta_x^{t''\to t'})=0,$$
which using Proposition \ref{prop:s0zeta_eval} and Lemma \ref{l:s0toruseval} yields:
$$s_x(\delta_x^{t'\to t})\cdot \eta_x^{t''\to
t'}=C_\xi(t'',t')\,g(x)\,\delta_x^{t'\to t}\cdot \varepsilon_x^{t''\to t'}
=C''_{\xi;t''\to t',t}\,C_\xi(t'',t')\,g(x)\,\eta_x^{t''\to t}.$$
Using again Lemma \ref{l:s0toruseval} (and degree constraints), it follows
that $$s_x(\delta_x^{t'\to t})={C''}_{\!\!\!\!\xi;t'',t'\to t}^{-1}\,C''_{\xi;t''\to
t',t}\,C_\xi(t'',t')\,g(x)\,1_x^{t'\to t}.$$
Setting $C'_\xi(t',t)={C''}_{\!\!\!\!\xi;t'',t'\to t}^{-1}\,C''_{\xi;t''\to
t',t}\,C_\xi(t'',t')$, we rewrite this as
$$s_x(\delta_x^{t'\to t})=C'_\xi(t',t)\,g(x)\,1_x^{t'\to t}$$
whenever $t'-t>t_0$, which is the desired result for the generators of $H^0(T^n,\K)$.

To extend the result to higher degree cohomology classes, we use the
product formulas of Lemma \ref{l:s0torusproduct}: given $t'>t+t_0$, and
choosing $t''>t'+t_0$, the Leibniz rule implies that
\begin{multline*}
s_x(\alpha\,\delta_x^{t'\to t})\cdot \delta_x^{t''\to t'}=
(\alpha\,\delta_x^{t'\to t})\cdot s_x(\delta_x^{t''\to t'})=
C'_\xi(t'',t')\,g(x)\,(\alpha\,\delta_x^{t'\to t})\cdot 1_x^{t''\to t'}\\ 
=C'_{\xi;t''\to t',t}\,C'_\xi(t'',t')\,g(x)\,(\alpha\,\delta_x^{t''\to t}),
\end{multline*}
and hence
$$s_x(\alpha\,\delta_x^{t'\to t})={C'}_{\!\!\xi;t'',t'\to t}^{-1}\,C'_{\xi;t''\to
t',t}\,C'_\xi(t'',t')\,g(x)\,(\alpha\,1_x^{t'\to
t})=C'_\xi(t',t)\,g(x)\,(\alpha\,1_x^{t'\to t}),$$
where the identity $C'_\xi(t',t)={C'}_{\!\!\xi;t'',t'\to t}^{-1}\,C'_{\xi;t''\to
t',t}\,C'_\xi(t'',t')$ follows from considering the special case $\alpha=1$.
\endproof

Given Propositions \ref{prop:HWL0} and \ref{prop:s0torus}, the remaining
step in the proof of Theorem \ref{thm:main} is a direct calculation of the
differential in \eqref{eq:s0torus}, with the aim of showing that
the Laurent polynomials $f$ and $g$ agree up to a constant scaling factor.

\subsection{Holomorphic sections of $W$ with boundary on product tori}

We now turn to the problem of explicitly determining the differential
on the complex \eqref{eq:s0torus}, i.e.\ counting $J$-holomorphic sections of
$W:Y\to\C$ over the region delimited by $\gamma_{t'}$ and $\gamma_{t''}$,
with boundary in the product torus $\mathfrak{t}_x$ in each fiber.
(In this section we use $t'$ and $t''$ instead of $t$ and $t'$ to avoid notation
conflicts with the Novikov parameter).

By Proposition \ref{prop:s0torus}, the
differential $s_x$ is given by multiplication by some element of $\K$; 
thus it is enough to determine the image of the generator
of $H^0(T^n,\K)$ (or equivalently, that of $H^n(T^n,\K)$); this amounts to counting $J$-holomorphic sections 
whose boundary passes through some prescribed input point in $W^{-1}(-1)$ (or output point in
$W^{-1}(c_{t'',t'})$ if we consider $H^n$ rather than
$H^0$; or in fact a point anywhere on the Lagrangian boundary condition, as the end
result does not depend on this choice).

While our definitions involve a perturbation of the standard complex 
structure $J_0$ near $W^{-1}(0)$ in order to achieve regularity of moduli
spaces, actually counting discs in practice requires one to consider the
limit as $J$ converges to the (non-regular) standard complex structure $J_0$. 
Under this limit, the $J$-holomorphic discs contributing to the differential
$s_x$ converge either to holomorphic discs (holomorphic sections of 
$W:Y\to\C$), or to stable configurations consisting of a holomorphic
disc (a section of $W$) together with one or more rational curves contained
inside the singular fiber $W^{-1}(0)$. (This is a standard instance of
Gromov compactness for a $C^\infty$-convergent sequence of almost-complex structures, cf.\ 
\cite[Theorem 5.3.1]{mcduff-salamon} for the closed case; as usual when
considering sections, it follows from positivity of intersection of the
non-vertical components with
the fibers of $W$ that any bubbles arising in the limit must be 
contained in a fiber of $W$\!, hence in $W^{-1}(0)$.)
Thus, the first step is to
understand moduli spaces of holomorphic sections of $W$ bounded by
$T_x(t'')\cup T_x(t')$.

\begin{proposition}\label{prop:torusdiscs}
For $t''-t'>t_0$ and $x\in (\K^*)^n$, the homotopy classes of holomorphic sections of $W:Y\to\C$ with boundary
on $T_x(t'')\cup T_x(t')$ are in one-to-one correspondence with the elements
of $P_\Z$. For each such class, the moduli space of sections
consists of a single orbit under the action of $T^{n}$, and the count of
sections through any given point of $\mathfrak{t}_x\subset W^{-1}(-1)$
is equal to one.
\end{proposition}

\proof Denote by $S$ the region of the complex plane delimited by $\gamma_{t''}$ and
$\gamma_{t'}$. Since $S$ contains the origin, a holomorphic section of $W:Y\to\C$ over
$S$ has intersection number one with $Z=W^{-1}(0)=\bigcup_\alpha Z_\alpha$, which is
the union of the irreducible toric divisors of $Y$. Hence it must intersect
exactly one of these, say $Z_\alpha$ for some $\alpha\in P_\Z$, and be
disjoint from $Z_{\alpha'}$ for all $\alpha'\neq \alpha$. 
For fixed $\alpha$, we are thus reduced to studying
holomorphic discs contained in $Y_\alpha=Y\setminus \bigcup_{\alpha'\neq \alpha}
Z_{\alpha'}$, the partial compactification of the open stratum of the toric
variety $Y$ obtained by adding the open stratum of $Z_\alpha$.

$Y_\alpha$ is biholomorphic to $\C\times (\C^*)^n$, and we choose such an
identification where the first coordinate is given by
$W=-z^{(0,\dots,0,1)}$, and the remaining coordinates $(z_1,\dots,z_n)\in(\C^*)^n$
are given by toric monomials, in such a way that product tori in the
fibers $W^{-1}(c)$, $c\in \partial S$ correspond to standard product tori in
$\{c\}\times (\C^*)^n$.

We parametrize holomorphic sections of $W_{|Y_\alpha}:Y_\alpha\to \C$ over $S$ by the first coordinate (i.e.,
$W$), so that the domain is $S$, and we are reduced to finding holomorphic
maps $S\to (\C^*)^n$, $w\mapsto (z_1(w),\dots,z_n(w))$ which satisfy the appropriate boundary conditions over
$\partial S$. Specifically, our boundary condition is given by product
tori in $(\C^*)^n$, i.e.\ the value of $|z_i|$ is prescribed at every point
of the boundary.  We claim that solutions, if they exist, are unique up to the action of 
$T^n$ on $(\C^*)^n$ by rotations. Indeed, if $z_i,\tilde{z}_i:S\to \C^*$
are both holomorphic and $|z_i(w)|=|\tilde{z}_i(w)|$ for all $w\in \partial S$,
then the ratio $\tilde{z}_i(w)/z_i(w)$ defines a holomorphic map from $S$ to $\C^*$,
taking values in the unit circle along $\partial S$; the open mapping principle thus
implies that it is constant, i.e.\ there exists $e^{i\theta}\in S^1$ such
that $\tilde{z}_i(w)=e^{i\theta}z_i(w)$ for all $w\in S$. Thus the moduli
space of sections in the given class consists of at most one $T^n$-orbit.

One approach to prove existence is to use complex analysis. For each $i\in
\{1,\dots,n\}$, the boundary condition prescribes the value of $\log|z_i|=\mathrm{Re}(\log z_i)$ 
at every point of $\partial S$. Using the Riemann mapping theorem to identify $S$ with the unit
disc, it is a classical result of Schwarz that, up to a pure imaginary
additive constant, there exists a unique analytic function $\log z_i:\mathrm{int}(S)\to\C$
(given by the Schwarz integral formula) whose real part has a continuous extension
and takes the prescribed values at the boundary of $S$ (see e.g.\
\cite[\S 4.6.3-4.6.4]{Ahlfors}).
Because the given real boundary condition along the unit circle is H\"older continuous
(even after pullback from $S$ to the disc, see e.g.\ \cite[Chapter 3]{Pommerenke}), the imaginary part $\mathrm{Im}(\log z_i)$
(the harmonic conjugate of $\mathrm{Re}(\log z_i)$) also has a (H\"older) continuous
extension to the boundary, given by the Hilbert transform of the real part
\cite[Theorem III.1.3]{Garnett}.
Exponentiating, we arrive at the desired mapping $z_i:S\to \C^*$, and
conclude that, up to the action of $T^n$ by rotation of the
coordinates of $(\C^*)^n$, there is a unique continuous map $w\mapsto
(z_1(w),\dots,z_n(w))$ from $S$ to $(\C^*)^n$ which is holomorphic over the
interior of $S$ and satisfies the given boundary conditions.

An alternative approach to existence is to use the invariance of the count of
holomorphic sections of $W$ upon deforming the given boundary
condition to a product one,
given by the same torus (in terms of the coordinates $z_i$) in all the fibers of $W$ over $\partial S$;
i.e.\ we modify the problem so that the prescribed value of $|z_i|$ is the
same at every point of $\partial S$, rather than possibly varying from one
point to another. (This can viewed either as deforming the totally real
boundary condition being imposed on the sections of $W$, or as keeping the
same Lagrangian boundary condition but modifying the
coordinates and the complex structure on $Y_\alpha$ by rescaling each of $z_1,\dots,z_n$ by an amount 
which varies smoothly over $S$.)
After this deformation, one is led to look for holomorphic
maps from $S$ to $(\C^*)^n$ such that $|z_i|$ is equal to a fixed constant
at every point of $\partial S$: in other terms, holomorphic discs
(parametrized by $S$) in $(\C^*)^n$ with boundary on a fixed product torus.
By the maximum principle, the only solutions are constant maps, and these are regular. 
Thus, in the deformed setting, the moduli space of sections consists of
precisely one $T^n$-orbit,
and the count of holomorphic sections through a given point is equal to one.
Because of the homotopy invariance of Floer-theoretic section-counting invariants under
deformations, it follows that the moduli space of sections for our initial
problem is also non-empty, consisting of a single $T^n$-orbit, and the count
of sections through a given point is equal to one.
\endproof

\begin{remark}
The argument can be simplified if we assume that
$\xi=-\frac{1}{2\pi}\mathrm{val}(x)$ lies
in the intersection of $n$ of the subsets $S_{\mathbf{v},\gamma}$,
$\mathbf{v}\in \mathcal{V}$ defined by \eqref{eq:Svgamma}; since non-empty such intersections always
exist, and our comparison of $f$ and $g$ only requires us to determine
the differential $s_x$ for $x$ of arbitrary fixed valuation, this simpler
setting would in fact suffice for our purposes. When $\xi$ lies in the
intersection of $n$ of the $S_{\mathbf{v},\gamma}$, by
Proposition \ref{prop:invce_Y} we can choose the toric monomials
$z_1,\dots,z_n$ in the above argument in such a way that they are all
invariant under parallel transport along $\partial S$ at all points of
$\mathfrak{t}_x$. This implies that the radii $|z_i|$ of the boundary
tori remain constant all along $\partial S$ (i.e., the boundary condition
consists of the same product torus
in $(\C^*)^n$ over each point of $\partial S$); we can then directly 
classify the holomorphic sections without appealing to complex analysis
nor to a deformation argument.
\end{remark}

Each of the families of holomorphic sections identified in Proposition
\ref{prop:torusdiscs}, representing a relative homology class $[D_\alpha]\in
H_2(Y,T_x(t'')\cup T_x(t'))$, contributes to the Floer differential on
$CF^*_{vert}(T_x(t''),T_x(t'))$ with a weight
\begin{equation}\label{eq:discweight}
\mathrm{weight}([D_\alpha])=t^{\,\raisebox{5pt}{$\int_{[D_\alpha]}\omega$}}\,
\mathrm{hol}([\partial D_\alpha])\,\exp(\textstyle\int_{[D_\alpha]}\mathfrak{b})\in \Lambda_{\geq 0}.
\end{equation}
In this formula, $\mathrm{hol}([\partial D_\alpha])$ denotes the holonomy of the local
system along the boundary of $D_\alpha$, which requires some clarification. Since the local
systems on $T_x(t')$ and $T_x(t'')$ are isomorphic over $T_x(t')\cap T_x(t'')$ 
(canonically over $W^{-1}(-1)$, and in a preferred manner up to a 
constant factor over $W^{-1}(c_{t'',t'})$ using the $T^n$-equivariant structure of $\mathfrak{t}_x$),
they can be glued into a local system on the portion of $T_x(t')\cup T_x(t'')$
which fibers over $\partial S$. Noting that this subset of $T_x(t')\cup
T_x(t'')$ can be deformed
isotopically to a product torus in $Y$, we choose the gluing at
$W^{-1}(c_{t'',t'})$ in such a way that the holonomy of the local system
along a loop which deforms to an orbit of the last $S^1$-factor of the
toric action (with moment map $\eta$) is equal to identity. (Meanwhile,
the holonomies along the first $n$ circle factors, within the fibers of $W$,
coincide with those of $\mathfrak{t}_x$.) With this choice
in hand, we define $\mathrm{hol}([\partial D_\alpha])$ to be the holonomy of the local system on $T_x(t')\cup T_x(t'')$ along
the boundary of $D_\alpha$.
Also, we denote by $\mathfrak{b}$ a representative of the bulk deformation
class which is supported near $W^{-1}(0)$ (so its pairing with $[D_\alpha]$ is well
defined). Specifically, we choose the bulk deformation to be of the form
\begin{equation}\label{eq:bulkcoeffs}
\mathfrak{b}=\sum_{\alpha \in P_\Z} \mathfrak{b}_\alpha\,\delta_{Z_\alpha},
\end{equation}
where the constants $\mathfrak{b}_{\alpha}\in \Lambda_{\geq 0}$ are coefficients to be
determined later, and $\delta_{Z_\alpha}$ is a representative of the
cohomology class Poincar\'e dual to the divisor $Z_\alpha$, supported in a
small neighborhood of $Z_\alpha$. Since $[D_\alpha]$ has intersection number
one with $Z_\alpha$ and zero with the other components of $W^{-1}(0)$, we
find that
$\exp(\textstyle\int_{[D_\alpha]}\mathfrak{b})=\exp(\mathfrak{b}_\alpha)$.

\begin{proposition}\label{prop:torusdiscweights}
For all $t''>t'+t_0$, $\alpha \in P_\Z$, and $x\in(\K^*)^n$, there exists a
nonzero constant $K_\xi(t'',t')$
depending only on $t',t''$ and $\xi=-\frac{1}{2\pi}\mathrm{val}(x)$ 
such that the weight of a holomorphic section of $W:Y\to\C$
bounded by $T_x(t'')\cup T_x(t')$ and representing the class $[D_\alpha]$ is given by
\begin{equation}\label{eq:torusdiscweights}
\mathrm{weight}([D_\alpha])=K_\xi(t'',t')\,t^{2\pi \nu(\alpha)}x^\alpha\,\exp(\mathfrak{b}_\alpha).
\end{equation}
\end{proposition}

\proof The portion of $T_x(t'')\cup T_x(t')$ which fibers over $\partial S$
can be deformed by an isotopy into a product torus in $Y$ (by deforming $S$
to a disc), so $H_2(Y,T_x(t'')\cup T_x(t'))\simeq H_2(Y,T^{n+1})\simeq
\Z^{P_\Z}$ (where the latter isomorphism follows from standard facts in toric
geometry). Concretely, this means that relative
homology classes are uniquely determined by their algebraic intersection
numbers with each of the toric divisors $Z_\alpha$.

Let $\alpha_1,\alpha_2\in P_\Z$ be two lattice points which are connected
by an edge in the subdivision $\mathcal{P}$ of $P$ determined by the
tropicalization of the Laurent polynomial $f$ (see Section \ref{s:LGmodel}), 
i.e.\ such that the toric divisors $Z_{\alpha_1},Z_{\alpha_2}\subset Y$ intersect along an
$(n-1)$-dimensional toric stratum $Z_{\alpha_1\alpha_2}$. In terms of the
moment polytope $\Delta_Y$, $Z_{\alpha_1\alpha_2}$ corresponds to 
the codimension 2 stratum of points $(\xi,\eta)$ where $\alpha_1$ and
$\alpha_2$ both achieve the maximum in the piecewise linear polynomial
$\varphi$, and 
\begin{equation}\label{eq:faceZalpha1alpha2}
\eta=\varphi(\xi)=\langle \alpha_1,\xi\rangle-\nu(\alpha_1)=
\langle \alpha_2,\xi\rangle-\nu(\alpha_2).
\end{equation}
The stabilizer of the $T^{n+1}$-action on $Y$ along
$Z_{\alpha_1\alpha_2}$ is the subtorus spanned by the weights
$(-\alpha_1,1)$ and $(-\alpha_2,1)$ (the generators of the two rays of 
the fan $\Sigma_Y$ which span the cone corresponding to
$Z_{\alpha_1\alpha_2}$, or equivalently, the normal vectors to the
face \eqref{eq:faceZalpha1alpha2} of $\Delta_Y$). 
Thus, we can define a 2-chain $D_{\alpha_1\alpha_2}$ in $Y$, with boundary 
in $T_x(t'')\cup T_x(t')$, by considering a path in the complex plane which
connects some $w_1\in \partial S$ to the origin, and in every fiber of $W$
over this path, a suitably chosen orbit of the $S^1$-action with weight
$(\alpha_1-\alpha_2,0)$. We take these $S^1$-orbits
to lie at moment map values which start at $\xi_1=-\frac{1}{2\pi}\mathrm{val}(x)$ over $w_1\in \partial
S$ (so that the boundary of our 2-chain lies in $T_x(t'')\cup T_x(t')$), 
and end at a point $(\xi_0,\eta_0)$ which satisfies \eqref{eq:faceZalpha1alpha2} over the
origin (whence the $S^1$-orbit collapses to a point by our above observation
on the stabilizer along $Z_{\alpha_1\alpha_2}$).

By comparing intersection numbers with the toric divisors of $Y$, we find
that, for a suitable choice of orientation, $[D_{\alpha_1\alpha_2}]=[D_{\alpha_2}]-[D_{\alpha_1}]$. Thus, since the
weight formula \eqref{eq:discweight} is manifestly multiplicative, we
conclude that 
\begin{equation}\label{eq:weightratio}
\mathrm{weight}([D_{\alpha_2}])=\mathrm{weight}([D_{\alpha_1\alpha_2}])\cdot
\mathrm{weight}([D_{\alpha_1}]).
\end{equation}
On the other hand, the weight of $D_{\alpha_1\alpha_2}$ can be
calculated explicitly. Parametrizing this disc by a map $u:D^2\to Y$ and
using polar coordinates $\rho$
(along the path in the moment polytope $\Delta_Y$) and $\theta$ (along the
$S^1$-orbits), and observing that
$\omega(\cdot,\partial_\theta u)=d(\langle \alpha_1-\alpha_2,\xi\rangle)$
by definition of the moment map, we have
$$\int_{D_{\alpha_1\alpha_2}} \!\!\!\omega=\iint_{D^2} \omega(\partial_\rho u,
\partial_\theta u)\,d\rho\,d\theta=2\pi\int_0^1 \partial_\rho(\langle
\alpha_1-\alpha_2,\xi(\rho)\rangle)\,d\rho=2\pi \langle
\alpha_1-\alpha_2,\xi_1-\xi_0\rangle.$$
Since $\xi_0$ satisfies \eqref{eq:faceZalpha1alpha2},
$\langle \alpha_1-\alpha_2,\xi_0\rangle=\nu(\alpha_1)-\nu(\alpha_2)$,
so $$\int_{D_{\alpha_1\alpha_2}}\!\!\omega = \langle
\alpha_2-\alpha_1,\mathrm{val}(x)\rangle+2\pi \nu(\alpha_2)-2\pi
\nu(\alpha_1).$$
Denoting by $y=(y_1,\dots,y_n)$ the holonomies of the local system of $\mathfrak{t}_x$ along
the various circle factors, the holonomy along the boundary of
$D_{\alpha_1\alpha_2}$ is given by $y^{\alpha_1-\alpha_2}$.
Recalling that $x_i=t^{\mathrm{val}(x_i)} y_i^{-1}$, we conclude that the
weight of $D_{\alpha_1\alpha_2}$ is
\begin{eqnarray}
\mathrm{weight}([D_{\alpha_1\alpha_2}])&=&y^{\alpha_1-\alpha_2}\,t^{\langle
\alpha_2-\alpha_1,\mathrm{val}(x)\rangle+2\pi
\nu(\alpha_2)-2\pi\nu(\alpha_1)}\,
\exp(\mathfrak{b}_{\alpha_2}-\mathfrak{b}_{\alpha_1})\\ \nonumber
&=&x^{\alpha_2-\alpha_1}\,t^{2\pi\nu(\alpha_2)-2\pi\nu(\alpha_1)}\,
\exp(\mathfrak{b}_{\alpha_2}-\mathfrak{b}_{\alpha_1}).
\end{eqnarray}
In light of \eqref{eq:weightratio}, and using connectedness of the
1-skeleton of the subdivision $\mathcal{P}$ (i.e., any two elements of $P_\Z$
can be connected via a sequence of elements of $P_\Z$ such that
the above calculation can be applied to consecutive terms in the sequence),
this implies that for fixed $t',t'',x$, the weight of $D_\alpha$ is
proportional to 
\begin{equation}\label{eq:torusdiscweightnoconstant}
x^\alpha t^{2\pi\nu(\alpha)} \exp(\mathfrak{b}_\alpha).
\end{equation}
This is basically the desired formula \eqref{eq:torusdiscweights}, except 
we have not yet shown that the scaling constant depends only on the
valuation of $x$ (and $t',t'')$ rather than on $x$ itself.

To show the constant only depends on $\xi$ (and $t',t''$), we observe that for fixed
$\xi=-\frac{1}{2\pi}\mathrm{val}(x)$, the only role played by $x$ 
is in determining the holonomy of the local system. Recalling that
$T_x(t'')\cup T_x(t')$ (after restriction to $\partial S$) is isotopic
to a product torus $\mathfrak{t}_x\times S^1\simeq T^{n+1}$ in $Y$, and noting that the boundary
of $D_\alpha$ represents the class $(-\alpha,1)$ in
$\pi_1(T_x(t'')\cup T_x(t))\simeq \pi_1(T^{n+1})\simeq \Z^{n+1}$,
we find that $\mathrm{hol}([\partial D_\alpha])=y^{-\alpha}$, so that
the dependence of the weight of $D_\alpha$ on $x$ is indeed as in
\eqref{eq:torusdiscweightnoconstant}, and the scaling factor $K_\xi(t'',t')$ 
does not depend on the holonomy, i.e.\ it depends only on $\xi=-\frac{1}{2\pi}\mathrm{val}(x)$ and not on $x$ itself.
\endproof

We now return to the problem of counting $J$-holomorphic sections of
$W:Y\to\C$ with boundary on $T_x(t')\cup T_x(t'')$.
As previously noted, when $J$ converges to the standard complex structure
$J_0$, the $J$-holomorphic discs contributing
to the differential \eqref{eq:s0torus} limit to stable curves consisting
of a holomorphic disc, representing
one of the classes $[D_\alpha]$ for some $\alpha\in P_\Z$ (by Proposition
\ref{prop:torusdiscs}), and a (possibly
empty) configuration of rational curves contained in $Z=W^{-1}(0)$,
representing some homology class $\beta\in H_2(Y)$ (with $[\omega]\cdot
\beta>0$ whenever $\beta\neq 0$).

\begin{definition}
For fixed $t',t'',\xi$, and for each $\alpha\in P_\Z$ and $\beta\in H_2(Y)$, we denote by
$n_{\alpha,\beta}$ the (signed) count of $J$-holomorphic sections of $W$
(for generic $J$ close to $J_0$) whose relative homology class in $H_2(Y,T_x(t')\cup T_x(t''))$ is equal to
$[D_\alpha]+\beta$, passing through a generic point of $\mathfrak{t}_x\subset
W^{-1}(-1)$. 
\end{definition}

By considering the limit as $J\to J_0$ and using the classification of
holomorphic discs in Proposition \ref{prop:torusdiscs}, we
see that every $J$-holomorphic section under consideration is in one of
these homology classes, $n_{\alpha,0}=1$ for all $\alpha\in P_\Z$,
and $n_{\alpha,\beta}=0$ for all $\beta\neq 0$ such that $[\omega]\cdot
\beta\leq 0$.

\begin{remark}
The invariance of counts of holomorphic sections under deformations of the
Lagrangian boundary condition implies that
$n_{\alpha,\beta}$ is independent of $t'$, $t''$ (as long as $t''-t'>t_0$) and
$\xi$; hence the notation. However, our argument does not depend on it, so we will not elaborate
further.
\end{remark}

Since the weight of a section in the class $[D_\alpha]+\beta$ is given by
$$\mathrm{weight}([D_\alpha]+\beta)=\mathrm{weight}([D_\alpha])\,
t^{[\omega]\cdot \beta}\,\exp([\mathfrak{b}]\cdot \beta),$$
we arrive at:

\begin{proposition}\label{prop:correctedcount}
The Laurent polynomial $g$ of Propositions
\ref{prop:HWL0}--\ref{prop:s0torus} satisfies
\begin{equation}\label{eq:correctedcount}
C'_\xi(t'',t')\,g(x)=K_\xi(t'',t')\,\sum_{\alpha\in P_\Z} t^{2\pi\nu(\alpha)}
x^{\alpha} \exp(\mathfrak{b}_\alpha)\,
\Bigl(1+\sum_{\substack{\beta\in H_2(Y)\\ {}[\omega]\cdot
\beta>0}} n_{\alpha,\beta}\,t^{[\omega]\cdot \beta}
\exp([\mathfrak{b}]\cdot \beta)\Bigr).
\end{equation}
\end{proposition}

\proof This follows directly from a comparison of the weighted counts of sections
which determine the differential on \eqref{eq:s0torus}
(the coefficient of $1_x^{t''\to t'}$ in $s_x(\delta_x^{t''\to t'})$)
as given by Proposition \ref{prop:s0torus} and by direct calculation of
$\sum_{\alpha,\beta} n_{\alpha,\beta}\,\mathrm{weight}([D_\alpha]+\beta)$.
\endproof

\begin{corollary}\label{cor:correctedcount}
There exists a constant $C\in \K^*$ such that 
\begin{equation}\label{eq:correctedcount2}
g(x)=C\,\sum_{\alpha\in P_\Z} t^{2\pi\nu(\alpha)}
x^{\alpha} \exp(\mathfrak{b}_\alpha)\,
\Bigl(1+\sum_{\substack{\beta\in H_2(Y)\\ {}[\omega]\cdot
\beta>0}} n_{\alpha,\beta}\,t^{[\omega]\cdot \beta}
\exp([\mathfrak{b}]\cdot \beta)\Bigr).
\end{equation}
\end{corollary}

\proof The key point is that, for any $\xi \in \R^n$, the coefficients of a Laurent polynomial in
$\K[x_1^{\pm 1},\dots,x_n^{\pm 1}]$ are determined by its evaluation at
points $x\in (\K^*)^n$ with fixed valuation $\mathrm{val}(x)=-2\pi\xi$.
Thus, comparing the left- and right-hand sides of \eqref{eq:correctedcount}
for fixed $\xi,t'',t'$ we find that $g(x)$ and the Laurent polynomial
appearing in the right-hand side coincide up to a constant factor.
Incidentally, this also implies that the ratio
$C'_\xi(t'',t')/K_\xi(t'',t')$ is a genuine constant independent of
$t',t''$ and $\xi$, and that the power series appearing as coefficients
in the right-hand side are independent of $t',t''$ and $\xi$ (in general this is slightly weaker
than asserting that the $n_{\alpha,\beta}$ themselves are independent of
these choices.)
\endproof

\begin{remark} \label{rem:mirror-map}
The power series in the right-hand side of \eqref{eq:correctedcount2} are
also exactly those which appear in expressions for the instanton-corrected
superpotential for product tori in the toric Calabi-Yau variety $Y$
(cf.\ e.g.\ \cite{AAK} and \cite{ChanLauLeung}), and more explicitly
in terms of Gromov-Witten invariants in \cite{ChanLauTseng}, where these quantities
are also interpreted as correction terms in the mirror map for the toric
variety $Y$. Indeed, deforming (a subset of) $T_x(t')\cup T_x(t'')$ to a
product torus in $Y$ it is apparent that the enumerative geometry problems
we consider here and those discussed in \cite{AAK,ChanLauLeung,ChanLauTseng} are equivalent.
\end{remark}

\begin{example}
Let $f(x)=t^{2\pi}x^{-1}+1+x$, so $H=f^{-1}(0)$ consists of two points. Then
$P_\Z=\{-1,0,1\}$, $\varphi(\xi)=\max(-\xi-1,0,\xi)$, and
$Y$ is isomorphic to the total space of $\O(-2)\to \CP^1$. In this example,
the term in \eqref{eq:correctedcount2} corresponding to $\alpha=0$
(i.e., discs in $Y$ which intersect the zero section $\CP^1$)
includes a non-trivial contribution from $\beta=[\CP^1]$, with
$n_{\alpha=0,[\CP^1]}=1$, whereas all the other $n_{\alpha,\beta}$ are zero (cf.\ e.g.\ \cite[Example
5.3.1]{ChanLauLeung}). Hence, $g(x)$ is
proportional to $$e^{\bb_{-1}}t^{2\pi}x^{-1}+e^{\bb_0}(1+t^{2\pi}e^{[\bb]\cdot[\CP^1]})
+e^{\bb_1}x,$$
which matches $f(x)$ when $\bb_1=\bb_{-1}=0$ and
$e^{\bb_0}(1+t^{2\pi}e^{-2\bb_0})=1$.
See also \cite[\S 5.3]{ChanLauLeung} for examples where infinitely many
$n_{\alpha,\beta}$ are non-zero.
On the other hand, 
the coefficients $n_{\alpha,\beta}$ all vanish
when every rational curve in $Y$ is contained in a toric stratum of complex codimension at least two.
\end{example}

Finally, we observe that, as in the above example, it is always possible by a suitable choice of the bulk deformation class
$[\mathfrak{b}]\in H^2(Y,\Lambda_{\geq 0})$ to ensure that the right-hand side
of \eqref{eq:correctedcount2} matches the Laurent polynomial $f$ used to
define the hypersurface $H$. 

\begin{proposition}\label{prop:matchbulk}
Given any collection of unitary (i.e., valuation zero) elements $a_\alpha\in \K^*$ 
for all $\alpha\in P_\Z$, there exists a unique collection of
unitary elements $e^{\mathfrak{b}_\alpha}\in \K^*$, $\alpha\in P_\Z$, such that
\begin{equation}\label{eq:matchbulk}
e^{\mathfrak{b}_\alpha}\,
\Bigl(1+\sum_{\substack{\beta\in H_2(Y)\\ {}[\omega]\cdot
\beta>0}} n_{\alpha,\beta}\,t^{[\omega]\cdot \beta}
e^{[\mathfrak{b}]\cdot \beta}\Bigr)=a_\alpha \quad \text{for all}\
\alpha\in P_{\Z}.\end{equation}
\end{proposition}

\proof We can solve for $e^{\bb_\alpha}$ order by order.
Namely, the series $\sum_\beta n_{\alpha,\beta}
t^{[\omega]\cdot \beta}\,e^{[\bb]\cdot \beta}$ consist of terms whose
valuations are positive and bounded below by some constant $\lambda>0$
(by Gromov compactness,
the symplectic areas of the rational curves which can appear in these
expressions form a discrete set).
Thus, \eqref{eq:matchbulk} implies that $e^{\bb_\alpha}=a_\alpha\ \mod\ t^\lambda$.
Moreover, once $e^{\bb_\alpha}$ is determined mod $t^{N\lambda}$ for some
$N\in\N$ and for all $\alpha\in P_\Z$, the power series appearing in the
left-hand side of \eqref{eq:matchbulk} are determined mod
$t^{(N+1)\lambda}$, and thus \eqref{eq:matchbulk} determines
$e^{\bb_\alpha}$ mod $t^{(N+1)\lambda}$ for all $\alpha\in P_\Z$.
\endproof

\proof[Proof of Theorem \ref{thm:main}]
We equip $Y$ with the bulk deformation class $\mathfrak{b}=\sum \bb_\alpha
\delta_{Z_\alpha}$, where the coefficients $\bb_\alpha$ are determined by Proposition
\ref{prop:matchbulk} so that the expression \eqref{eq:correctedcount2} 
agrees with the Laurent polynomial $f$ in \eqref{eq:intropoly} up to
scaling by a nonzero constant. The result then follows from Proposition
\ref{prop:HWL0} and Corollary \ref{cor:correctedcount}. \endproof

\section{Complete intersections} \label{s:ci}

\subsection{Geometric setup} \label{ss:ci_setup}
In this section we describe the geometric setup for extending Theorem \ref{thm:main}
to complete intersections in $(\C^*)^n$. 
Consider $k$ Laurent polynomials
\begin{equation}\label{eq:f_i}
f_i=\sum_{\alpha \in P_{i,\Z}} a_{i,\alpha}\,t^{2\pi\nu_i(\alpha)}x^\alpha
\in \K[x_1^{\pm 1},\dots,x_n^{\pm 1}],
\qquad 1\leq i\leq k,
\end{equation}
where the finite subsets $P_{i,\Z}\subset \Z^n$, the exponents
$\nu_i(\alpha)\in\R$, and the coefficients $a_{i,\alpha}$ ensure that the
hypersurfaces $H_i=f_i^{-1}(0)$ satisfy the same
``tropical smoothness'' conditions as in Section \ref{s:LGmodel}, and
additionally we assume that the tropical hypersurfaces associated to the
tropicalizations
\begin{equation}\label{eq:phi_i}
\varphi_i(\xi)=\max\{\langle \alpha,\xi\rangle-\nu_i(\alpha)\,|\,\alpha\in
P_{i,\Z}\}
\end{equation} are in generic position relative to each other 
(i.e., all intersections between strata are transverse).
Following \cite[Section 11]{AAK}, we define $Y$ to be the K\"ahler toric
$(n+k)$-fold defined by the moment polytope
\begin{equation}
\Delta_Y=\{(\xi,\eta_1,\dots,\eta_k)\in \R^n\oplus \R^k\,|\,\eta_i\geq
\varphi_i(\xi)\ \forall i=1,\dots,k\}.
\end{equation}
Dually, $Y$ is also described by a fan $\Sigma_Y\subseteq \R^n\oplus \R^k$, whose rays are
generated by the integer vectors $(-\alpha,e_i)$ for all $1\leq i\leq k$
and $\alpha\in P_{i,\Z}$, where 
$e_1,\dots,e_k$ is the standard basis of $\Z^k$.

For $1\leq i\leq k$, we define $W_i:Y\to\C$ to be the negative of the toric
monomial with weight $(0,e_i)=(0,\dots,0,1,0,\dots,0)$ (where the 1 is
in the $(n+i)$-th position). (Thus, the zero set of $W_i$ is the union of the toric divisors of $Y$
corresponding to the rays of $\Sigma_Y$ generated by
$(-\alpha,e_i)$, $\alpha\in P_{i,\Z}$, or equivalently, to 
the facets of $\Delta_Y$ on which $\eta_i=\varphi_i(\xi)$.)
The candidate mirror to the complete intersection $\mathbf{H}=H_1\cap\dots
\cap H_k$ is then the Landau-Ginzburg model $(Y,W_1+\dots+W_k)$; however,
our version of the (fiberwise wrapped) Fukaya category of this
Landau-Ginzburg model will involve Lagrangian submanifolds which are
simultaneously admissible for each of the projections $W_1,\dots,W_k$. 
Accordingly, we view our $k$ monomials as the components of a (toric) map 
$$\mathbf{W}=(W_1,\dots,W_k):Y\to\C^k.$$ We call $(Y,\mathbf{W})$ the
{\em toric Landau-Ginzburg mirror} to the complete intersection $\mathbf{H}$ determined
by the Laurent polynomials
$(f_1,\dots,f_k)$.

In the course of our argument, we will also
consider mirrors of partial intersections determined by a subset of the
Laurent polynomials $f_1,\dots,f_k$.
Given any subset $I\subset\{1,\dots,k\}$, denote by $W_I=(W_i)_{i\in I}:Y\to
\C^{|I|}$ the projection of $\mathbf{W}$ onto the subset of coordinates
associated to $I$. We also write $\overline{I}=\{1,\dots,k\}\!-\!I$ for the complement of $I$. 

\begin{proposition}\label{prop:ci_fiber}
Given 
any $c_{\overline{I}}\in
(\C^*)^{k-|I|}$, the submanifold
$Y_I=W_{\overline{I}}^{-1}(c_{\overline{I}})\subset Y$
equipped with the restriction of $W_I$ is isomorphic (as a
toric K\"ahler manifold together with an $|I|$-tuple of monomials) to the toric
Landau-Ginzburg mirror of the complete intersection determined by
$(f_i)_{i\in I}$.
\end{proposition}

\noindent For $I=\emptyset$, this says that the fiber of $\mathbf{W}$ over
a point of $(\C^*)^k$ is isomorphic to $(\C^*)^n$.

\proof Algebraically, $W_{\overline{I}}:Y\to \C^{k-|I|}$ is a dominant toric
morphism, induced by the morphism of fans from $\Sigma_Y$ to the fan
of $\C^{k-|I|}$ induced by the linear
map from $\R^n\oplus \R^k$ to $\R^{k-|I|}$ given by projection to 
the $(n+i)$-th coordinates for all $i\in \overline{I}$ (we call these the
components {\em indexed by} $\overline{I}$). Thus, the fibers of
$W_{\overline{I}}$ over the points of the open dense orbit $(\C^*)^{k-|I|}$
are all isomorphic, and described by the fiber of the morphism of fans over
the trivial cone $\{0\}$, i.e.\ the intersection of $\Sigma_Y$ with the
subspace $\R^n\oplus \R^{I}\subset \R^n\oplus\R^k$; or, dually, the
projection of $\Delta_Y$ from $\R^n\oplus \R^k$ onto $\R^n\oplus \R^{I}$
given by forgetting the components $\eta_i$ for $i\in \overline{I}$. This
agrees exactly with the toric variety $Y_I$ obtained by applying our
construction to the complete intersection determined by the Laurent
polynomials $f_i$ for $i\in I$. Moreover, it is clear that the monomials
$W_i$ for $i\in I$ restrict from $Y$ to $Y_I$ in the expected manner
(the toric weights match after forgetting the components indexed by the
elements of $\overline{I}$). 

Symplectically, we observe that the moment map $\mu_I:Y\to \R^n\times \R^I$ 
for the action of $\T^n\times
\T^I$ (the subtorus which preserves the fibers of $W_{\overline{I}}$) is
obtained from the moment map $\mu$ of the $\T^{n+k}$-action on $Y$ by forgetting the
components indexed by the elements of $\overline{I}$. The image of $\mu_I$
is therefore $$\Delta_{Y|I}=\{(\xi,(\eta_i)_{i\in I})\,|\,\eta_i\geq
\varphi_i(\xi)\ \forall i\in I\}\subset \R^n\oplus \R^I.$$ 
Moreover,
$W_{\overline{I}}$ maps every stratum of $Y$ on which $(\C^*)^{\overline{I}}$ acts
freely (i.e., the strata where $\eta_i>\varphi_i(\xi)$ $\forall i\in \overline{I}$) onto the open 
stratum $(\C^*)^{k-|I|}$; this implies that every such stratum intersects
$W_{\overline{I}}^{-1}(c_{\overline{I}})$. In particular,
$W_{\overline{I}}^{-1}(c_{\overline{I}})$ contains points in strata which
map to the vertices of $\Delta_{Y|I}$ under $\mu_I$, as well as strata which
map to its unbounded edges. By convexity of the moment map image (and 
given that there are no other toric fixed points, hence no additional vertices), this
implies that the restriction of $\mu_I$ to $W_{\overline{I}}^{-1}(c_{\overline{I}})$ is surjective onto $\Delta_{Y|I}$.
Thus the K\"ahler form on the generic fiber of
$W_{\overline{I}}$ has moment polytope equal to $\Delta_{Y|I}$, as expected.
\endproof

\begin{example}
One case where the geometry of $(Y,\mathbf{W})$ is particularly simple is
when $\mathbf{H}$ is a product of hypersurfaces
in $(\C^*)^{n_i}$, $i=1,\dots,k$, i.e.\ each Laurent polynomial $f_i$
involves a different subset of the coordinates $x_1,\dots,x_n$ ($n=\sum
n_i$). In this case, $Y$ ends up being the product of the mirrors we
associate to each hypersurface $f_i^{-1}(0)\subset (\C^*)^{n_i}$, with
$W_1,\dots,W_k$ the (pullbacks of the) respective superpotentials.
In general $Y$ is not a product, but the above considerations nonetheless make it
possible to argue in terms of subsets of the collection $\{f_1,\dots,f_k\}$.
\end{example}

We can also describe the toric K\"ahler manifold $Y$ in terms of toric reduction, as
we have done in \S \ref{ss:ham_red} for the case of hypersurfaces. We start from the
product $\prod_{i=1}^k \C^{P_{i,\Z}}$, equipped with the product of the toric
K\"ahler forms described in \S \ref{s:toricCN}. Denote by $M$ the kernel of
the surjective map 
\begin{equation}\label{eq:ci_surjlattices}
\textstyle \prod_{i=1}^k \Z^{P_{i,\Z}} \to \Z^n\oplus \Z^k
\end{equation}
which maps the generator corresponding to $\alpha\in
P_{i,\Z}$ to the element $(-\alpha,e_i)$ of $\Z^n\oplus \Z^k$, and by
$\T_M=M\otimes (\R/\Z)$ the corresponding subtorus of $\prod \T^{P_{i,\Z}}$.
Dualizing \eqref{eq:ci_surjlattices} we have a short exact sequence
$$0\to \R^{n+k}\stackrel{\iota}{\longrightarrow} \prod \R^{P_{i,\Z}}
\stackrel{\pi}{\longrightarrow} M_\R^*\to 0,$$
where the first map is given by
$$\iota(\xi_1,\dots,\xi_n,\eta_1,\dots,\eta_k)=\bigl(-\langle
\alpha,\xi\rangle+\eta_i\bigr)_{\alpha\in P_{i,\Z},\ 1\leq i\leq k}.$$
Viewing the exponents $\nu_i(\alpha)$ in \eqref{eq:f_i} as an
element $(\nu_1,\dots,\nu_k)$ of $\prod \R^{P_{i,\Z}}$, we
consider the reduction of $\prod \C^{P_{i,\Z}}$ by $\T_M$ at the level
$\lambda=\pi(\nu_1,\dots,\nu_k)$, and observe that
$$\mu^{-1}(\lambda)/\T_M\simeq Y,$$ since the moment polytope for the
action of $\T^{n+k}\simeq (\prod \T^{P_{i,\Z}})/\T_M$ on the reduced space
is the intersection of $\pi^{-1}(\lambda)=\mathrm{Im}(\iota)+(\nu_1,\dots,\nu_k)$ with the non-negative orthant in
$\prod \R^{P_{i,\Z}}$, which is naturally identified with $\Delta_Y$. 

The toric K\"ahler manifold $Y$, its K\"ahler form
$\omega_Y$, and $\mathbf{W}=(W_1,\dots,W_k)$ are thus obtained by
Hamiltonian reduction from the product of the spaces
$\C^{P_{i,\Z}}$ for $i=1,\dots,k$, each equipped with the toric K\"ahler form of 
\S \ref{s:toricCN} and the functions 
$W_{0,i}=-\prod_{\alpha\in P_{i,\Z}}
z_{i,\alpha}:\C^{P_{i,\Z}}\to\C$. (More precisely: the pullback of $W_{0,i}$ to
$\prod \C^{P_{i,\Z}}$ is $\T_M$-invariant and descends to $W_i:Y\to\C$.)

This description of $(Y,\mathbf{W})$ as a
reduction of the product of $k$ ``standard'' Landau-Ginzburg models
$(\C^{P_{i,\Z}},W_{0,i})$ corresponds to viewing 
$\mathbf{H}$ as the intersection of an $n$-dimensional algebraic subtorus  of the open
stratum of $\prod_{i=1}^k \PP(\K^{P_{i,\Z}})$ with a product of
$(|P_{i,\Z}|-2)$-dimensional pairs of pants, as in Remark \ref{rmk:reduction}.

\subsection{The fiberwise wrapped Fukaya category of $(Y,\mathbf{W})$}

The construction of the partially wrapped Fukaya category $\Wrap(Y,\mathbf{W})$ parallels that
introduced in  Section~\ref{s:Fukayacat}, except we now consider properly
embedded Lagrangian submanifolds of $Y$ whose image under {\em each} of the
projections $W_i:Y\to\C$ agrees outside of a compact subset with a finite
union of admissible arcs in the complex planes; in fact, we shall only
consider Lagrangians which fiber over
product of U-shaped arcs (the same arcs $\gamma_t$ as in our main
construction) with respect to $\mathbf{W}:Y\to\C^k$.

As before, we control the behavior of holomorphic curves by equipping $Y$ with
a compatible almost-complex structure $J$ making each of $W_1,\dots,W_k$
holomorphic outside of a neighborhood of the zero fiber (as before, $J$ will
be taken to agree with the standard complex structure of $Y$ except for a
small perturbation near $\bigcup_i W_i^{-1}(0)$), and by choosing
a continuous weakly plurisubharmonic function $h:Y\to [0,\infty)$ which
is proper on the fibers of $\mathbf{W}$; 
in addition, we fix 
a non-negative wrapping Hamiltonian $H:Y\to\R$. The functions $H$ and $h$ are required to satisfy the same
conditions as in Section \ref{s:Fukayacat} with respect to {\em each} of
$W_1,\dots,W_k$, i.e.\ with respect to the whole horizontal distribution
given by the symplectic orthogonals to the fibers of $\mathbf{W}:Y\to\C^k$,
thus ensuring that the maximum principle estimates of \S \ref{s:Fukayacat} (with
respect to $h$ and to the various $|W_i|$) continue to hold. Specific
choices of $h$ and $H$ satisfying these requirements are given below.

\subsubsection{Parallel transport preserves fiberwise monomial admissibility}
The function $h$ is again defined as the maximum of the
(rescaled) norms of certain monomials $z^\v\in \O(Y)$ for $\v$ in a set of
``extremal'' vectors $\mathcal{V}$ (primitive integer vectors parallel to
the unbounded edges of $\Delta_Y$),
\begin{equation}\label{eq:ci_hmaxzv}
h=\max\{|z^\v|^{1/\delta(\v)},\ \v\in\mathcal{V}\},
\end{equation}
where $\delta(\v)$ is defined below in \eqref{eq:ci_dofv}.
As in the case of hypersurfaces, the key point which ensures that $h$ has
all the required properties is that, at every point outside of a bounded
subset of each fiber of $\mathbf{W}$, any monomial $z^\v$ which achieves the maximum 
in \eqref{eq:ci_hmaxzv} is invariant under parallel transport between the
fibers of $\mathbf{W}$ (Propositions \ref{prop:ci_invce_Y} and
\ref{prop:ci_largest_invce_Y} below). This property, which amounts to a compatibility of fiberwise
monomial admissibility with parallel transport, is proved similarly to the
arguments in Section \ref{ss:invce_Y}.

Given a vector $\v=(\vec{v},v^{1,0},\dots,v^{k,0})\in \Z^n\oplus \Z^k$, the
toric monomial $z^\v$ defines a regular function on $Y$ if and only
\begin{equation}\label{eq:ci_zvorder}
v^{i,\alpha}:=(-\alpha,e_i)\cdot \v=v^{i,0}-\alpha\cdot\vec{v}\geq 0
\qquad\text{for all}\ 1\leq i\leq k\text{ and }\alpha\in P_{i,\Z}.
\end{equation}
In fact $z^\v$ vanishes to order $v^{i,\alpha}$ along the toric divisor
of $Y$ which corresponds to the ray $(-\alpha,e_i)$ of the fan $\Sigma_Y$.
Next we observe that the monomial $$\prod_{i=1}^k \prod_{\alpha\in P_{i,\Z}}
z_{i,\alpha}^{v^{i,\alpha}}\in \,\O\!\left({\textstyle \prod} \C^{P_{i,\Z}}\right)$$
is invariant under the action of $\T_M$ and descends to $z^\v\in \O(Y)$
under reduction.

For $\v\in \Z^{n+k}$ satisfying \eqref{eq:ci_zvorder}, $i\in \{1,\dots,k\}$,
and $\gamma>0$ small, we define a subset $S_{\v,i,\gamma}$ of $\R^n$
as in \eqref{eq:Svgamma}, namely we set
\begin{equation}\label{eq:ci_Svgamma}
S_{\v,i,\gamma}=\{\xi\in \R^n\,|\,\langle \alpha,\xi\rangle-\nu_i(\alpha)<
\varphi_i(\xi)-\gamma\|\xi\|\ \forall \alpha \in P_{i,\Z}\ \text{such
that}\ v^{i,\alpha}>0\}.
\end{equation}
The exact same argument as in the proof of Proposition \ref{prop:invce_Y}
then shows:

\begin{proposition} \label{prop:ci_invce_Y}
Given $\v\in\Z^{n+k}$ satisfying \eqref{eq:ci_zvorder} and $i\in
\{1,\dots,k\}$, the monomial $z^\v\in \O(Y)$ is locally invariant under
parallel transport between the fibers of the map $W_i:Y\to\C$ at every point $z\in
Y$ whose moment map coordinates $(\xi,\eta)$ satisfy $\xi\in S_{\v,i,\gamma}$
as well as lower bounds on $|W_i(z)|$ and on $\|\xi\|$ as in Proposition
\ref{prop:invce_Y}. \qed
\end{proposition}

The first consequence, setting $\v=(0,e_j)$
and observing that $S_{(0,e_j),i,\gamma}=\R^n$ for all $i\neq j$,
is that $W_j=-z^{(0,e_j)}$ is invariant under parallel transport in the
direction of $W_i$ for all $i\neq j$. (Inspection of the argument shows that
in this case no restriction on $|W_i(z)|$ or on
$\|\xi\|$ is needed: the point is that the lift of $W_j$ to $\prod \C^{P_{i,\Z}}$ 
only involves the variables $z_{j,\alpha}$, all of
which are preserved under parallel transport for the $i$-th component.)
This ensures that 
the parallel transports along the different factors in the base of the fibration
$\mathbf{W}:Y\to\C^k$ commute with each other, and that
the parallel transport of a Lagrangian in a fiber of $\mathbf{W}$ over a product of 
arcs in $\C^k$ is well-defined.

Next, to each $\vec{v}\in \Z^n$, we associate an element of $\Z^{n+k}$ as
follows: set $v^{i,0}=\max\{\alpha\cdot \vec{v},\
\alpha\in P_{i,\Z}\}$, and $\v=(\vec{v},v^{1,0},\dots,v^{k,0})$.
 Denote by $A_{\vec{v},i}$ the set of $\alpha\in P_{i,\Z}$ which
achieve the maximum in the definition of $v^{i,0}$, or equivalently, those
$\alpha$ for which $v^{i,\alpha}$ as defined by
\eqref{eq:ci_zvorder} is zero. Denoting by $\Delta_{\alpha,i}$ the
polyhedral subset of $\R^n$ where $\alpha$ achieves the maximum in
$\varphi_i$, we observe that $S_{\v,i,\gamma}$ is nonempty (for sufficiently
small $\gamma$) and is a retract 
of $\bigcup_{\alpha \in A_{\vec{v},i}} \Delta_{\alpha,i}$ obtained by
removing those points which are too close to some other $\Delta_{\alpha',i}$,
$\alpha'\not\in A_{\vec{v},i}$.  We also note that the $\Delta_{\alpha,i}$
appearing in this union are those which are unbounded in the direction of
$\vec{v}$.  Given this, we define $\mathcal{V}$ to be the set of all $\v$
obtained by this process from some $\vec{v}\in \Z^n$ which is the primitive outward normal
vector to any facet of the Newton polytope $P_i$ of any of the Laurent
polynomials $f_i$, $1\leq i\leq k$. Equivalently and much more concisely, the elements of
$\mathcal{V}$ are the primitive tangent vectors to the unbounded edges of
$\Delta_Y$. 

For $\v\in \mathcal{V}$ and $v^{i,\alpha}$ as in \eqref{eq:ci_zvorder}, we set 
\begin{equation}\label{eq:ci_dofv}
\delta(\v)=\sum_{i=1}^k \frac{d_i(\v)}{2N_i}, \qquad
\text{where}\ d_i(\v)=\sum_{\alpha \in P_{i,\Z}} v^{i,\alpha}\quad
\text{and}\ N_i=|P_{i,\Z}|.
\end{equation}

For sufficiently small $\gamma>0$,
and for all $\v\in \mathcal{V}$, $S_{\v,\gamma}:=\bigcap_{i=1}^k S_{\v,i,\gamma}$ is non-empty
(it is a retract of the union of those regions of $\R^n$ delimited by the union of the
tropical hypersurfaces of $\varphi_1,\dots,\varphi_k$ which are unbounded in
the direction of $\vec{v}$), and the union $\bigcup_{\v\in \mathcal{V}}
S_{\v,\gamma}$ covers the complement of a compact subset in $\R^n$. 
We have the following analogue of Proposition \ref{prop:largest_invce_Y}:

\begin{proposition} \label{prop:ci_largest_invce_Y}
There exist positive constants $\gamma_0$ and $K_0$ such that, at every
point $z\in Y$ with $|W_i(z)|^2\geq (\varepsilon
e^{\delta})^{N_i/(N_i-1)}$ $\forall i$ 
and whose moment map coordinates $(\xi,\eta)$ satisfy
$\|\xi\| \geq K_0 |\mathbf{W}(z)|^2$,
if $\v_0\in \mathcal{V}$ achieves the maximum in
\eqref{eq:ci_hmaxzv} then $\xi\in S_{\v_0,\gamma_0}$.
\end{proposition}


\proof  Consider a point
$z\in Y$ and its lift $(z_{i,\alpha})\in \mu^{-1}(\lambda)\subset \prod
\C^{P_{i,\Z}}$. 
For each $i$, let $\alpha_{i,0},\alpha_{i,1}\in P_{i,\Z}$
correspond to the smallest, resp.\ largest $|z_{i,\alpha}|$ (or
equivalently, moment map coordinate $\mu_{i,\alpha}$) of all $\alpha\in
P_{i,\Z}$.  By Lemma \ref{l:momentmapest} (2), up to bounded constant
factors, $\mu_{i,\alpha_{i,0}}\!\sim\! |W_i(z)|^2$, while
$\mu_{i,\alpha_{i,1}}\!\sim\! |z_{i,\alpha_{i,1}}|^{2N_i}$. 
Bounding $\mu_{i,\alpha_{i,1}}\!-\mu_{i,\alpha_{i,0}}$ in terms of $\|\xi\|$
 as in the proof of Proposition
\ref{prop:largest_invce_Y}, we find that $\mu_{i,\alpha_{i,1}}\!\sim\! \|\xi\|$ and 
hence $|z_{i,\alpha_{i,1}}|\!\sim\!\|\xi\|^{1/(2N_i)}$ up to a bounded factor whenever $\|\xi\|\gg |W_i(z)|^2$.

We now proceed as in the proof of Proposition \ref{prop:largest_invce_Y}: if $\xi \in S_{\v,\gamma}$
then $|z_{i,\alpha}|$ satisfies a lower bound \eqref{eq:zalpha_lowerbound} by a constant
multiple of $|z_{i,\alpha_{i,1}}|\!\sim\!\|\xi\|^{1/(2N_i)}$
for all $\alpha\in P_{i,\Z}-A_{\vec{v},i}$ (the constant depends on $\gamma$).
Hence, $|z^\v|$ has a lower bound by a
constant multiple of $\|\xi\|^{\sum v^{i,\alpha}/2N_i}=\|\xi\|^{\delta(\v)}$
(where the constant again depends on $\gamma$). Applying this for some fixed
$\gamma=\gamma_1>0$ such that $\bigcup_{\v\in\mathcal{V}} S_{\v,\gamma}$
covers the complement of a compact subset in $\R^n$, we find that
$h(z)=\max\{|z^\v|^{1/\delta(\v)},\ \v\in\mathcal{V}\}$ is bounded from below by a constant
$c(\gamma_1)$ times $\|\xi\|$ (still assuming that $\|\xi\|\gg |\mathbf{W}|^2$).

Conversely, if $\xi\not\in S_{\v,\gamma}$ for $\gamma>0$ (now chosen much
smaller than $\gamma_1$) then there exists some $i$ and $\alpha\in
P_{i,\Z}-A_{\vec{v},i}$ such that $|z_{i,\alpha}|$ satisfies the upper bound
\eqref{eq:zalpha_upperbound}, which implies that $|z^\v|$ is
bounded by a constant times $\gamma^{1/2N_i}$ times $\|\xi\|^{\delta(\v)}$.
Choosing $\gamma=\gamma_0$ sufficiently small (so that
$\gamma_0^{1/(2N_i\delta(\v))}$ is much smaller than $c(\gamma_1)$), this
implies that $|z^\v|^{1/\delta(v)}$ cannot achieve the maximum in
\eqref{eq:ci_hmaxzv}.
\endproof

Propositions \ref{prop:ci_invce_Y} and \ref{prop:ci_largest_invce_Y} imply
that $h=\max\{|z^\v|^{1/\delta(v)},\ \v\in\mathcal{V}\}$ is invariant under
parallel transport between the fibers of $\mathbf{W}$ outside of a compact
subset of each fiber. This in turn implies, first, that perturbed
holomorphic curves satisfy maximum principles with
respect to $|\mathbf{W}|$ and $h$ as in
Propositions \ref{prop:maxprinciple_base}--\ref{prop:maxprinciple_fiber},
and second, that we can construct admissible
Lagrangian submanifolds of $Y$ by parallel transport of (monomially
admissible) Lagrangian
submanifolds of
the fiber of $\mathbf{W}$ (i.e., $(\C^*)^n$) over products of admissible arcs.

\subsubsection{The wrapping Hamiltonian}
We define the wrapping Hamiltonian $H:Y\to\R$ as in Section
\ref{ss:wrappingHam}: the moment map coordinates of $\prod
\C^{P_{i,\Z}}$ descend to real-valued functions
$\mu_{i,\alpha}$ on $Y$ ($i=1,\dots,k$,
 $\alpha\in P_{i,\Z}$), given by
$$\mu_{i,\alpha}=\eta_i-\langle \alpha,\xi\rangle+\nu_i(\alpha).$$
We then define $H:Y\to\R$ by
\begin{equation}\label{eq:ci_wrappingHam}
H=\sum_{i=1}^k \Bigl( \sum_{\alpha\in P_{i,\Z}}
\mu_{i,\alpha}-|P_{i,\Z}|\,m(\{\mu_{i,\alpha}\}_{\alpha\in P_{i,\Z}})\Bigr),
\end{equation}
where $m$ is a smooth approximation of the minimum function as in Definition
\ref{def:Min}. Propositions \ref{prop:wrapping_ham} and
\ref{prop:wrapping_flatness} carry over with essentially the same proofs.
To summarize:

\begin{proposition}\label{prop:ci_wrappingHam}
The wrapping Hamiltonian $H$ only depends on $(\xi_1,\dots,\xi_n)$, and as a function of these coordinates
it is proper and convex. The flow generated by $H$
preserves the fibers of $\mathbf{W}$, and within each fiber
it preserves monomial admissibility with respect to the collection of
monomials $z^\v$, $\v\in \mathcal{V}$: if $\ell\subset
\mathbf{W}^{-1}(c)$ is monomially admissible with phase
angles $\arg(z^\v)=\varphi_\v$, $\v\in\mathcal{V}$, then its image under the
time $t$ flow is monomially admissible at infinity with phase angles
$\varphi_\v+t\,d(\v)$, where $d(\v)=\sum_{i=1}^k d_i(\v)=
\sum\limits_{i,\alpha} v^{i,\alpha}$. \qed
\end{proposition}

\subsubsection{The fiberwise wrapped category}

As in Section \ref{s:Fukayacat} we first associate to $(Y,\mathbf{W})$ a
directed category whose objects are a given collection of admissible Lagrangian
submanifolds of $Y$, whose
images under each of the projections $W_1,\dots,W_k$ agree near infinity with
some fixed collection of radial straight lines in the complex plane, and
their images under an autonomous flow $L(t)=\phi^t\rho^t(L)$, where $\rho^t$
is the lifted admissible isotopy generated by applying the same autonomous
flow $\rho$ as in \S \ref{sec:defin-direct-categ} to each factor of $\C^k$, and $\phi^t$ is the flow generated by
the wrapping Hamiltonian $H$. This geometric setup gives rise to quasi-units
and continuation maps with the exact same properties as in Section
\ref{ss:quasiunits}, and we again define $\Wrap(Y,\mathbf{W})$ to be the
localization of the directed category with respect to the quasi-units.

\begin{remark}
Our construction of $\Wrap(Y,\mathbf{W})$ is rather {\em ad hoc}, but it
can be recast in the language of monomial admissibility on $Y$, using the
collection of toric monomials $\{z^\v,\ \v\in\mathcal{V}\}\cup\{W_1,\dots,W_k\}$.
Indeed, our conditions on objects of $\Wrap(Y,\mathbf{W})$ require
each of these monomials to have locally constant argument
(equal to a prescribed phase angle, or a pair of possible phase angles in the case
of $W_i$) over each end of the Lagrangian within a suitable subset of
$Y$; and the flow we consider has the effect of
increasing the phase angles within the interval $(-\pi,\pi)$ for each $W_i$,
and in an unbounded manner for $z^\v$ (i.e., we have removed the ``stops'' that
monomial admissibility would normally place at each $\arg(z^\v)=\pi$).

Even though the appropriate notions have yet to be developed outside of the
Liouville setting, one also expects that monomial admissibility 
can be recast in the language of stops in the sense of \cite{GPS2} 
(see \cite{HanlonHicks} for an instance of this), or even better, 
wrapped Floer theory on a (non-exact) sector with sectorial 
corners, in the spirit of \cite[Section 12]{GPS2}.
A rough candidate for the appropriate sector with corners is the subset of $Y$
consisting of those points where $\Re(W_i)\geq -R$ for all $i=1,\dots,k$,
for some $R\gg 0$; however, making the collection of hypersurfaces
$\{\Re(W_i)=-R\}$, $i=1,\dots,k$ sectorial requires a modification of the K\"ahler form on $Y$.
\end{remark}

\subsection{The main theorem}

As in Section \ref{ss:L0setup}, we fix a properly embedded U-shaped admissible arc $\gamma_0$ in the complex
plane which crosses the real axis at $-1$, and consider the admissible Lagrangian submanifold
$L_0\subset Y$ obtained by parallel transport over $\gamma_0\times
\dots\times \gamma_0\subset \C^k$ of the real positive locus $\ell_0\cong(\R_+)^n$
in $\mathbf{W}^{-1}(-1,\dots,-1)\cong (\C^*)^n$. 

\begin{theorem}\label{thm:ci_main}
For a suitable choice of bulk deformation class $\mathfrak{b}\in
H^2(Y,\Lambda_{\geq 0})$, the fiberwise wrapped Floer cohomology ring
$H\Wrap^*(L_0,L_0)$ is isomorphic to the quotient $\K[x_1^{\pm
1},\dots,x_n^{\pm 1}]/(f_1,\dots,f_k)$, i.e.\ the ring of functions of the
complete intersection $\mathbf{H}$.
Hence, the derived category of coherent sheaves of $\mathbf{H}$ admits a
fully faithful quasi-embedding into $\Wrap(Y,\mathbf{W})$.
\end{theorem}

As in Section \ref{s:calculation}, the main step to calculate the fiberwise wrapped Floer cohomology
$H\Wrap^*(L_0,L_0)$ is to determine the Floer complex of $L_0(t')$ and
$L_0(t)=\phi^t\rho^t(L_0)$ for $t'-t$ sufficiently positive. We start by
observing that $L_0(t)$ is obtained from $\ell_0(t)=\phi^t(\ell_0)$ by
parallel transport over $\gamma_t\times \dots\times \gamma_t$ (where
$\gamma_t=\rho^t(\gamma_0)$ as in Section \ref{s:calculation}). Thus, for
$t'-t>t_0$, the intersections of $L_0(t')$ and $L_0(t)$ 
lie in the fibers of $\mathbf{W}$ above the
$2^k$ points $(c_1,\dots,c_k)\in \C^k$ where each $c_i$ belongs to
$\gamma_t\cap \gamma_{t'}=\{-1,c_{t',t}\}.$  

For $I\subset \{1,\dots,k\}$ we
denote by $c_I\in \C^k$ the point with coordinates $c_i=-1$ if $i\not\in I$
and $c_i=c_{t',t}$ if $i\in I$. We then find that, for $t'-t>t_0$,
\begin{equation}\label{eq:ci_CFL0structure}
CF^*(L_0(t'),L_0(t))=\bigoplus_{I\subset \{1,\dots,k\}} C_I(t',t)[|I|],
\end{equation}
where $C_I(t',t)=CF^*(\ell_{I,-}(t'),\ell_{I,+}(t))$ is the Floer complex of the fiberwise Lagrangians 
obtained by intersecting $L_0(t')$ and $L_0(t)$ with $\mathbf{W}^{-1}(c_I)$, and the 
grading shift by $|I|$ comes from considering the grading contributions of the phase
angles of the arcs $\gamma_t$ and $\gamma_{t'}$ in the various factors of
$\C^k$. Moreover, by considering intersection numbers of holomorphic discs with
fibers of $\mathbf{W}$ (outside a small neighborhood of the coordinate planes), we find that
the Floer differential maps each summand $C_I(t',t)$ of
\eqref{eq:ci_CFL0structure} to the span of the $C_{I'}(t',t)$ for
$I'\subseteq I$.

Thus, the complex \eqref{eq:ci_CFL0structure} carries a natural filtration
(by $|I|$); we can proceed as in Section \ref{s:calculation} and calculate $HF^*(L_0(t'),L_0(t))$ as 
the cohomology of a ``vertical Floer complex'' built from the fiberwise
Floer cohomology groups $$H^*(C_I(t',t))=HF^*(\ell_{I,-}(t'),\ell_{I,+}(t)),$$
together with the maps from $H^*(C_I(t',t))$ to
$H^*(C_{I'}(t',t))$ for $I'\subsetneq I$ induced by the relevant portions
of the Floer differential on \eqref{eq:ci_CFL0structure} (i.e., discs
which are not contained within the fibers of $\mathbf{W}$).

Observing that for each $\v=(\vec{v},v^{1,0},\dots,v^{k,0})$
the restriction of the monomial $z^\v$ to
$\mathbf{W}^{-1}(c_1,\dots,c_k)\simeq (\C^*)^n$
is given by $\prod_{i=1}^k (-c_i)^{v^{i,0}}\,z_1^{v_1}\dots z_n^{v_n}$, the same calculation as in 
Example \ref{ex:slopes_L0} shows that the monomially
admissible Lagrangian sections $\ell_{I,-}(t')$ and $\ell_{I,+}(t)$ in
$\mathbf{W}^{-1}(c_I)$ have slopes 
\begin{eqnarray*}
\sigma_{I,-}(t')&=&\left(t'd(\v)-\bigl({\textstyle \sum_{i\in I}}
v^{i,0}\bigr)(\arg(c_{t',t})+\pi)\right)_{\v\in \mathcal{V}}\quad
\text{and}\\
\sigma_{I,+}(t)&=&\left(t\,d(\v)-\bigl({\textstyle \sum_{i\in I}}
v^{i,0}\bigr)(\arg(c_{t',t})-\pi)\right)_{\v\in \mathcal{V}.}
\end{eqnarray*}
Because $H$ is convex, for $t'-t$ sufficiently large (larger than some constant $t_1\geq t_0$)
\begin{equation}\label{eq:ci_slopes}
\sigma_I(t'-t)=\sigma_{I,-}(t')-\sigma_{I,+}(t)=\left((t'-t)\,d(\v)-2\pi {\textstyle\sum_{i\in I}}
v^{i,0}\right)_{\v\in \mathcal{V}}
\end{equation}
is the slope of a convex Hamiltonian for all $I\subset \{1,\dots,k\}$, so
that the results of Section \ref{ss:fiberwise-floer} apply to the Floer
cohomology groups $HF^*(\ell_{I,-}(t'),\ell_{I,+}(t))$.
In particular, these cohomology groups are concentrated in degree zero;
since the differential on the vertical Floer complex has degree 1, the only
non-zero connecting maps are those which take $H^0(C_I(t',t))$ to
$H^0(C_{I'}(t',t))$ for $I'\subset I$, $|I'|=|I|-1$. Writing
$I=I'\cup\{i\}$, we denote by $s_{I,i}$ the relevant portion of the
differential.

Next, we recall that for $t'-t>t_1$ and $I\subset\{1,\dots,k\}$,
$HF^0(\ell_{I,-}(t'),\ell_{I,+}(t))$  has a basis consisting of action-rescaled
Floer generators $\zeta_{I,p}^{t'\to t}$, whose elements are indexed by the
points of $P_I(t'-t)\cap (2\pi \Z)^n$, where
$P_I(t'-t)$ is the polytope associated to the slope $\sigma_I(t'-t)$ by
\eqref{eq:mon_adm_polytope}. For $I=\emptyset$ we also use the notation
$\vartheta_p^{t'\to t}=\zeta_{\emptyset,p}^{t'\to t}$. Hence:

\begin{proposition}
For $t'-t>t_1$, the Floer cohomology $HF^*(L_0(t'),L_0(t))$ is isomorphic to
the cohomology of the vertical Floer complex
\begin{equation}\label{eq:ci_CFL0vert}
CF^*_{vert}(L_0(t'),L_0(t))=
\bigoplus_{I\subset \{1,\dots,k\}}
HF^{*+|I|}(\ell_{I,-}(t'),\ell_{I,+}(t))\simeq 
\!\!\!\!\bigoplus_{\substack{I\subset \{1,\dots,k\}\\ p\in P_I(t'-t)\cap (2\pi\Z)^n}}
\!\!\!\!\K\cdot \zeta_{I,p}^{t'\to t},
\end{equation}
where the generators $\zeta_{I,p}^{t'\to t}$ $($in degree $-|I|)$ correspond
to intersections in $\mathbf{W}^{-1}(c_I)$, rescaled by action within
the fiber; together with a differential which is a sum of maps
$$s_{I,i}: HF^0(\ell_{I,-}(t'),\ell_{I,+}(t))\to
HF^0(\ell_{I',-}(t'),\ell_{I',+}(t))$$
for all $I=I'\sqcup \{i\}\subset \{1,\dots,k\}$.
\end{proposition}

Since the projections $W_1,\dots,W_k:Y\to \C$ are
holomorphic outside of a small neighborhood of the origin, 
the open mapping principle implies that
any $J$-holomorphic disc which contributes to the portion of the Floer
differential mapping $C_I(t',t)$ to $C_{I'}(t',t)$ ($I'\subset I$) is
contained within a single fiber of $W_i$ (over either $-1$ or $c_{t',t}$)
whenever $i\in I'$ or $i\not\in I$, while for $i\in I-I'$ it is a section
(except possibly near the origin) of $W_i:Y\to\C$ over the bounded region
delimited by $\gamma_{t'}$ and $\gamma_t$.

Thus, in the case at hand, the contributions to the differentials
$s_{I,i}$ correspond to holomorphic discs which are contained in a
level set of $W_{\overline{\imath}}=(W_j)_{j\neq i}:Y\to \C^{k-1}$.
By Proposition \ref{prop:ci_fiber}, this fiber $Y_i$, equipped with the
restriction of $W_i$, is isomorphic to the mirror of the hypersurface
$H_i=f_i^{-1}(0)$ considered in our main argument. Moreover, the
restrictions of $L_0(t')$ and $L_0(t)$ to $(Y_i,W_i)$ are exactly the same
sort of fibered admissible
Lagrangians we have considered in Section \ref{s:calculation} -- even though
for $I'\neq \emptyset$ the relevant fiberwise monomially admissible Lagrangian
sections differ from those
previously considered by the monodromy of $W_{i'}$ around the origin for $i'\in
I'$, as is manifest from the expression \eqref{eq:ci_slopes} for the slopes 
$\sigma_I(t'-t)$ and $\sigma_{I'}(t'-t)$.
Despite this minor difference, the core calculation of Section \ref{s:calculation} applies
to this setting, and implies:

\begin{proposition}
For all $I=I'\sqcup \{i\}\subset \{1,\dots,k\}$, the differential
$$s_{I,i}: HF^0(\ell_{I,-}(t'),\ell_{I,+}(t))\to HF^0(\ell_{I',-}(t'),\ell_{I',+}(t))$$
is, up to a nonzero multiplicative constant $C_{I,i}(t',t)\in \K^*$,
given by multiplication by a Laurent polynomial $g_i(x)=\sum_{\bar{p}\in P_{i,\Z}} c_{i,\bar{p}}
x^{\bar{p}}\in \K[x_1^{\pm 1},\dots,x_n^{\pm 1}]$ with the same Newton polytope
as $f_i$. Namely, for $p'\in P_I(t'-t)\cap (2\pi\Z)^n$,
$$s_{I,i}(\zeta_{I,p'}^{t'\to t})=C_{I,i}(t',t)\sum_{\bar{p}\in P_{i,\Z}}
c_{i,\bar{p}}\, \zeta_{I',p'+2\pi \bar{p}}^{t'\to t}.$$
\end{proposition}

\noindent Moreover, equipping $Y$ with a bulk deformation class
$\mathfrak{b}=\sum_i\sum_{\alpha\in P_{i,\Z}} \mathfrak{b}_{i,\alpha}
\delta_{Z_{i,\alpha}}$, where the $\delta_{Z_{i,\alpha}}$ are Poincar\'e dual to the irreducible toric
divisors of $Y$ and the coefficients $\mathfrak{b}_{i,\alpha}\in \Lambda_{\geq 0}$ are
determined as in Proposition \ref{prop:matchbulk}, ensures that $g_i=f_i$
for all $i$. 

Thus, denoting by $\K[(x_i^{\pm 1})]_P$ the subspace of $\K[(x_i^{\pm 1})]$
consisting of Laurent polynomials
whose Newton polytope is contained in $\frac{1}{2\pi} P$, we have:

\begin{proposition}
For a suitable choice of bulk deformation class $\mathfrak{b}\in
H^2(Y,\Lambda_{\geq 0})$, for $t'-t>t_1$ the Floer cohomology group
$HF^*(L_0(t'),L_0(t))$ is concentrated in degree zero and isomorphic (as a
vector space), via $\vartheta_{p}^{t'\to t}\mapsto x^{\bar{p}}$, to the quotient 
\begin{equation}\label{eq:ci_truncatedlaurent}
\K[(x_i^{\pm 1})]_{P_0(t'-t)}\,\big/\,
\bigl(f_1 \K[(x_i^{\pm 1})]_{P_{\{1\}}(t'-t)}+\dots+f_k \K[(x_i^{\pm
1})]_{P_{\{k\}}(t'-t)}\bigr).
\end{equation}
\end{proposition}

The Floer product 
\begin{equation}\label{eq:ci_floerproduct}
CF^*_{vert}(L_0(t''),L_0(t'))\otimes CF^*_{vert}(L_0(t'),L_0(t))\to
CF^*_{vert}(L_0(t''),L_0(t))\end{equation} 
can be determined as in
Section \ref{s:calculation}, by observing that any contributing $J$-holomorphic disc
projects under
$W_i:Y\to\C$ to either a single point or
a triangular region of the complex plane delimited by $\gamma_{t''}$,
$\gamma_{t'}$ and $\gamma_t$ (not enclosing the origin), and reducing to a calculation within the fiber
of $\mathbf{W}$. This yields an analogue of Proposition
\ref{prop:CFL0_product}:

\begin{proposition}
For $t''-t'>t_1$ and $t'-t>t_1$, the product
\eqref{eq:ci_floerproduct} is given by
\begin{equation}\zeta_{I,p}^{t'\to t}\cdot \zeta_{J,p'}^{t''\to
t'}=\begin{cases}
C_{I,J,t'',t',t}\,\zeta_{I\,\sqcup\, J,\,p+p'}^{t''\to t}\quad & \text{if }I\cap
J=\emptyset,\\
0 & \text{if }I\cap J\neq \emptyset,
\end{cases}
\end{equation}
for all $I,J\subset \{1,\dots,k\}$, $p\in P_I(t'-t)\cap (2\pi \Z)^n$, $p'\in
P_J(t''-t')\cap 2\pi\Z)^n$, where $C_{I,J,t'',t',t}$ is a non-zero constant.
In particular, for $I=J=\emptyset$ we have
\begin{equation}\label{eq:ci_product}
\vartheta_p^{t'\to t}\cdot \vartheta_{p'}^{t''\to t'}=\vartheta_{p+p'}^{t''\to t}.
\end{equation}
\end{proposition}

It follows from \eqref{eq:ci_product} that the cohomology-level product
structure corresponds to multiplication of Laurent polynomials on the
quotient spaces \eqref{eq:ci_truncatedlaurent}.

Finally, the quasi-unit $e^{t'\to t}\in HF^0(L_0(t'),L_0(t))$ is again given by
$e^{t'\to t}=\vartheta_0^{t'\to t}$, by the same argument as in Proposition
\ref{prop:CFL0_quasiunit}. Thus, computing $H\Wrap(L_0,L_0)$ as a colimit of the Floer
cohomology groups $HF^*(L_0(t'),L_0(t))$ as $t'-t\to \infty$ amounts to taking the colimit of
\eqref{eq:ci_truncatedlaurent} under the naive inclusion maps, and we arrive
at $$H\Wrap(L_0,L_0)\simeq \K[x_1^{\pm 1},\dots,x_n^{\pm
1}]/(f_1,\dots,f_k),$$ which completes the proof of Theorem
\ref{thm:ci_main}.

\begin{remark}
It is not a coincidence that the structure of the vertical Floer complex 
\eqref{eq:ci_CFL0vert} matches that of the Koszul complex which resolves
$i_*\O_{\mathbf{H}}$. This can be understood using the same perspective 
as in Section \ref{ss:functorial}, given the interpretation of the
Landau-Ginzburg models $(Y_I,W_I)$, $I\subset \{1,\dots,k\}$ provided by 
Proposition \ref{prop:ci_fiber}  and observing that for
$I=I'\sqcup\{i\}$ the categories $\Wrap(Y_I,W_I)$ and
$\Wrap(Y_{I'},W_{I'})$ are related to each other by cap and cup functors which
correspond under mirror symmetry  to the inclusion and restriction functors
between the derived categories of $H_I=\bigcap_{i\in I} f_i^{-1}(0)$ and
$H_{I'}$.
\end{remark}

\begin{remark}
The object $L_0$ is expected to generate $\Wrap(Y,\mathbf{W})$, which would
imply that the embedding of Theorem \ref{thm:ci_main} is an equivalence. The argument should
proceed by induction on $k$, using stop removal. Namely, for
$I=I'\sqcup\{i\}$ it should follow
from a suitable stop removal result (for the stop at $W_i\to -\infty$)
that $\Wrap(Y_I,W_{I'})$ is the quotient of
$\Wrap(Y_I,W_I)$ by the image of the cup functor from
$\Wrap(Y_{I'},W_{I'})$. On the other hand, the
category $\Wrap(Y_I,W_{I'})$ is expected to be trivial for $I'$ a strict
subset of $I$; at least, SYZ mirror symmetry suggests that $(Y_I,W_{I'})$
admits a B-side Landau-Ginzburg mirror whose superpotential has no critical 
points \cite{AAK}, which implies the triviality of its derived category of 
singularities.
Thus, one expects that $\Wrap(Y_I,W_I)$ is generated by the image 
under the cup functor of a generator of $\Wrap(Y_{I'},W_{I'})$;
the result then follows by induction on $k$.
\end{remark}

\end{document}